\documentclass[a4paper,10pt,reqno]{amsart}

\usepackage{fullpage}
\usepackage[utf8]{inputenc}
\usepackage{hyperref}
\usepackage{amsmath}
\usepackage{amsfonts}
\usepackage{amssymb}
\usepackage{amsthm}

\DeclareMathOperator{\WF}{WF}
\DeclareMathOperator{\real}{Re}
\DeclareMathOperator{\imag}{Im}
\DeclareMathOperator{\supp}{supp}
\DeclareMathOperator{\singsupp}{sing\, supp}
\DeclareMathOperator{\Char}{Char}

\DeclareMathOperator*{\proj}{proj}
\DeclareMathOperator{\dist}{dist}

\newcommand{\COT}{ U\times\mathbb{R}^n \!\setminus\! \{0\}}

\newcommand{\fa}{\;\,\forall\,}
\newcommand{\ex}{\;\,\exists\,}

\newcommand{\N}{\mathbb{N}}
\newcommand{\R}{\mathbb{R}}
\newcommand{\C}{\mathbb{C}}

\newcommand{\E}{\mathcal{E}}
\newcommand{\D}{\mathcal{D}}
\newcommand{\CC}{\mathcal{C}}
\newcommand{\An}{\mathcal{A}}
\newcommand{\G}{\mathcal{G}}
\newcommand{\fG}{\mathfrak{G}}
\newcommand{\F}{\mathcal{F}}
\newcommand{\bG}{\mathbf{G}}
\newcommand{\sP}{\mathsf{P}}
\newcommand{\sX}{\mathsf{X}}

\newcommand{\norm}[2][]{\left\lVert #2 \right\rVert_{#1}}
\newcommand{\vNorm}[3][P]{\left\lVert #3\right\rVert_{#2}^{#1}}
\newcommand{\alp}{{\lvert\alpha\rvert}}
\newcommand{\Betr}[2][]{\left\lvert #2 \right\rvert_{#1}}
\newcommand{\bet}{{\lvert\beta\rvert}}
\newcommand{\xit}{{\lvert\xi\rvert}}
\newcommand{\etat}{\lvert\eta\rvert}

\newcommand{\Beu}[2]{\mathcal{E}^{( #1 )} ( #2 )}
\newcommand{\Rou}[2]{\mathcal{E}^{\{ #1 \}} ( #2 )}
\newcommand{\DC}[2]{\mathcal{E}^{[ #1 ]} ( #2 )}

\newcommand{\vBeu}[3][P]{\mathcal{E}^{ (#2) } ( #3 ; #1 )}
\newcommand{\vRou}[3][P]{\mathcal{E}^{\{ #2 \}} ( #3; #1 )}
\newcommand{\vDC}[3][P]{\mathcal{E}^{[ #2 ]} ( #3 ; #1  )}

\newcommand{\temp}{\mathcal{S}^\prime(\R^n)}

\newcommand{\eps}{\varepsilon}

\newcommand{\Dp}{\mathcal{D}^\prime}
\newcommand{\Ep}{\mathcal{E}^\prime}

\newcommand{\bM}{\mathbf{M}}
\newcommand{\bN}{\mathbf{N}}
\newcommand{\fM}{\mathfrak{M}}
\newcommand{\fN}{\mathfrak{N}}
\newcommand{\bL}{\mathbf{L}}
\newcommand{\fL}{\mathfrak{L}}
\newcommand{\fQ}{\mathfrak{Q}}
\newcommand{\fR}{\mathfrak{R}}
\newcommand{\bB}{\mathbf{B}}
\newcommand{\fB}{\mathfrak{B}}
\newcommand{\fW}{\mathfrak{W}}
\newcommand{\bW}{\mathbf{W}}
\newcommand{\W}{\mathcal{W}}
\newcommand{\fT}{\mathfrak{T}}
\newcommand{\fS}{\mathfrak{S}}
\newcommand{\bS}{\mathbf{S}}
\newcommand{\fJ}{\mathfrak{J}}
\newcommand{\U}{\mathcal{U}}

\theoremstyle{plain}
\newtheorem{Thm}{Theorem}[section]
\newtheorem{Prop}[Thm]{Proposition}
\newtheorem{Lem}[Thm]{Lemma}
\newtheorem{Cor}[Thm]{Corollary}

\theoremstyle{definition}
\newtheorem{Def}[Thm]{Definition}

\theoremstyle{remark}
\newtheorem{Rem}[Thm]{Remark}
\newtheorem{Ex}[Thm]{Example}
\newtheorem{Not}[Thm]{Notation}

\numberwithin{equation}{section}

\title{The Theorem of Iterates for elliptic\\ and non-elliptic Operators}
\author{Stefan F\"urd\"os}
\address{Faculty of Mathematics, University of Vienna, Oskar-Morgenstern Platz 1, 1090 Vienna, Austria}

\email{stefan.fuerdoes@univie.ac.at}
\thanks{The first author was supported by FWF grant I 3472 and J 4439. 
	The second author was supported by FWF projects P 32905 and P 33417}

\author{Gerhard Schindl}
\address{Faculty of Mathematics, University of Vienna, Oskar-Morgenstern Platz 1, 1090 Vienna, Austria}
\email{gerhard.schindl@univie.ac.at}


\begin{document}






\begin{abstract}
We introduce a new approach for the study of the Problem of Iterates using the theory on general ultradifferentiable structures developed in the last years. Our framework generalizes many of the previous settings including the Gevrey case and enables us, for the first time, to prove non-analytic Theorems of Iterates for non-elliptic differential operators. In particular, by generalizing a Theorem of Baouendi and Metivier we obtain the Theorem of Iterates for  analytic hypoelliptic operators of principal type with respect to several non-analytic ultradifferentiable structures. 
\end{abstract}


\maketitle

\section{Introduction}\label{Introduction}
In recent years there has been renewed interest in the Problem of Iterates, i.e.\ 
the study of vectors of differential operators, we mention in particular 
 \cite{MR3380075}, \cite{MR3661157}, \cite{MR3652556}, \cite{MR3208537}, \cite{MR3556261}, \cite{MR3043156}, 
\cite{Derridj2017},  \cite{MR4036740}, \cite{derridj2019}, \cite{Derridj2019a},
\cite{MR4098643}, 
\cite{HoepfnerRampazo} and \cite{MR3857012}.
For the history of the problem we refer to
 the survey \cite{MR1037999}.
 
 The aim of this paper is to present a new approach to the Problem of Iterates
 using the ultradifferentiable structures introduced in \cite{MR3285413}
 and \cite{MR3462072}, which generalizes and unifies many of the previous cases.
 
In our context  an ultradifferentiable structure $\U$ is
a subalgebra of smooth functions which is defined by estimates
on the derivatives of its elements.
Well-known ultradifferentiable structures include
 the Denjoy-Carleman classes which are given
by weight sequences
and the Braun-Meise-Taylor classes whose defining data are weight functions.
The latter  were originally introduced by  \cite{Beurling61} and \cite{Bjorck66},
but the modern formulation of these classes was given in \cite{MR1052587}.
The classes discussed in \cite{MR3285413}, which are determined by weight matrices,
i.e.\ families of weight sequences, encompass both Denjoy-Carleman classes and
Braun-Meise-Taylor classes. 
Other examples of ultradifferentiable spaces are
the Gelfand-Shilov classes, cf.\ \cite{MR0230128}
and  the recently introduced $L^p$-ultradifferentiable classes, see  \cite{MR3936096}.

Then ultradifferentiable vectors of some operator $P$
associated to the structure $\U$
are those functions (or distributions) which satisfy the 
defining estimates of $\U$ for the iterates $P^k$ of $P$.
Thus the \emph{Problem of Iterates} in its general form can rather casually
be formulated as the following question:
\begin{quote}
	Given an operator $P$ suppose that a function (or distribution) $u$ satisfies
	the defining estimates of an ultradifferentiable structure $\U$ for the iterates $P^k$
	of $P$. Can we conclude that $u$ satisfies these estimates for all derivatives?
\end{quote}
Or more concisely, are the ultradifferentiable vectors of $P$ with respect to
$\U$ already ultradifferentiable functions of 
class $\U$?
If the answer to this question is "yes'' then we say that the \emph{Theorem of Iterates} holds
for the operator $P$ and the structure $\U$.
 
Our main goal is to develop a unified approach to 
the problem of iterates using the recent development of the theory of general 
ultradifferentiable classes given in \cite{MR3285413}, \cite{MR3462072} and in particular
the microlocal theory in
\cite{FURDOS2020123451}.
This approach allows us not only to unify and generalize previously known results but also to treat cases which have not been available in the literature up to now.
In particular, in the case of principal type operators we are able to use
the technical estimate in \cite{MR654409} to infer the Theorem of Iterates
for a wide variety of ultradifferentiable classes, which include
quasianalytic and non-quasianalytic classes.
We note that, to our knowledge, this is the first time the Theorem of Iterates
is proven for a non-elliptic operator and a non-analytic ultradifferentiable structure.

In the case of Braun-Meise-Taylor classes our main Theorem takes 
a relatively concise form. However, in order to formulate it correctly, 
we need to recapitulate some notations:
We say that a differential operator $P$ defined on some open set $U\subseteq\R^n$
is of principal type\footnote{We follow here the classic definition, see e.g.\
	\cite{MR0296509} and the references therein. It sometimes does not agree 
	with the definition of principal type operators given in modern treatises, 
	for example in \cite[Chapter 26]{MR2512677}.}
or that $P$ is an operator with simple real characteristics if the principal
symbol $p_d$ of $P$ satisfies
\begin{equation*}
	\Betr{p_d(x,\xi)}+\sum_{j=1}^{n}\Betr{\partial_{\xi_j}p_d(x,\xi)}\neq 0
\end{equation*}
for all $(x,\xi)\in\COT$.

A weight function in the sense of \cite{MR1052587} 
is a continuous and increasing function $\omega: [0,\infty)\rightarrow [0,\infty)$
with $\omega(0)=0$ which satisfies
\begin{align}\tag{$\alpha$}\label{om2}
	\omega(2t)&=O(\omega(t))\quad\text{as } t\rightarrow \infty,\\
	\tag{$\beta$}\label{om3}
	\log t&=o(\omega(t))\quad\;\text{as }t\rightarrow\infty,\\
	\label{om4}\tag{$\gamma$}
	\varphi_\omega&=\omega\circ\exp\quad\,\text{is convex.}
\end{align}
We set 
\begin{equation*}
	\norm[V,\omega,h]{f}=\sup_{\substack{x\in V\\ \alpha\in\N_0^n}}
	\Betr{D^\alpha f(x)} e^{-\tfrac{1}{h}\varphi_\omega^\ast(h\alp)},
\end{equation*}
where $V\Subset U$ is a relatively compact subset of $U$, $f\in\E(U)$ is a smooth functions,
 $h>0$ and $\varphi^\ast_\omega(t):=\sup_{s\geq 0}(st-\varphi_\omega(s))$ 
is the conjugate function of $\varphi_\omega$.
The Roumieu class (of ultradifferentiable functions)
associated with $\omega$ is given by
\begin{align*}
	\Rou{\omega}{U}&=\left\{f\in\E(U):\;\fa V\Subset U\ex h>0\quad
	\norm[V,\omega,h]{f}<\infty\right\}\\
	\intertext{and the Beurling class associated to $\omega$ is}
	\Beu{\omega}{U}&=\left\{f\in\E(U):\;\fa V\Subset U\fa h>0\quad
	\norm[V,\omega,h]{f}<\infty\right\}.
\end{align*}
Similarly, for a partial differential operator $P$ of order $d$
with analytic coefficients we set 
\begin{align*}
	\vRou{\omega}{U}&=\Bigl\{u\in\Dp(U):\;\fa V\Subset U \ex h>0\quad
	\lVert u\rVert^P_{V,\omega,h}<\infty\Bigr\}\\
	\intertext{and}
	\vBeu{\omega}{U}&=\Bigl\{u\in\Dp(U):\;\fa V\Subset U \fa h>0\quad
	\lVert u\rVert^P_{V,\omega,h}<\infty\Bigr\},
\end{align*}
where
\begin{equation*}
	\vNorm{V,\omega,h}{u}=\sup_{k\in\N_0}\norm[L^2(V)]{P^ku}
	e^{-\tfrac{1}{h}\varphi_\omega^\ast(hdk)}.
\end{equation*}
Our main result in the case of weight functions is:
\begin{Thm}\label{omMainThm}
	Let $U\subseteq\R^n$ be an open set 
	and $P$ a hypoelliptic operator of principal type with analytic coefficients in $U$.
	Furthermore assume that $\omega$ is a weight function satisfying
	\begin{equation}\tag{$\Xi$}\label{om7}
		\exists H> 0:\quad	\omega(t^2)=O(\omega(Ht)) \qquad t\rightarrow \infty.
	\end{equation}
	
	Then 
	\begin{align*}
		\vRou{\omega}{U}&=\Rou{\omega}{U},\\
		\vBeu{\omega}{U}&=\Beu{\omega}{U}.
	\end{align*}
\end{Thm}
We may note  that the condition  \eqref{om7} 
 has appeared in various applications of 
Braun-Meise-Taylor classes, e.g.\ in the study of global pseudodifferential operators
in \cite{MR3999031}.

\subsection{Preliminaries}
We denote by $\N=\{1,2,\dots\}$ the set of positive integers and by $\N_0=\N\cup\{0\}$
the set of non-negative integers. Furthermore $U\subseteq\R^n$ is always an open set.
In this paper we focus on linear differential operators with analytic coefficients, i.e.
\begin{equation*}
P(x,D)=\sum_{\alp\leq d}a_\alpha(x)D^\alpha
\end{equation*}
 with  $a_\alpha\in \An(U)$.
 We use here the convention 
$D_j=-i\partial_{x_j}$. Then the symbol of $P$ is
\begin{align*}
p(x,\xi)&=\sum_{\alp \leq d}a_{\alpha}(x)\xi^\alpha\\
\intertext{and}
p_d(x,\xi)&=\sum_{\alp=d}a_\alpha(x)\xi^\alpha
\end{align*}
is the principal symbol of $P$.
The characteristic set of $P$ is given by
\begin{equation*}
\Char P=\left\{(x,\xi)\in\COT:\; p_d(x,\xi)=0\right\}.
\end{equation*}
Hence $\Char(P)=\emptyset$ if and only if $P$ is elliptic.

We say that a distribution $u\in\Dp(U)$ is an analytic vector of the operator $P$ if
for any $V\Subset U$ there are constants $C,h>0$ such that
	\begin{equation*}
	\norm[L^2(V)]{P^k u}\leq Ch^k k!
	\end{equation*}
for all $k\in\N_0$.
We write $\An(U;P)$ for the space of analytic vectors of $P$.
In \cite{komatsu1962} and \cite{KotakeNarasimhan} it was shown separately 
that if $P$ is elliptic then $\An(U;P)=\An(U)$. 
A similar result was proven in \cite{MR107176} for elliptic systems of vector fields.

We can consider this problem in a more general setting, if we replace the factor 
$k!$ in the estimate above by e.g.\ $(k!)^s$.
Recall that a smooth function $f\in\E(U)$ is an $s$-Gevrey function, $s\geq 1$, 
if for 
all $V\Subset U$ there are constants $C,h>0$ such that
\begin{equation*}
\sup_{x\in V} \Betr{D^\alpha f(x)}\leq Ch^\alp(\alp!)^s,\qquad \fa \alpha\in\N^n_0. 
\end{equation*}
The space of $s$-Gevrey functions on $U$ is denoted by $\G^s(U)$.
Analogously, an $s$-Gevrey vector $u$ of $P$ is a distribution $u\in\Dp(U)$
which satisfies the estimate
\begin{equation*}
\norm[L^2(V)]{P^ku}\leq Ch^k(k!)^s.
\end{equation*}
We denote the space of $s$-Gevrey vectors of $P$ by $\G^s(U;P)$ 
and if $P$ is elliptic then $\G^s(U;P)=\G^s(U)$ for all $s\geq 1$ according to \cite{MR548225}.

In fact, M\'{e}tivier \cite[Theorem 1.2]{doi:10.1080/03605307808820078} 
showed that the ellipticity of an analytic differential operator $P$ can be characterized by the regularity of its
non-analytic Gevrey vectors: If $s>1$ then
$P$ is elliptic if and only if $\G^s(U;P)=\G^s(U)$.

Clearly the Problem of Iterates is closely related to other regularity questions of 
the operator $P$, see e.g.\ \cite{MR1037999}. This connection has been extensively
 studied for operators with constant coefficients, see for example 
\cite{MR3595351}, \cite{MR3661157},  \cite{MR3208537}, \cite{MR2721087}
  and  \cite{MR0318660}.
However, in the wake of M\'{e}tivier's Theorem
 the study of vectors of a differential operator with variable coefficients has mainly split into two directions:
\begin{itemize}
\item If $P$ is elliptic then the Theorem of Iterates has been
 proven for a large class of
 ultradifferentiable structures: 
 e.g.\ for Denjoy-Carleman classes in \cite{MR557524},
 for Braun-Meise-Taylor classes in \cite{MR3652556},
 for Gelfand-Shilov classes in \cite{MR2801277}
 and for $L^q$-ultradifferentiable functions in \cite{HoepfnerRampazo}.
 In particular, in \cite{MR557524}  a microlocal elliptic Theorem of Iterates
 for Denjoy-Carleman classes is proven: If $u$ is an ultradifferentiable vector
 of $P$ with respect to a weight sequence $\bM$ then $\WF_{\{\bM\}}u\subseteq\Char P$,
 where $\WF_{\{\bM\}}u$ denotes the ultradifferentiable wavefront set of $u$
 with respect to $\bM$
introduced by \cite{MR0294849}.
 \item If $P$ is non-elliptic then it might still be possible to show 
 that analytic vectors are analytic, 
 cf.\  the surveys \cite{MR1037999} and \cite{Derridj2017}.
 For non-analytic Gevrey vectors  one tries to determine the loss of regularity 
 in terms of the Gevrey scale $(\G^s)_s$.
More precisely,
  we want to find for each $s>1$ some $s^\prime>s$ such that
 every $s$-Gevrey vector is an $s^\prime$-Gevrey function. 
  This approach was used for example in \cite{MR654409}, \cite{MR632764}, \cite{MR1037999} and \cite{Derridj2017}.
\end{itemize}

The  simplest class of non-elliptic operators with variable coefficients are the operators of principal type.
The main result on  Gevrey vectors of principal type operators is
the following result of Baouendi and M\'{e}tivier \cite[Theorem 1.3]{MR654409}:
	If $P$ is a hypoelliptic operator of principal type with analytic coefficients in $U\subseteq\R^n$
	then for each $V\Subset U$
	 there is some $\delta>0$ such that for all $s\geq 1$ we have
that every $s$-Gevrey vector $u$ of $P$ in $U$ is an $s^\prime$-Gevrey function in $V$
	where 
	$s^\prime=(ds-\delta)/(d-\delta).$

In this paper we are going to generalize the result of Baouendi and M\'{e}tivier
using the new theory on ultradifferentiable structures 
defined by weight matrices introduced in \cite{MR3285413}
which in turn will yield the Theorem of Iterates for hypoelliptic operators of principal type
with respect to ultradifferentiable structures given by certain weight matrices.
For example, the following observation was the starting point of this paper:
The prototypical example of a nontrivial weight matrix is the Gevrey matrix
\begin{equation*}
	\fG=\left\{\bG^s=\bigl((k!)^s\bigr)_k:\;s>1\right\}.
\end{equation*}
For a discussion of the properties of $\fG$ we refer to \cite[Section 5]{MR3285413}.
The Roumieu classes of ultradifferentiable functions and vectors associated to
$\fG$ are
\begin{align*}
	\Rou{\fG}{U}&=\left\{f\in\E(U):\;\fa V\Subset U\ex s>1:\;f\in\G^s(V)\right\} \\
	\intertext{and}
	\vRou{\fG}{U}&=\left\{u\in\Dp(U):\;\fa V\Subset U\ex s>1:\;u\in\G^s(V;P)\right\},\\
	\intertext{respectively, whereas the Beurling classes are given by}
	\Beu{\fG}{U}&=\left\{f\in\E(U):\;\fa V\Subset U\fa s>1:\;f\in\G^s(V)\right\}\\
	&=\bigcap_{s>1}\G^s(U)\\
	\intertext{and}
	\vBeu{\fG}{U}&=\left\{u\in\Dp(U):\;\fa V\Subset U\fa s>1:\;u\in\G^s(V;P)\right\}\\
	&=\bigcap_{s>1}\G^s(U;P),
\end{align*}
respectively.
\begin{Prop}\label{GevreyCor}
	Let $P$ be a hypoelliptic partial differential operator of principal type with 
	analytic coefficients on an open set $U\subseteq\R^n$.
	Then 
	\begin{equation*}
		\vDC{\fG}{U}=\DC{\fG}{U}.
	\end{equation*}
\end{Prop}
\begin{Not}
	Throughout the article, we are going to use the convention that $[\ast]=\{\ast\},(\ast)$ where $\ast=\bM,\fM,\omega$.
\end{Not}
\begin{proof}[Proof of Proposition \ref{GevreyCor}]
	Since $\G^s(U)\subseteq\G^s(U;P)$ for all $s\geq 1$, cf.\ 
	\cite{MR1037999},	
	it is enough to show $\vDC{\fG}{U}\subseteq\DC{\fG}{U}$.
	If $u\in\vBeu{\fG}{U}$ then $u\in\G^{1+\sigma}(U;P)$ for all $\sigma>0$.
	For $V\Subset U$ we have that $u\vert_V\in\G^{1+\sigma^\prime}(V)$ by 
	\cite[Theorem 1.3]{MR654409} where 
	$\sigma^\prime=(d(\sigma+1)-\delta)/(d-\delta)-1=d\sigma/(d-\delta)$
	for some $\delta\geq 0$ depending on the operator and $V$. 
	Thence $u\vert_V\in\bigcap_{\sigma^\prime>0}\G^{1+\sigma^\prime}(V)=\Beu{\fG}{V}$ 
	for all $V\Subset U$ and therefore $u\in\Beu{\fG}{U}$.
	
	If $u\in\vRou{\fG}{U}$ then for every $V\Subset U$ there is some $s>1$ such
	that $u\vert_V\in\G^s(V;P)$. \cite[Theorem 1.3]{MR654409} implies that
	for every $W\Subset V$ there is some $s^\prime>1$ such that 
	$u\vert_W\in\G^{s^\prime}(W)$.
	It follows that $u\in\Rou{\fG}{U}$.
\end{proof}

The concept of weight matrix was introduced in \cite{MR3285413} in
order to deal simultaneously with Denjoy-Carleman classes and Braun-Meise-Taylor classes.
It is well known that the Gevrey classes can be realized as Denjoy-Carleman classes
or as Braun-Meise-Taylor classes, but
 in general weight sequences and weight  
functions describe different classes, cf.\ \cite{BonetMeiseMelikhov07}.

The theory of weight matrices allows us to deal with countable intersections
and also countable unions (in the sense of germs)
of Denjoy-Carleman classes, which will be of some importance in our considerations.
For example, $\DC{\fG}{U}$ can neither be described as Denjoy-Carleman classes
nor as Braun-Meise-Taylor classes, cf.\ \cite[Theorem 5.22]{MR3285413}.

Weight matrices have been used to generalize and unify results regarding ultradifferentiable classes in various areas, see e.g.\ \cite{sectorextensions}, \cite{MR3711790},  \cite{MR3722569}
or \cite{MR3865684}.
In particular, in \cite{FURDOS2020123451} we defined the ultradifferentiable wavefront set associated with classes given by weight matrices and generalized 
and unified results on the wavefront set for
Denjoy-Carleman classes
 proved in \cite{MR4149078} and \cite{MR0294849} and for Braun-Meise-Taylor classes in
\cite{MR2595651}.

As we have seen,
 we can associate to each weight matrix (or weight sequence or weight function)
two different ultradifferentiable classes, the Roumieu class and the Beurling class, respectively.
Since the Gevrey classes are Roumieu classes, such spaces have been mainly
studied as for example in \cite{MR557524}. 
But when both Roumieu and Beurling classes have been
considered, there seems to be no much difference regarding the results obtained,
see e.g.\ \cite{MR3380075}, \cite{MR3208537}  or also Theorem \ref{omMainThm} above.
Nevertheless, we will notice that in the case of weight matrices there is
occasionally a difference between the Beurling and the Roumieu case when we regard
vectors of a non-elliptic operator.
\subsection{Outline of the paper}
We want to present in this paper  a throughout introduction to the theory of
ultradifferentiable vectors associated to weight matrices. 
In Section \ref{Uclasses} we recall for the convenience of the reader
 the definitions and facts from the theory of weight matrices we are going to need, including some statements concerning the ultradifferentiable wavefront set,
 which have not been explicitly stated in \cite{FURDOS2020123451}.
 Then we show in Section \ref{UVectors} that the microlocal theory
in \cite{MR557524} can be extended to classes given by weight matrices.
In particular we prove the elliptic Theorem of Iterates for these classes.
We should note that the restriction to analytic operators allows us 
to work with
rather weak conditions on the weight matrix.
In fact, we require only that the associated classes are invariant under
the action of analytic differential operators and under the composition with analytic diffeomorphisms.

Next we want to generalize Proposition \ref{GevreyCor} to other weight matrices.
In order to do so
 we introduce in Section \ref{sec:Scales} the notion of ultradifferentiable scales,
which can be considered as a special kind of weight matrices.
This allows us to extend \cite[Theorem 1.3]{MR654409}
 (i.e.\ Theorem \ref{ScaleCorollary1}) and Proposition \ref{GevreyCor}
(cf.\ Theorem \ref{Theorem2}) to ultradifferentiable scales and their associated
weight matrices, respectively.
We will see that many families of weight sequences, which have been studied
previously in the literature, constitute ultradifferentiable scales, including
the scale $(\bN^q)_{q> 1}$ of $q$-Gevrey sequences which are given by
$N_k^q=q^{k^2}$ and the scale $(\bB^\lambda)_{\lambda>0}$ given by
$B_k^\lambda=k!(\log (k+e))^{\lambda k}$. 

In Section \ref{WeightFctScales}
the proof of Theorem \ref{omMainThm} and especially condition \eqref{om7}
are discussed. Furthermore, we discuss in the second part of this section
how the exact definition of ultradifferentiable scales is tied to
the rather precise estimates obtained in \cite{MR654409}
 and how to modify it for the study of vectors of other operators.

In the final section we have included some selected topics.
  In Subsection \ref{Appendix1} we observe 
that the theory of ultradifferentiable scales developed in Section 
\ref{sec:Scales} can also be applied
to generalize the results of \cite{ASENS_1980_4_13_4_397_0}.
Subsection \ref{Appendix2}  explores for which weight sequences $\bM$ 
the associated weight 
function $\omega_{\bM}$ satisfies \eqref{om7}.
In the next subsection we take a first look at vectors determined by
a family of weight functions.
We close the paper with the proof of the following variant of 
\cite[Theorem 1.2]{doi:10.1080/03605307808820078}, where $\Rou{\bN^q}{U}$ and $\vRou{\bN^q}{U}$
are the Roumieu class and the space of Roumieu vectors of $P$ associated with the
weight sequence $\bN^q$, respectively.
\begin{Thm}\label{MetivierThm2}
	Let $P$ be a differential operator with analytic coefficients in $U$
	 and $q>1$.
	Then the following statements are equivalent:
	\begin{enumerate}
		\item $P$ is elliptic.
		\item $\vRou{\bN^q}{U}=\Rou{\bN^q}{U}$.
	\end{enumerate}
\end{Thm}

\section{Ultradifferentiable classes}\label{Uclasses}
\subsection{Weight matrices}
A sequence $\bM=(M_k)_{k\in\N_0}$ of positive numbers is a weight sequence if it
is normalized, i.e.\ $M_0=1$,  
$\lim_{k\rightarrow\infty}(M_k)^{1/k}=\infty$
and logarithmically convex, i.e.\
\begin{equation}\label{logconvexity}
\bigl(M_{k}\bigr)^2\leq M_{k-1}M_{k+1}
\end{equation}
for all $k\in\N$.
Note that for any such weight sequence $\bM$ we have
\begin{equation}\label{logconvexity2}
M_{j}M_k\leq M_{j+k},\qquad j,k\in\N_0.
\end{equation}
We are also going to use frequently the sequence $m_k=M_k/k!$.

For a weight sequence $\bM$, a bounded open set $V\subseteq\R^n$ and a constant $h>0$
we set
\begin{equation*}
\norm[V,\bM,h]{f}=\sup_{\substack{x\in V\\ \alpha\in\N^n_0}}
\frac{\lvert D^\alpha f(x)\rvert}{h^\alp M_\alp},\qquad f\in\E(V).
\end{equation*}
We define the Roumieu class (over an open set $U\subseteq\R^n$) associated to $\bM$ as
\begin{align*}
\Rou{\bM}{U}&=\left\{f\in\E(U):\;\fa V\Subset U\ex h>0:\;\;
\norm[V,\bM,h]{f}<\infty\right\}\\
\intertext{whereas the Beurling class associated with $\bM$ is}
\Beu{\bM}{U}&=\left\{f\in\E(U):\;\fa V\Subset U\fa h>0:\;\;
\norm[V,\bM,h]{f}<\infty\right\}.
\end{align*}
Clearly, the vector space  $\DC{\bM}{U}$ is an algebra with respect to the pointwise multiplication, due to \eqref{logconvexity2}.

Recall that a subspace $E\subseteq\E(U)$ is said to be quasianalytic if
$E$ contains no non-trivial functions of compact support, i.e.\ $E\cap\D(U)=\{0\}$. 
In the case of Denjoy-Carleman classes $\DC{\bM}{U}$ 
quasianalyticity is characterized by the Denjoy-Carleman theorem 
(see e.g.\ \cite{MR1996773}):
\begin{Thm}
Let $\bM$ be a weight sequence. 
The space $\DC{\bM}{U}$ is quasianalytic if and only if
\begin{equation}\label{quasicondition}
\sum_{k=0}^\infty \frac{M_k}{M_{k+1}}=\infty.
\end{equation}
\end{Thm}
We say that the sequence $\bM$ is quasianalytic if it satisfies \eqref{quasicondition}. 
Otherwise $\bM$ is non-quasianalytic.

If $\bM$ and $\bN$ are two sequences we write
\begin{align*}
\bM\leq\bN\quad &:\Longleftrightarrow \quad \fa k\in\N_0: M_k\leq N_k,
\\
\bM\preceq\bN\quad &:\Longleftrightarrow \quad \left(\tfrac{M_k}{N_k}\right)^{1/k}
\text{ is bounded for } k\rightarrow \infty,
\\
\bM\lhd\bN\quad &:\Longleftrightarrow \quad \left(\tfrac{M_k}{N_k}\right)^{1/k}\rightarrow
0 \text{ if } k\rightarrow \infty,
\end{align*}
and $\bM\approx \bN$ when $\bM\preceq\bN$ and $\bN\preceq\bM$.
We recall that $\DC{\bM}{U}\subseteq\DC{\bN}{U}$ if  $\bM\preceq\bN$ and
 $\Rou{\bM}{U}\subseteq\Beu{\bN}{U}$  when $\bM\lhd\bN$.

 For later use we note the following result in the spirit of 
\cite[Lemma 6]{MR550685}. Throughout the paper,
if not indicated otherwise,
 we are going to consider the constants appearing in the proofs to
be generic, that is they may change their value from line to line.
\begin{Lem}\label{AuxillaryLem1}
Let $\bM$ be a weight sequence and $\bL^\prime$ be a sequence with $L^\prime_0=1$ and
$\bG^1\preceq \bL^\prime\lhd\bM$. Then
there is a weight sequence $\bN$ such that
\begin{equation*}
\bL^\prime\leq\bN\lhd\bM.
\end{equation*}
\end{Lem}
\begin{proof}
For each $h>0$ we denote by $C_h$ the smallest constant $C>0$ such that
\begin{equation*}
L_k^\prime\leq Ch^kM_k
\end{equation*}
holds for all $k\in\N_0$. We define a new sequence $\bL$ by setting
\begin{equation*}
L_k:=\inf_{h>0}C_hh^kM_k.
\end{equation*}
Clearly $\bL^\prime\leq\bL$. 
If we put $\mu_k=M_k/M_{k-1}$ and $\lambda_k=L_k/L_{k-1}$ for $k\in\N$
then we recall from \cite[Lemma 6]{MR550685} that 
$\mu_k/\lambda_k$ is increasing and unbounded.

Set
\begin{equation*}
\nu_k=\max\Bigl\{\sqrt{\mu_k},\max_{1\leq j\leq k}\lambda_j\Bigr\}
\end{equation*}
for $k\in\N$ and define the sequence $\bN$ by $N_0=1$ and 
\begin{equation*}
N_k=\prod_{j=1}^k\nu_j
\end{equation*}
if $k\in\N$. The sequence $\nu_k$ is increasing since $\mu_k$ is increasing,
thence $\bN$ satisfies \eqref{logconvexity}.
It is easy to see that $\bL\leq\bN$ and therefore $k\leq C\sqrt[k]{N_k}$
for some constant $C>0$ independent of $k\in\N$.
It follows that $\bN$ is a weight sequence since $(N_k)^{1/k}\geq (M_k)^{1/2k}$.

It remains to prove $\bN\lhd\bM$. For this it is enough to show
\begin{equation*}
\lim_{k\rightarrow\infty}\frac{\nu_k}{\mu_k}=0.
\end{equation*}
We have
\begin{equation*}
\frac{\nu_k}{\mu_k}=\max\biggl\{\bigl(\mu_k\bigr)^{-\tfrac{1}{2}},
\bigl(\mu_k\bigr)^{-1}\max_{1\leq j\leq k}\lambda_j\biggr\}
\end{equation*}
for all $k\in\N$. 
For each $\varepsilon>0$ there has to exist $k_\eps\in\N$ such that
$\lambda_k/\mu_k\leq \eps$ for all $k\geq k_\eps$. Hence
\begin{equation*}
\frac{\nu_k}{\mu_k}\leq\max\biggl\{\bigl(\mu_k\bigr)^{-\tfrac{1}{2}},\eps,
\bigl(\mu_k\bigr)^{-1}\max_{1\leq j\leq k_\eps}\lambda_j\biggr\}
\end{equation*}
and thus
$\frac{\nu_k}{\mu_k}\leq \varepsilon$
for large enough $k$.
\end{proof}

Following \cite{FURDOS2020123451} we say that a weight sequence $\bM$ is semiregular if 
\begin{gather}
\lim_{k\rightarrow\infty}\sqrt[k]{m_k}=\infty\label{AnalyticInclusion}\\
\ex C>0\fa k\in\N_0:\quad M_{k+1}\leq C^{k+1}M_k.\label{DC-Derivclosed}
\end{gather}
Observe that \eqref{AnalyticInclusion} implies that for all $\gamma>0$ there
is some constant $C>0$ such that
\begin{equation}\label{AnalyticInclAlt}
k^k\leq C \gamma^k M_k.
\end{equation}
\begin{Rem}\label{SeqRem}
If $\bM$ is a weight sequence then $\An(U)\subsetneq\DC{\bM}{U}$ if and only if
$\bM$ satisfies \eqref{AnalyticInclusion}. 
On the other hand,
if $\bM$ satisfies \eqref{DC-Derivclosed} then $\DC{\bM}{U}$ is closed under derivation, i.e.\ if $f\in\DC{\bM}{U}$
then also $\partial_jf\in\DC{\bM}{U}$ for all $1\leq j\leq n$.
We may also note that \eqref{DC-Derivclosed} is equivalent to
\begin{equation*}
	\ex C>0 \;\fa k\in\N_0:\quad m_{k+1}\leq C^{k+1}m_k.
\end{equation*}

If $\bM$ is semiregular then $\E^{[\bM]}$ is closed under composition with 
analytic mappings, that is, if $\Phi: U\rightarrow V$ is an analytic mapping
between two open sets $U\subseteq\R^{n_1}$ and $V\subseteq\R^{n_2}$ then
for all $f\in\DC{\bM}{V}$ we have $f\circ\Phi\in\DC{\bM}{U}$, 
cf.\ \cite{MR1996773} and \cite{FURDOS2020123451}, respectively.
\end{Rem}
\begin{Ex}\label{WSequencesExamples}
	We present some examples of weight sequences, which will appear throughout the paper.
\begin{enumerate}
	\item The Gevrey class of order $s> 1$ is defined by the
	semiregular non-quasianalytic weight sequence $\bG^s=(G_k^s)_k=(k!^s)_k$.
	Note that $\bG^s\lhd\bG^t$ if and only if $s<t$.
	\item Let $q,r>1$ be two parameters.
	The weight sequence $\bL^{q,r}=(L^{q,r}_k)_k$ defined by $L^{q,r}_k=k!q^{k^r}$
	is semiregular if and only if $r\leq 2$.
	Observe that for all $q,r,s>1$ we have $\bG^s\lhd\bL^{q,r}$.
	Furthermore $\bL^{q_0,r_0}\lhd \bL^{q_1,r_1}$ if $r_0<r_1$ and $q_0,\,q_1>1$
	arbitrary or if $r_0=r_1$ and $1<q_0<q_1$.
	\item Let $\sigma>0$. The semiregular weight sequence
	$\bB^\sigma=(B^\sigma_k)_k$ given by 
	$B^\sigma_k=k!(\log(k+e))^{\sigma k}$ is quasianalytic
	if and only if $\sigma\leq 1$, cf.\ \cite{MR2384272}.
	\item We can generalize the previous example, cf.\ \cite{MR2822315}:
	 For $j\in\N$ we define
	the function $\log^{(j)}$ recursively by 
	\begin{align*}
	\log^{(1)}(t)&=\log t, & \log^{(j+1)}(t)&=\log\bigl(\log^{(j)}(t)\bigr), & t &\text{ large enough}.
	\end{align*}
Furthermore set $e^{(1)}=e$ and $e^{(j+1)}=e^{(e^{(j)})}$.

We consider the 2-parameter family of semiregular 
weight sequences $\bB^{j,\sigma}$, $j\in\N$, $\sigma>0$, given by 
\begin{equation*}
B^{j,\sigma}_k=k!\left(\log^{(j)}\bigl(k+e^{(j)}\bigr)\right)^{\sigma k}.
\end{equation*}
We have that $\bB^{1,\sigma}=\bB^\sigma$ and $\bB^{j,\sigma}$ is quasianalytic
when $j\geq 2$.
If $j_1<j_2$ then $\bB^{j_2,\sigma}\lhd\bB^{j_1,\tau}$ for any $\sigma,\tau>0$
and $\bB^{j,\sigma}\lhd\bB^{j,\tau}$ for $\sigma<\tau$.
\end{enumerate}	
\end{Ex}

A weight matrix $\fM$ is a family of weight sequences such that for each pair
$\bM,\bN\in\fM$ we have either $\bM\leq\bN$ or $\bN\leq\bM$.
The Roumieu class associated with the weight matrix $\fM$ is
\begin{align*}
\Rou{\fM}{U}&=\left\{f\in\E(U):\;\fa V\Subset U\ex\bM\in\fM\ex h>0:\;\;
\norm[V,\bM,h]{f}<\infty\right\}\\
\intertext{and the corresponding Beurling class is defined by}
\Beu{\fM}{U}&=\left\{f\in\E(U):\;\fa V\Subset U\fa\bM\in\fM\fa h>0:\;\;
\norm[V,\bM,h]{f}<\infty\right\}.
\end{align*}
Observe that $\Beu{\fM}{U}=\bigcap_{\bM \in \fM}\Beu{\bM}{U}$
and $\E^{\{\fM\}}=\bigcup_{\bM \in \fM}\E^{\{\bM\}}$ in the sense of germs. 
It follows that $\DC{\fM}{U}$ is an algebra due to the
definition of the weight matrix.

Let $\fM$ and $\fN$ be two weight matrices. We define 
\begin{align*}
\fM \{\preceq\} \fN \quad &:\Longleftrightarrow  \quad  \fa \bM \in \fM \ex \bN \in \fN : \bM \preceq \bN,
\\
\fM (\preceq) \fN \quad &:\Longleftrightarrow  \quad  \fa \bN\in \fN \ex \bM \in \fM : \bM \preceq \bN,
\\
\fM \{\lhd) \fN \quad &:\Longleftrightarrow  \quad  \fa \bM\in \fM \fa \bN \in \fN : \bM \lhd \bN.
\end{align*}
Furthermore we write $\fM[\approx]\fN$ if $\fM[\preceq]\fN$ and $\fN[\preceq]\fM$.
It follows that $\DC{\fM}{U}\subseteq\DC{\fN}{U}$ if  $\fM[\preceq]\fN$ and
$\Rou{\fM}{U}\subseteq\Beu{\fN}{U}$ if $\fM\{\lhd)\fN$.

For each weight matrix $\fM$ there exists a countable weight matrix 
$\fL$ by \cite[Lemma 2.5]{FURDOS2020123451} such that
$\fM[\approx]\fL$, i.e.\ $\DC{\fM}{U}=\DC{\fL}{U}$.
It follows that $\Rou{\fM}{U}$ is non-quasianalytic if and only if
there is some non-quasianalytic sequence $\bM\in\fM$.
On the other hand 
$\Beu{\fM}{U}$ is non-quasianalytic if and only if all sequences
$\bM\in\fM$ are non-quasianalytic, see \cite[Sect. 4]{MR3601829}.
We should note that in the last statement in particular the fact
that $\fM$ is equivalent to a countable weight matrix is important:
The intersection of uncountably many non-quasianalytic Denjoy-Carleman classes might be quasianalytic, cf. \cite{MR220049}.

Combined with Remark \ref{SeqRem}  we moreover conclude that $\An(U)\subsetneq\DC{\fM}{U}$ when
\begin{equation}\label{matrixAnal}
\fa\bM\in\fM:\quad \lim_{k\rightarrow \infty}\sqrt[k]{m_k}=\infty.
\end{equation}
A weight matrix $\fM$ is called $R$-semiregular 
if $\fM$ satisfies \eqref{matrixAnal} and
\begin{align}
\label{RouDeriv}\fa \bM\in\fM\ex \bN\in\fM\ex C>0:\quad M_{k+1}&\leq C^{k+1}N_k, \qquad k\in\N_0,\\
\intertext{and $B$-semiregular if \eqref{matrixAnal} and}
 \fa \bN\in\fM\ex \bM\in\fM\ex C>0:\quad M_{k+1}&\leq C^{k+1}N_k, \qquad k\in\N_0,\label{BeuDeriv}
\end{align} 
hold. 
We say that $\fM$ is semiregular if the matrix is $R$- and $B$-semiregular.
We also write $[$semiregular$]$ when we mean $R$-semiregular or $B$-semiregular
depending on the case considered.
\begin{Rem}\label{DerivClosedAlternative}
	We list here some consequences of the conditions above.
	\begin{enumerate}
\item Let $\fM$ be a weight matrix.
The spaces $\Rou{\fM}{U}$ and $\Beu{\fM}{U}$ are closed under derivation when $\fM$ 
satisfies \eqref{RouDeriv} or \eqref{BeuDeriv}, respectively.
If $\fM$ is $[$semiregular$]$ then $\E^{[\fM]}$ is closed under
 composition with analytic mappings, see  \cite[Theorem 2.9]{FURDOS2020123451}.
 Hence, if $X$ is an analytic manifold, then $\DC{\fM}{X}$ is well defined.
 \item If $\ell\in\N$ is fixed and $\bM,\bN$ are two weight sequences satisfying
 $M_{k+\ell}\leq \gamma^{k+\ell}N_k$ for some constant 
 $\gamma>0$ independent of $k\in\N_0$ then 
 there are constants $C,h>0$ such that
 \begin{equation}\label{AltDerivClosed}
 (M_k)^{\tfrac{k+\ell}{k}}\leq Ch^kN_k, \qquad \quad k\in\N.
 \end{equation}
Indeed, it follows from \eqref{logconvexity} that the sequence $((L_k)^{1/k})_k$ is increasing 
for any weight sequence $\bL$. Thus we have
\begin{equation*}
	(M_k)^{1/k}\leq (M_{k+\ell})^{1/(k+\ell)}\leq \gamma (N_{k})^{1/(k+\ell)}
\end{equation*}
for all $k\in\N$.

 Hence if $\fM$ is a weight matrix satisfying \eqref{RouDeriv} then by iterating \eqref{RouDeriv} we obtain
 for each $\bM\in\fM$ and $\ell\in\N$ there are $\bN\in\fM$ and
  constants $C,h>0$ such that
 \eqref{AltDerivClosed} holds.
  Similarly for a weight matrix $\fM$ with
 \eqref{BeuDeriv} we have that for all $\bN\in\fM$ and $\ell\in\N$ there
 exist $\bM\in\fM$ and $C,h>0$ satisfying \eqref{AltDerivClosed}.
 \item An equivalent condition to \eqref{RouDeriv} is 
 \begin{align*}
 	\fa\bM\in\fM \ex\bN\in\fM\;\ex C>0:\; m_{k+1}\leq C^{k+1}n_k\qquad k\in\N_0.
 \end{align*}
Similarly, we can without loss of generality replace in $\eqref{BeuDeriv}$ 
$M_{k+1}$ and $N_{k}$ by $m_{k+1}$ and $n_k$, respectively.
 \end{enumerate}
\end{Rem}

\begin{Ex}\label{WMatrixExamples}
	Here are some families of weight matrices which will play a prominent role later on.
	\begin{enumerate}
		\item The Gevrey matrix $\fG=\{\bG^s:s>1\}$ is semiregular.
		Both spaces $\DC{\fG}{U}$ are non-quasianalytic and furthermore
		we have the identity $\Beu{\fG}{U}=\bigcap_{s>1}\Rou{\bG^s}{U}=
		\bigcap_{s>1}\G^s(U)$ since $\bG^s\lhd\bG^t$ for $s<t$, cf.\ 
		\cite{MR3285413}.
		\item Using the sequences $\bL^{q,r}$ we can define two families of semiregular matrices:
		\begin{align*}
		\fQ^r&=\left\{\bL^{q,r}:q>1\right\} & r&>1,\\
		\fR^q&=\left\{\bL^{q,r}:r>1\right\} & q&>1.
		\end{align*}
		Since $\bL^{q_0,r_0}\lhd\bL^{q_1,r_1}$ for $r_1>r_0$ and all $q_0,q_1>1$
		we have that $\fR^q[\approx]\fR^{q^\prime}$ for any $q,q^\prime>1$.
		Hence if we set $\fR=\fR^e$ then $\DC{\fR}{U}=\DC{\fR^q}{U}$ for all $q>1$.
		Furthermore $\fQ^r[\preceq]\fQ^{r^\prime}$ for all $r<r^\prime$ and
		$\fR(\preceq)\fQ^r$ and $\fQ^r\{\preceq\}\fR$ for all $r>1$.
		We note finally that $\fG\{\lhd)\fR$.
		\item For $j\in\N$ consider the semiregular weight matrix 
		$\fB^j=\{\bB^{j,\sigma}:\,\sigma>0\}$. 
		Then the spaces $\Beu{\fB^1}{U}$ and $\DC{\fB^j}{U}$, $j\geq 2$,
		are quasianalytic and $\Rou{\fB^1}{U}$ is non-quasianalytic.
	
	Analogously to above we set
	\begin{equation*}
	\fJ^\sigma=\left\{\bB^{j,\sigma}:j\in\N\right\},\quad \sigma>0.
	\end{equation*}
	Then $\fJ^\sigma(\preceq)\fB^{j,\sigma}$ for all $j\in\N$ and $\sigma>0$.
	Clearly we have also that $\fJ^{\sigma_1}(\approx)\fJ^{\sigma_2}$ for all
	 $\sigma_1,\sigma_2>0$  
	 and finally $\fJ^\sigma\{\approx\}\{\bB^{1,\sigma}\}$ for any $\sigma>0$.
	\end{enumerate}
\end{Ex}
\begin{Not}
	We say that $f$ is an ultradifferentiable function of class $[\ast]$ in $U$,
	$\ast=\bM,\fM,\omega$, if $f\in\DC{\ast}{U}$. 
\end{Not}
\subsection{Weight functions}\label{subsec:Weightfct}
In this section we discuss briefly the relationship between weight functions 
and weight matrices, as described in \cite{MR3285413}. 

Recall that a weight function $\omega: [0,\infty)\rightarrow[0,\infty)$  
in the sense of \cite{MR1052587}
is  continuous, increasing, $\omega(0)=0$, $\omega(t)\rightarrow \infty$ and 
satisfies the conditions \eqref{om2}, \eqref{om3} and \eqref{om4}.
If $\sigma,\tau$ are two weight functions then we write
\begin{align*}
\sigma\preceq\tau\quad & :\Longleftrightarrow\quad \tau(t)=O(\sigma(t))\quad
\text{if } t\rightarrow\infty,\\
\sigma\lhd\tau \quad & :\Longleftrightarrow \quad\tau(t)=o(\sigma(t))
\quad \text{ if }t\rightarrow\infty.
\end{align*}
It follows that $\DC{\sigma}{U}\subseteq\DC{\tau}{U}$ 
if $\sigma\preceq\tau$ and
$\sigma\lhd\tau$ implies $\Rou{\sigma}{U}\subseteq\Beu{\tau}{U}$.
We write $\omega\sim\sigma$ when $\omega\preceq\sigma$ and $\sigma\preceq\omega$.

\begin{Rem}
	\begin{enumerate}
		\item It is well known that the weight function $t^{1/s}$ generates
		the Gevrey class of order $s$, i.e.\ $\G^s(U)=\Rou{t^{1/  s}}{U}$ and
		in particular, for $s=1$, $\An(U)=\G^1(U)=\Rou{t}{U}$ is the space
		of analytic functions.
		It follows that if $\omega$ is a weight function such that 
		$\omega(t)=o(t^\alpha)$ for some $0<\alpha\leq 1$ then
		\begin{equation*}
		\G^{1/\alpha}(U)\subseteq\DC{\omega}{U}.
		\end{equation*}
		\item In general, weight sequences and weight functions
		describe distinct spaces, see \cite{BonetMeiseMelikhov07}.
		\item The space $\DC{\omega}{U}$ is quasianalytic if and only if
		\begin{equation}\label{omNQ}
		\int\limits_{1}^{\infty}\frac{\omega(t)}{t^2}\,dt=\infty.
		\end{equation}
		We say that a weight function $\omega$ is quasianalytic if $\omega$ 
		satisfies \eqref{omNQ} and non-quasianalytic otherwise. 
		If $\omega$ is non-quasianalytic
		then $\omega(t)=o(t)$ for $t\rightarrow\infty$. 
	\end{enumerate}
\end{Rem}

According to \cite{MR1052587} we can without loss of generality
 assume that $\omega$ vanishes on  $[0,1]$. 
Then the Young conjugate $\varphi^\ast_\omega(t)=\sup_{s\geq 0}(st-\varphi_\omega(s))$
 of $\varphi_\omega=\omega\circ\exp$ is
 convex, increasing,
 $\varphi_\omega^\ast(0)=0$ and $\varphi_\omega^{\ast\ast}=\varphi_\omega$.\\
Furthermore both functions 
	\begin{equation*}
	t\longmapsto \frac{\varphi_\omega(t)}{t}\quad
	\text{and}\quad
	t\longmapsto \frac{\varphi^\ast_\omega(t)}{t}
	\end{equation*}
	are increasing on $[0,\infty)$.

\begin{Def}
	Let $\omega$ be a weight function such that $\omega\vert_{[0,1]}=0$.
	The associated weight matrix $\fW=\{\bW^\lambda=(W^\lambda_k)_k:\;\lambda>0\}$
	to $\omega$ is given by 
	\begin{equation}
	W^\lambda_k=\exp\left[\lambda^{-1}\varphi_\omega^\ast(\lambda k)\right].
	\end{equation}
\end{Def}
We summarize the basic properties of $\fW$ from 
\cite[Section 5]{MR3285413}:
First, each sequence $\bW^\lambda$ is a weight sequence and $\fW$ satisfies 
\begin{equation}\label{omMG}
W^\lambda_{j+k}\leq W^{2\lambda}_jW^{2\lambda}_k
\end{equation}
for all $j,k\in\N_0$ and all $\lambda>0$. 
We note that \eqref{omMG} implies \eqref{RouDeriv} and \eqref{BeuDeriv}.
Furthermore we have
\begin{equation}\label{newexpabsorb}
\fa h\geq 1\ex A\ge 1\fa\ell>0\ex D\geq 1\fa j\in\N_0:\;\;\;h^jW^\ell_j\leq D W^{A\ell}_j.
\end{equation}
From \eqref{newexpabsorb} we obtain that
\begin{equation*}
\DC{\omega}{U}=\DC{\fW}{U}
\end{equation*}
as topological vector spaces.
Finally, $\omega(t)=o(t)$ if and only if 
\begin{equation*}
\lim_{k\rightarrow\infty}\left(w_k^\lambda\right)^{1/k}=\infty
\end{equation*}
for all $\lambda>0$, where $w^\lambda_k=W^\lambda_k/k!$.
\begin{Prop}
\begin{enumerate}
\item Let $\omega$ be a weight function with $\omega(t)=o(t)$ when $t\rightarrow\infty$.
Then the associated weight matrix is semiregular.
\item If $\sigma$, $\tau$ are two weight functions with associated weight matrices $\fS$, $\fT$ then
$\sigma\sim\tau$ if and only if $\fS\{\approx\}\fT$ and $\fS(\approx)\fT$.
\end{enumerate}
\end{Prop}
\begin{Ex}\label{q-GevreyOmega}
	 	Let $s>1$.
		If we consider the associated weight matrix $\fW^s=\{\bW^{\lambda,s}:\;\lambda>0\}$ of 
		the weight function $\omega_s(t)=(\max\{0,\log t\})^s$ then we
		observe that after a reparametrization of the matrix we have
		$W^{\lambda ,s}_k=e^{\lambda k^r}$ where $r=s/(s-1)$, i.e.\
		the parameters $s$ and $r$ are conjugated:
		 $\tfrac{1}{s}+\tfrac{1}{r}=1$, see
		\cite[subsection 5.5]{MR3711790}. 
	If $q=e^\lambda$ then clearly $\bW^{\lambda,s}\preceq\bL^{q,r}$ and
	it is easy to see that $\bL^{q,r}\lhd \bW^{\lambda^\prime,s}$ when
	$\lambda<\lambda^\prime$. 
	It follows that $\fW^s[\approx]\fQ^r$.
		Hence 
		\begin{equation*}
		\DC{\omega_s}{U}=\DC{\fQ^r}{U}, \qquad \tfrac{1}{s}+\tfrac{1}{r}=1.
		\end{equation*}
\end{Ex}

\subsection{The ultradifferentiable wavefront set}
The ultradifferentiable wavefront set for Roumieu classes given by weight sequences
was introduced in \cite{MR0294849}.
In \cite{MR2595651}  the wavefront set was defined in the case of weight functions.
These definitions have been generalized by \cite{FURDOS2020123451} to 
the category of classes given by weight matrices.

For the convenience of the reader
we recall from \cite{FURDOS2020123451} the definition of the wavefront set associated
to classes given by weight matrices.
We give also a summary of the results we need later on,
observing in particular that, in analogy to the results of \cite{MR0294849} in the case of a single weight sequence, semiregularity of the weight matrix is 
sufficient for the ultradifferentiable microlocal elliptic regularity Theorem
to hold for operators with analytic coefficients;
a fact that was not explicitly stated in \cite{FURDOS2020123451} because in that
paper we worked in a more general setting.

We define the Fourier transform of a distribution $u\in\Ep(U)$ to be
\begin{equation*}
\hat{u}(\xi)=\F(u)(\xi)=\left\langle u(x),e^{ix\xi}\right\rangle_x,
\end{equation*}
where the bracket on the right-hand side denotes the distributional action.
\begin{Def}\label{WF-Definition}
	Let $\fM$ be a weight matrix, $u\in\D^\prime(U)$ and 
$(x_0,\xi_0)\in \COT$. Then
	\begin{enumerate}
		\item  $(x_0,\xi_0)\notin\WF_{\{\fM\}}u$
		 iff there exist a neighborhood $V$ of $x_0$, 
		 a conic neighborhood $\Gamma$ of $\xi_0$, 
		and a bounded sequence $(u_k)_k\subseteq\E^\prime(U)$ 
		with $u_k\vert_V=u\vert_V$  such that for some $\bM\in\fM$ and a constant $h>0$ we have
		\begin{equation}\label{WF-defining}
		\sup_{\substack{\xi\in\Gamma\\ k\in\N}}
		\frac{|\xi|^k\bigl|\widehat{u}_k(\xi)\bigr|}{h^k M_k}<\infty,
		\end{equation}
		\item  $(x_0,\xi_0)\notin\WF_{(\fM)}u$ iff 
		there exist a neighborhood $V$ of $x_0$, a conic neighborhood $\Gamma$ of $\xi_0$, and  
		a bounded sequence $(u_k)_k\subseteq\E^\prime(U)$ with $u_k\vert_V=u\vert_V$
		such that  \eqref{WF-defining} is satisfied
		for all $\bM\in\fM$ and all $h>0$.
	\end{enumerate}
\end{Def}
The basic properties of the ultradifferentiable 
wavefront set are summarized in the following
Proposition.
\begin{Prop}[{\cite[Proposition 5.4(1)--(4)]{FURDOS2020123451}}]\label{WFBasicProperties}
	Let $\fM,\fN$ be weight matrices and $u\in\Dp(U)$. 
	Then the following statements hold:
	\begin{enumerate}
		\item $\WF_{[\fM]} u$ is a closed subset of $\COT$ which is conic in the second variable.
		\item $\WF_{\{\fM\}}u\subseteq\WF_{(\fM)}u$.
		\item $\WF u\subseteq\WF_{[\fN]}u\subseteq \WF_{[\fM]}u$ if $\fM[\preceq]\fN$.
		\item $\WF_{(\mathfrak{N})}u \subseteq \WF_{\{\fM\}}u$ if $\fM \{\lhd)\mathfrak{N}$.
	\end{enumerate}
\end{Prop}
If we assume that $\fM$ satisfies additional conditions then we can show 
more properties of $\WF_{[\fM]}u$:
\begin{Prop}[{\cite[Proposition 5.6(1)]{FURDOS2020123451}}]\label{WF-Intersection}
	Let $\fM$ be a weight matrix satisfying \eqref{matrixAnal}
	 and $u\in\Dp(U)$.
		We have 
		\begin{equation*}
		\WF_{\{\fM\}} u = \bigcap_{\bM \in \fM} \WF_{\{\bM\}} u 
		\quad \text{ and } \quad
		\WF_{(\fM)} u = \overline{\bigcup_{\bM \in \fM} \WF_{(\bM)} u}.
		\end{equation*}
\end{Prop}
Similar to the smooth category we define the $[\fM]$-singular support $\singsupp_{[\fM]} u$ of a distribution $u\in\Dp(U)$ 
as the complement of the largest subset $V\subseteq U$ such that $u\vert_V\in\DC{\fM}{V}$. We have
\begin{Prop}[{\cite[Proposition 5.4(6)]{FURDOS2020123451}}]\label{Singsupport}
	Let $\fM$ be a $[$semiregular$]$ weight matrix and $u\in\Dp(U)$.
Then
\begin{equation*}
\pi_1\left(\WF_{[\fM]}u\right)=\singsupp_{[\fM]} u
\end{equation*}
where $\pi_1:\COT\rightarrow U$ is the projection to the first variable.

\end{Prop}

It is possible to choose the distributions $u_k$ in Definition \ref{WF-Definition}
 in a special manner.
For our purpose a simplified version of \cite[Lemma 5.3]{FURDOS2020123451} is sufficient:
\begin{Lem}\label{WFalternative}
	Let $\fM$ be $[$semiregular$]$, $u\in\Dp(U)$,
$K\subseteq U$ a compact subset and $F\subseteq\R^n\setminus\{0\}$ a closed cone
such that 
\begin{equation*}
\WF_{[\fM]} u\cap K\times F=\emptyset.
\end{equation*}
Furthermore assume that $\chi_k\in\D(U)$ is a sequence of functions with common support in
$K$ and for all $\alpha\in\N_0^n$ there are constants $C_\alpha,h_\alpha>0$
such that
	\begin{equation*}
	\Betr{D^{\alpha+\beta}\chi_k}\leq C_\alpha h_\alpha^\bet k^\bet,
	\qquad \bet\leq k.
	\end{equation*}
	If $\mu$ is the order of $u$ near $K$ then 
	the sequence $(\chi_k u)_k$ is bounded in
	$\E^{\prime,\mu}(K)$ and
	\begin{enumerate}
		\item in the Roumieu case we have
	\begin{equation}\label{AltWFestimate}
	\sup_F \xit^{k}\Betr{\F(\chi_k u)(\xi)}\leq Ch^k M_{k}
	\end{equation}
	for some $\bM\in\fM$ and $C,h>0$.
	\item In the Beurling case for all $\bM\in\fM$ and all $h>0$ there is
	$C=C_{\bM,h}$ such that the estimate \eqref{AltWFestimate} holds.
	
\end{enumerate}
\end{Lem}
Lemma \ref{WFalternative} allows us directly to generalize the proof of 
\cite[Theorem 5.4]{MR0294849} in order to show the microlocal elliptic regularity theorem
 for classes
given by weight matrices and operators with analytic coefficients. 
For a more general version see \cite[Theorem 7.1]{FURDOS2020123451}. 
\begin{Thm}\label{EllipticWFThm}
	Let $P$ be a differential operator with analytic coefficients on $U$ and $\fM$ be a $[$semiregular$]$ weight matrix.
	Then we have
	\begin{equation*}
	\WF_{[\fM]}Pu\subseteq\WF_{[\fM]}u\subseteq\WF_{[\fM]} (Pu)\cup\Char P
	\end{equation*}
	for all $u\in\Dp(U)$.
\end{Thm}
\section{Ultradifferentiable vectors}\label{UVectors}
\subsection{Microlocal theory}
The aim of this section is to generalize the microlocal theory presented in \cite{MR557524} to the setting of weight matrices.
In order to accomplish this we have to use a more generalized notion of
vectors than the one from Section \ref{Introduction}.
For this we need to recall some notions.

Let $\sigma\in\R$. We denote the Sobolev space of order $\sigma$ by
$H^\sigma(\R^n)$, which is equipped with the norm
\begin{equation*}
\norm[\sigma]{g}=\left(\int\!(1+\xit)^{2\sigma}
\Betr{\hat{g}(\xi)}^2\,d\xi\right)^{\tfrac{1}{2}}.
\end{equation*}

The localized Sobolev space $H^\sigma_{loc}(U)$ consists of those distributions $g\in\Dp(U)$ which
satisfy $\varphi g\in H^\sigma(\R^n)$ for all $\varphi\in\D(U)$.
It is a locally convex space whose topology is given by the seminorms 
\begin{equation*}
	g\longmapsto \norm[\sigma]{\varphi g}.
\end{equation*}
A different seminorm on $H^\sigma_{loc}(U)$ is 
\begin{equation}\label{SobolevNorm}
	\norm[H^\sigma(V)]{g}=\inf\left\{\norm[\sigma]{G}:\; G\in H^\sigma(\R^n),\; G\vert_V=g\right\}
\end{equation}
where $V\Subset U$. 

\begin{Def}\label{DefinitionVectors}
	Let $\fM$ be a weight matrix, $\sP=\{P_1,\dotsc,P_\ell\}$ 
	a system of differential operators 
	of order $d_j$, $j=1,\dotsc,\ell$,
	with analytic coefficients in the open set $U\subseteq\R^n$ and $\sigma\in\R$.
	If $V\Subset U$, $\bM\in\fM$ and $h>0$ then we set
	\begin{equation*}
	\vNorm[\sP,\sigma]{V,\bM,h}{u}=\sup_{\substack{k\in\N_0\\\alpha\in\{1,\dotsc,\ell\}^k}}
	\frac{\norm[H^\sigma(V)]{P^\alpha u}}{h^{d_\alpha}M_{d_\alpha}}
	\end{equation*}
	where $P^\alpha=P_{\alpha_1}\dots P_{\alpha_k}$,
	$d_\alpha=d_{\alpha_1}+\dots +d_{\alpha_k}$ and
	$u\in\Dp(U)$ such that $P^\alpha u\in H^\sigma_{loc}(U)$ for all $\alpha\in\{1,\dotsc,\ell\}^k$ and $k\in\N_0$.
	In the case $k=0$ we use the convention $\{1,\dotsc,\ell\}^0=0$ and $d_0=0$.
	We set
	\begin{align*}
	\E^{\{\fM\}}_\sigma(U;\sP)&=
	\left\{u\in\Dp(U):\; \fa V\Subset U\ex \bM\in\fM\ex h>0:\;
	\vNorm[\sP,\sigma]{V,\bM,h}{u}<\infty\right\},\\
	\E_\sigma^{(\fM)}(U,\sP)&=
	\left\{u\in\Dp(U):\; \fa V\Subset U\fa \bM\in\fM\fa h>0:\;
	\vNorm[\sP,\sigma]{V,\bM,h}{u}<\infty\right\}.
	\end{align*}
An element of
$\vDC[\sP]{\fM}{U}=\bigcup_\sigma\E_\sigma^{[\fM]}(U;\sP)$
is called an ultradifferentiable vector of class $[\fM]$
(or an $[\fM]$-vector)
of the system $\sP$.
We also  define  $\E_{loc}^{[\fM]}(U;\sP)$ to be the space of those
distributions $u\in\Dp(U)$ such that for all $x_0\in U$ 
 there is a neighborhood $V$ of $x_0$ such that 
 $u\vert_V\in\vDC[\sP]{\fM}{V}$.
If $\sP=\{P\}$ consists of a single operator then we write
$\vDC{\fM}{U}=\vDC[\sP]{\fM}{U}$.
\end{Def}

\begin{Prop}\label{VectorProperties}
	Let $\fM$, $\fN$ be weight matrices and $\sP$ be a system of 
	analytic differential operators. Then the following holds:
	\begin{enumerate}
		\item If $\fM[\preceq]\fN$ then $\vDC[\sP]{\fM}{U}\subseteq\vDC[\sP]{\fN}{U}$.
		\item If $\fM\{\lhd)\fN$ then $\vRou[\sP]{\fM}{U}\subseteq\vBeu[\sP]{\fN}{U}$.
		\item If $\fM$ satisfies \eqref{matrixAnal} then $\DC{\fM}{U}\subseteq\vDC[\sP]{\fM}{U}$.
	\end{enumerate}
\end{Prop}
\begin{proof}
	If $\fM\{\preceq\}\fN$ and $V\Subset U$ are given
then for all $\bM\in\fM$ and all $h>0$ there are $\bN\in\fN$ and 
$h^\prime,C>0$ such that 
	\begin{equation*}
	\vNorm[\sP,r]{V,\bN,h^\prime}{u}\leq C\vNorm[\sP,r]{V,\bM,h}{u}
	\end{equation*}
	for $u\in\Dp(U)$. Hence (1) holds in the Roumieu case. 
	The proofs of the other statements in (1) and (2) are similar.
	
	In order to show (3), recall that \eqref{matrixAnal} implies that for
	all $\bM\in\fM$ and all $\gamma>0$ there is a constant $C>0$ such that
	\begin{equation*}
	k!\leq C\gamma^kM_k,\quad k\in\N_0.
	\end{equation*}
	We may also note that if $Q=\sum_{\bet\leq d}a_\beta(x)D^\beta$
	is an operator with analytic coefficients in $U$ then for each
	$V\Subset U$ we can 
	find constants $C,r>0$ such that 
	\begin{equation*}
	\Betr{D^\alpha a_\beta(x)}\leq C r^\alp \alp !, \quad x\in V,\alpha\in\N_0^n.
	\end{equation*}
	This means that for $f\in\Rou{\fM}{U}$,
	 and $\alpha\in\N_0^n$
	we can estimate in $V\Subset U$ that
	\begin{equation*}
	\begin{split}
	\sup_{x\in V}\Betr{D^\alpha Qf(x)}&
	\leq C\sum_{\bet\leq d}\sum_{\alpha^\prime\leq\alpha}
	\binom{\alpha}{\alpha^\prime}r^{\lvert\alpha-\alpha^\prime\rvert}\left\lvert \alpha-\alpha^\prime\right\rvert!
	h^{\lvert\alpha^\prime+\beta\rvert}M_{\lvert\alpha^\prime+\beta\rvert}\\
	&\leq C\sum_{\bet\leq d}\sum_{\alpha^\prime\leq\alpha}
	\binom{\alpha}{\alpha^\prime}h^{\lvert\alpha-\alpha^\prime\rvert}
	M_{\lvert\alpha-\alpha^\prime\rvert}
	h^{\lvert\alpha^\prime+\beta\rvert}M_{\lvert\alpha^\prime+\beta\rvert}\\
	&\leq C2^{\alp}\sum_{\bet\leq d}h^{\alp+\bet}M_{\alp+\bet}\\
	&\leq C(2h)^{\alp +d}M_{\alp+d}
	\end{split}
	\end{equation*}
for some weight sequence $\bM\in\fM$ and $h\geq 1$.
Iterating this argument we conclude that there are a weight sequence $\bM\in\fM$ and
constants $C,h>0$ such that, when $\alpha\in\{1,\dotsc \ell\}^k$,
$k\in\N$, we have
\begin{equation*}
\sup_{x\in V}\Betr{P^\alpha f(x)}\leq Ch^{d_\alpha}M_{d_\alpha},
\end{equation*}
where $P^\alpha$ and $d_\alpha$ are defined as in Definition \ref{DefinitionVectors}.
Therefore
\begin{equation*}
\norm[L^2(V)]{P^\alpha f}\leq Ch^{d_\alpha}M_{d_\alpha}
\end{equation*}
for some constants $C,h>0$ independent of $k\in\N$ and
 $\alpha\in\{1,\dotsc,\ell\}^k$ and hence $f\in\vRou[\sP]{\fM}{U}$.
 
 If $f\in\Beu{\bM}{U}$ and $V\Subset U$ then we define a sequence
 $\bL^\prime\lhd\bM$ by setting
 \begin{equation*}
L^\prime_k=\max\left\{k!,\sup_{x\in V}\sup_{\alp \leq k}
	\Betr{D^\alpha f(x)}\right\}.
 \end{equation*}
 According to Lemma \ref{AuxillaryLem1} for each $\bM\in\fM$
 there is a weight sequence $\bN$ such
 that $\bG^1\leq \bL^\prime\leq \bN\lhd\bM$ and by construction we have that
 there are constants $\gamma>0$ and $C>0$ such that
 \begin{equation*}
 \Betr{D^\alpha f(x)}\leq C\gamma^\alp N_\alp,\quad \alpha\in\N^n_0,
 \end{equation*}
for $x\in V$. 
We obtain
\begin{equation*}
\begin{split}
\sup_{x\in V}\Betr{D^\alpha Qf(x)}
&\leq C\sum_{\bet \leq d}\sum_{\alpha^\prime\leq\alpha}
\binom{\alpha}{\alpha^\prime}r^{\lvert\alpha-\alpha^\prime\rvert}
\Betr{\alpha-\alpha^\prime}!\gamma^{\lvert\alpha^\prime-\beta\rvert}
N_{\lvert\alpha^\prime+\beta\rvert}\\
&\leq C\sum_{\bet \leq d}\sum_{\alpha^\prime\leq\alpha}
\binom{\alpha}{\alpha^\prime}r^{\lvert\alpha-\alpha^\prime\rvert}
N_{\lvert\alpha-\alpha^\prime\rvert}\gamma^{\lvert\alpha^\prime-\beta\rvert}
N_{\lvert\alpha^\prime+\beta\rvert}\\
&\leq C\sum_{\bet \leq d}
h_1^{\alp+\bet}N_{\alp+\bet}\\
&\leq Ch_1^{\alp+d}N_{\alp+d}.
\end{split}
\end{equation*}
Thence for each $\bM\in\fM$ and $h>0$ there is a constant $C>0$ such that
\begin{equation*}
\sup_{x\in V}\Betr{D^\alpha Qf(x)}\leq Ch^{\alp+d}M_{\alp+d}.
\end{equation*}
 From this estimate it follows in the same manner as in the Roumieu case
 that $f\in\vBeu[\sP]{\fM}{U}$.
\end{proof}

\begin{Rem}\label{SubRemark}
	Traditionally, the $L^2$-norm is mainly used in the definition of vectors, but
	 in the literature the norm in the definition of vectors
	is chosen according to the techniques used in the paper in question, see e.g.\ the discussion in \cite{MR1037999}.
	We have already mentioned that
	Definition \ref{DefinitionVectors} is more general than the definition of vectors used in Section \ref{Introduction}, because, as we will see in a moment,
	Definition \ref{DefinitionVectors} is microlocalizable, cf.\ \cite{MR557524} and
	\cite{MR632764}.
	
	However, cf.\ \cite{MR1037999}, if the system $\sP=\{P_1,\dotsc,P_\ell\}$ is 
	subelliptic, that is for each $V\Subset U$ there is
	$\eps>0$ such that for all $\sigma\in\R$ the estimate
	\begin{equation}\label{Subelliptic}
	\norm[\sigma+\eps]{\varphi}\leq 
	C\left[\sum_{j=1}^\ell \norm[\sigma]{P_j\varphi}+\norm[\sigma]{\varphi}\right],
	\qquad \varphi\in\D(V),
	\end{equation}
	holds
	for some $C>0$, then we obtain that
	\begin{equation*}
\E_\sigma^{[\fM]}(U;\sP)=\vDC[\sP]{\fM}{U}
	\end{equation*}
for all $\sigma\in\R$	when $\fM$ is $[$semiregular$]$.
	
	Indeed, if $u\in\E_\sigma^{[\fM]}(U;\sP)$ then by definition
	$P^\alpha u\in H^\sigma_{loc}(U)$ for all $\alpha$.
	It is well known that \eqref{Subelliptic} implies therefore 
	$u\in\E(U)=H^\infty_{loc}(U)$, see e.g.\ \cite{MR2304165}
	or \cite{MR132896}.
	Furthermore, $\E_{\sigma}^{[\fM]}(U;\sP)\subseteq\E_{\tau}^{[\fM]}(U;\sP)$
	for $\tau\leq \sigma$ since $\norm[\tau]{g}\leq \norm[\sigma]{g}$ for all 
	$g\in H^\sigma(\R^n)$.
	
	If now $V$ and $W$ are two open sets with $V\Subset W\Subset U$ then 
	\eqref{Subelliptic} implies that
	\begin{equation}\tag{3.2$^\prime$}\label{Subelliptic1}
		\norm[H^{\sigma+\eps}(V)]{f}\leq C\left[\sum_{j=1}^\ell\norm[H^{\sigma}(V)]{P_jf}+\norm[H^\sigma(V)]{f}\right],
		\qquad f\in\E(U),
	\end{equation}
where $\eps$ is the subellipticity index of $W$, see \ref{SubAppendix}.
	
	We suppose for a moment that $\bM, \bN\in \fM$ are two weight sequences for
	which there exists a constant $\gamma\geq 1$ such that
	\begin{equation}\label{SubellipticA}
	M_{k+1}\leq \gamma^{k+1}N_k,\qquad k\in\N_0.
	\end{equation}
If we combine \eqref{Subelliptic1} with \eqref{SubellipticA} we conclude that
\begin{equation*}
	\vNorm[\sP,\sigma+\eps]{V,\bN,h}{u}\leq C\vNorm[\sP,\sigma]{V,\bM,h/\gamma}{u}
\end{equation*}
for $u\in\E(U)$. 

If $\fM$ is $B$-semiregular and $u\in\E_{\sigma}^{(\fM)}(U;\sP)$ then by definition
\begin{equation*}
	\vNorm[\sP,\sigma]{V,\bM,h}{u}<\infty
\end{equation*}
for all $V\Subset U$, all $\bM\in\fM$ and all $h>0$.
Thence, by the above arguments we can conclude
that actually
\begin{equation*}
	\vNorm[\sP,\sigma]{V,\bM,h}{u}<\infty
\end{equation*}
for all $V\Subset U$, all $\bM\in\fM$, all $h>0$ and every $\sigma\in\R$, that is
\begin{equation*}
	\E_{\sigma}^{(\fM)}(U;\sP)=\E_{\tau}^{(\fM)}(U;\sP)
\end{equation*}
for all $\sigma,\tau\in\R$.
The Roumieu case follows similarly.
\end{Rem}

We are now able to begin to extend the microlocal theory developed in \cite{MR557524} for 
Roumieu vectors given by a semiregular weight sequence of an operator with analytic coefficients
to vectors associated to a [semiregular] weight matrix.
We follow mainly the presentation given in \cite{MR632764}.
We start with a characterization of the property of being a vector
by the Fourier transform.
\begin{Thm}\label{FourierCharThm}
Let $P$ be a differential operator of order $d$ with analytic coefficients in $U$, $u\in\Dp(U)$, $x_0\in U$
and $\fM$ be a  weight matrix.
Then
\begin{enumerate}
	\item $u\in\vRou{\fM}{V}$ for some neighborhood $V$ of $x_0$ if
	and only if there are a neighborhood $W$ of $x_0$ and a sequence 
	$f_k\in\Ep(U)$ such that $f_k\vert_W=\left(P^ku\right)\vert_W$
	and 
	\begin{equation}\label{microlocalEst1}
	\Betr{\hat{f}_k(\xi)}\leq Ch^{k}M_{dk}\bigl(1+\Betr{\xi}\bigr)^\nu,\qquad
	\qquad \fa\xi\in\R^n,
	\end{equation}
	for a sequence $\bM\in\fM$ and some constants $C,h>0$ and $\nu\in\R$.
	\item $u\in\vBeu{\fM}{V}$ for some neighborhood $V$ of $x_0$ if
	and only if there are a neighborhood $W$  of $x_0$, a sequence $f_k\in\Ep(U)$
	and a constant $\nu\in\R$ such that  $f_k\vert_W=\left(P^ku\right)\vert_W$
	and for all $\bM\in\fM$ and every 
	$h>0$ there is some $C>0$ so
	\eqref{microlocalEst1} is satisfied.
\end{enumerate}
\end{Thm}
\begin{proof}
We begin with the Roumieu case. Hence suppose that
 $u\in\E^{\{\fM\}}_\sigma(V;P)$ for some neighborhood $V$ of $x_0$ and $\sigma\in\R$. 
 Following \cite{MR632764} let $W_2\Subset W_1\Subset V$ be two neighborhoods of $x_0$ and choose
$\varphi,\psi\in\D(W_1)$ with
 $\psi\varphi=\varphi$ 
and $\varphi = 1$ in $W_2$.
If we set $f_k=\varphi P^ku$ then $f_k\in\Ep(V)$ and $f_k=P^ku$ in $W_2$.
Furthermore
\begin{equation*}
\begin{split}
\Betr{\hat{f}_k(\xi)}
&=\Betr{\mathcal{F}\left(\psi\varphi P^k u\right)(\xi)}
=\frac{1}{(2\pi)^n}\Betr{\hat{\varphi}\ast\mathcal{F}\left(\psi P^ku\right)(\xi)}\\
&=\frac{1}{(2\pi)^n}\Betr{\int_{\R^n}\!(1+\Betr{\eta})^{-\sigma}
\hat{\varphi}(\xi-\eta)(1+\Betr{\eta})^{\sigma}
\mathcal{F}\left(\psi P^k u\right)(\eta)\,d\eta\,}\\
&\leq C\norm[H^\sigma(\R^n)]{\psi P^k u}
\left(\int_{\R^n}\!(1+\Betr{\eta})^{-2\sigma}\Betr{\hat{\varphi}(\xi-\eta)}^2\,
d\eta\right)^{\tfrac{1}{2}}\\
&\leq C\norm[H^\sigma(W_1)]{P^k u}(1+\Betr{\xi})^{-\sigma}
\left(\int_{\R^n}(1+\Betr{\xi-\eta})^{2\lvert \sigma\rvert}
\Betr{\hat{\varphi}(\xi-\eta)}^2\,
d\eta\right)^{\tfrac{1}{2}}\\
&\leq Ch^kM_{dk}(1+\Betr{\xi})^{-\sigma}
\norm[H^{\lvert \sigma\rvert}(\R^n)]{\varphi}\\
&\leq Ch^kM_{dk}(1+\Betr{\xi})^{\nu}
\end{split}
\end{equation*}
for some $\bM\in\fM$ and some constants $h>0$ and $\nu=-\sigma$.

On the other hand assume that there is a sequence $f_k\in\Ep(U)$ and
a neighborhood $V$ of $x_0$  such that
 $f_k\vert_V=P^ku\vert_V$ and \eqref{microlocalEst1} holds for
 some $\bM\in\fM$ and constants $C,h,\nu>0$.
 Now let $\sigma\leq -\nu-(n+1)/2$. Then we obtain for every  $W\subseteq V$ that
\begin{equation*}
\begin{split}
\norm[H^\sigma(W)]{P^ku}&\leq \norm[H^\sigma(\R^n)]{f_k}\\
&=\left(\int_{\R^n}\!(1+\Betr{\xi})^{2\sigma}\Betr{\hat{f}_k(\xi)}^2\,d\xi
\right)^{\tfrac{1}{2}}\\
&\leq Ch^kM_{dk}
\left(\int_{\R^n}\!(1+\Betr{\xi})^{2(\sigma+\nu)}\,
d\xi\right)^{\tfrac{1}{2}}\\
&\leq C^\prime h^k M_{dk}
\end{split}
\end{equation*}
for some $C^\prime>0$ since $\sigma$ was chosen appropriately.

The Beurling case follows in a similar manner.
\end{proof}
In the definition of the wavefront set of iterates the estimate \eqref{microlocalEst1} will correspond to 
\eqref{WF-defining}.
The following statement is going to provide a correspondence of the boundedness
of the sequence $u_k$ in Definition \ref{WF-Definition}.

\begin{Prop}[{\cite[Proposition 1.6]{MR632764}}]\label{BoundedErsatz}
	Let $u\in\Dp(U)$, $P$ be an analytic partial differential operator
	 of order $d$ and $K\subseteq U$ be a compact set.
	Furthermore assume that $\chi_k\in\D(U)$ is a sequence of functions with
	common support in $K$ satisfying
	\begin{equation*}
	\Betr{D^\alpha\chi_k(x)}\leq C(Ck)^{\alp}
	\end{equation*}
	for $\alp\leq k\in\N_0$ and some constant $C>0$.
	
	If $p\in\N$ and $q\in\N_0$ then the sequence $f_k=\chi_{pdk+q}u$
	obeys the estimate
	\begin{equation*}
	\Betr{\hat{f}_k(\xi)}\leq C^\prime(C^\prime(dk+\Betr{\xi}))^{dk+\nu}\qquad \xi\in\R^n,\quad k\in\N_0,
	\end{equation*}
	for some constants $C^\prime,\nu>0$.
\end{Prop}
\begin{Def}\label{vWFDef}
	Let $P$ be a differential operator with analytic coefficients
	of order $d$,
	$\fM$ a weight matrix, $u\in\Dp(U)$ and $(x_0,\xi_0)\in\COT$.
	Then we say that
	\begin{enumerate}
		\item $(x_0,\xi_0)\notin\WF_{\{\fM\}}(u;P)$ if
		there is a neighborhood $V$ of $x_0$, a conic neighborhood
		$\Gamma$ of $\xi_0$ and a sequence $f_k\in\Ep(U)$
		satisfying $f_k\vert_V=(P^ku)\vert_V$ and there are
		a sequence $\bM\in\fM$ and constants $C,h>0$ and $\nu\in\R$
		 such that
		\begin{align}
		\left\lvert\hat{f}_k(\xi)\right\rvert\label{vWF-Def1}
		&\leq Ch^k\left[\bigl(M_{dk}\bigr)^{\tfrac{1}{dk}}
		+\xit\right]^{\nu+dk}
		& \forall k&\in\N,\;\forall \xi\in\R^n\\
		\left\lvert\hat{f}_k(\xi)\right\rvert
		&\leq Ch^kM_{dk}(1+\xit)^\nu
		& \forall k&\in\N,\;\forall \xi\in\Gamma.\label{vWF-Def2}
		\end{align}
		
		\item $(x_0,\xi_0)\notin\WF_{(\fM)}(u;P)$ if
		there is a neighborhood $V$ of $x_0$, a conic neighborhood
		$\Gamma$ of $\xi_0$ and a sequence $f_k\in\Ep(U)$
		with $f_k\vert_V=(P^ku)\vert_V$ and there exists some $\nu\in\R$
		such that for all $\bM\in\fM$
		and all $h>0$ there is a constant $C>0$ for which the estimates
		\eqref{vWF-Def1} and \eqref{vWF-Def2} are satisfied.
	\end{enumerate}
\end{Def}
It is easy to see that $\WF_{[\fM]}(u;P)$ satisfies the same basic properties as 
$\WF_{[\fM]} u$, cf.\ Proposition \ref{WFBasicProperties}:
\begin{Prop}\label{WFIteratesProperties}
	Let $\fM$ and $\fN$ be two weight matrices and $u\in\Dp(U)$. Then:
	\begin{enumerate}
		\item $\WF_{[\fM]}(u;P)$ is a closed, conic in the second variable, subset of $\COT$.
		\item $\WF_{\{\fM\}}(u;P)\subseteq\WF_{(\fM)}(u;P)$.
		\item If $\fM[\preceq]\fN$ then $\WF_{[\fN]}(u;P)\subseteq\WF_{[\fM]}(u;P)$ 
		for all $u\in\Dp(U)$.
		\item If $\fM\{\lhd)\fN$ then $\WF_{(\fN)}(u;P)\subseteq\WF_{\{\fM\}}(u;P)$.
	\end{enumerate}
\end{Prop}
We have also a variant of Lemma \ref{WFalternative}:
\begin{Lem}\label{vWFCharLem}
Let $\fM$ be $[$semiregular$]$, $u\in\Dp(U)$, and $K\subseteq U$ be a compact
subset, $F\subseteq\R^n$ a closed cone and $\chi_k\in\D(U)$  a sequence of functions
with support in $K$ such that for all $\alpha\in\N_0^n$ there are constants $C_\alpha,h_\alpha>0$ with
\begin{equation}\label{Testfunkt}
\Betr{D^{\alpha+\beta}\chi_k}\leq C_\alpha(h_\alpha k)^{\bet}\qquad \bet\leq k\in\N.
\end{equation}
\begin{enumerate}
	\item If $\WF_{\{\fM\}} (u;P)\cap K\times F=\emptyset$ then there are 
	a sequence $\bM\in\fM$ and
	constants $C,h>0$ and $\nu\in\R$ such that the sequence $\chi_{dk}P^ku$
	satisfies $\eqref{vWF-Def2}$ for $\xi\in F$.
	\item If $\WF_{(\fM)} (u;P)\cap K\times F=\emptyset$ then there is
	some $\nu\in\R$ such that for all $h>0$ and all $\bM\in\fM$
	the estimate \eqref{vWF-Def2} holds for the sequence 
	$\chi_{dk}P^ku$ in $F$.
\end{enumerate}
\end{Lem}
\begin{proof}
First we prove the Roumieu case.
Let $x_0\in K$ and $\xi_0\in F$. 
Then $(x_0,\xi_0)\notin\WF_{\{\fM\}}(u;P)$ and we choose $V$, $\Gamma$ and 
$f_k$ according to Definition \ref{vWFDef}.
If $\supp \chi_{dk}\subseteq V$ then $\chi_{dk}P^ku=\chi_{dk}f_k$
and therefore
\begin{equation*}
(2\pi)^n\mathcal{F}\left(\chi_{dk}P^ku\right)(\xi)
=\int\!\hat{\chi}_{dk}(\xi-\eta)\hat{f}_k(\eta)\,d\eta.
\end{equation*}
Note that without loss of generality we can always assume $\nu\geq 0$.
We observe that \eqref{Testfunkt} gives
\begin{equation*}
\Betr{\eta^{\alpha+\beta}\hat{\chi}_k(\eta)}\leq C_\alpha h_\alpha^\bet
k^{\bet},\qquad \alpha,\beta\in\N_0^n,\;\bet\leq k\in\N,
\end{equation*}
for some $C_\alpha,h_\alpha>0$.
It follows that there are constants $C,h>0$ such that
\begin{equation}\label{Estimate1}
\Betr{\hat{\chi}_k(\eta)}\leq Ch^k(1+\etat)^{-n-1-\nu}.
\end{equation}
For $\ell,j\geq 0$ we have, (cf.\ \cite[p.\ 26]{FURDOS2020123451})
\begin{equation*}
\etat^{\ell +j}\leq \sum_{\lvert\gamma\rvert=\ell+j}\binom{\ell +j}{\gamma}
\bigl\lvert\eta^\gamma\bigr\rvert.
\end{equation*}
If $j\leq k$ then
\begin{equation*}
\begin{split}
\etat^{\ell +j}\Betr{\hat{\chi}_k(\eta)}&\leq 
\sum_{\lvert\gamma\rvert=\ell+j}\binom{\ell+j}{\gamma}
\Betr{\eta^\gamma\hat{\chi}_{k}(\eta)}\\
&\leq n^{\ell j}\sum_{\substack{\alp\leq \ell\\ \bet=j}}
\Betr{\eta^{\alpha+\beta}\hat{\chi}_{k}(\eta)}\\
&\leq Ch^jk^j
\end{split}
\end{equation*}
for some $C,h>0$.
For $\bM\in\fM$ we obtain
\begin{equation*}
\begin{split}
\left(\bigl(M_k\bigr)^{1/k}+\etat\right)^{k+\nu+n+1}\Betr{\hat{\chi}_k(\eta)}
&=\sum_{\ell=0}^{\nu +n+1}\sum_{j=0}^{k}\binom{\nu+n+1}{\ell}
\binom{k}{j}\\
&\qquad\quad\times\bigl(M_k\bigr)^{(k+\nu+n+1-j-\ell)/k}
\etat^{j+\ell}\Betr{\hat{\chi}_k(\eta)}\\
&\leq Ch^k \bigl(M_k\bigr)^{\tfrac{k+n+1+\nu}{k}}.
\end{split}
\end{equation*}
Since $\fM$ is $R$-semiregular it follows from Remark \ref{DerivClosedAlternative}(2)
that for each $\bM\in\fM$ there are $\bN\in\fM$, $C,h>0$ such that
\begin{equation}\label{Estimate2}
\Betr{\hat{\chi}_k(\eta)}\leq Ch^k N_k\left(\bigl(M_k\bigr)^{1/k}+\etat\right)^{-k-\nu-n-1}.
\end{equation}
The estimate \eqref{Estimate1} implies
\begin{equation*}
\begin{split}
\int_{\Gamma}\! \Betr{\hat{\chi}_{dk}(\xi-\eta)}\bigl\lvert\hat{f}_k(\eta)\bigr\rvert\,d\eta
&\leq Ch^kM_{dk}\int_{\R^n}\!(1+\Betr{\xi-\eta})^{-n-1-\nu}(1+\Betr{\eta})^{\nu}\,d\eta\\
&\leq Ch^kM_{dk}(1+\lvert\xi\rvert)^{\nu}.
\end{split}
\end{equation*}

On the other hand choose a closed cone $\Gamma_1\subseteq\Gamma\cup\{0\}$
with $\xi_0\in\Gamma_1$. Then there is a constant $c>0$ such that
\begin{equation*}
\lvert\xi-\eta\rvert\geq c(\xit +\lvert\eta\rvert)
\end{equation*}
for all $\xi\in\Gamma_1$
and  $\eta\notin\Gamma$.
If we also use \eqref{Estimate2} and set $\tilde{c}=\min\{1,c\}$
then it follows that
for each $\bM\in\fM$ there is some $\bN\in\fM$ such that
\begin{equation*}
\begin{split}
\int\limits_{\R^n\setminus\Gamma}\!\!\negthickspace
\Betr{\hat{\chi}_{dk}(\xi-\eta)}\Betr{\hat{f}_k(\eta)}\,d\eta
&\leq Ch^kN_{dk}\!\!\negthickspace
\int\limits_{\lvert\xi-\eta\rvert\geq c(\xit +\lvert\eta\rvert)}\negthickspace\negthickspace\negthickspace\negthickspace
\negthickspace
\left[\bigl(M_{dk}\bigr)^{\tfrac{1}{dk}}+\lvert\xi-\eta\rvert\right]^{-dk-\nu-n-1}\\
&\qquad\qquad\qquad\qquad\quad\times
\left[\bigl(M_{dk}\bigr)^{\tfrac{1}{dk}}+\etat\right]^{dk+\nu}d\eta\\
&\leq Ch^kN_{dk}\int_{\R^n}\negmedspace
\left[\bigl(M_{dk}\bigr)^{\tfrac{1}{dk}}+\tilde{c}\etat\right]^{-n-1}\,d\eta\\
&\leq Ch^kN_{dk}.
\end{split}
\end{equation*}

We have shown that if $\supp\chi_k\subseteq U$ and $\xi_0\in F\setminus\{0\}$
there is a closed conic neighborhood $\Gamma^\prime$ of $\xi_0$ such that
\begin{equation}\label{WFCharacteq}
\Betr{\F\left(\chi_{dk}P^ku\right)(\xi)}\leq C_0h_0^kM_{dk}(1+\xit)^{\nu_0},\quad \xi\in\Gamma^\prime,
\end{equation}
for some $C_0,h_0>0$, $\bM\in\fM$ and $\nu_0\in\R$.
Since $\xi_0\in F\setminus\{0\}$ was chosen arbitrarily, note that 
$F$ can be covered by a finite number of cones like $\Gamma^\prime$
and therefore \eqref{WFCharacteq} holds in $F$ for some constants $C,h$ and $\nu_0$ as long as $\supp\chi_k\subseteq U$ is a small enough neighborhood of $x_0$.
But $K$ is compact hence we can argue as in the proof of \cite[Lemma 8.4.4]{MR1996773}.
There is a finite number of such open sets $U_j$ that cover $K$
and we can choose a partition of unity $\chi_{j,k}\in\D(U_j)$ 
such that $(\chi_{j,k})_k$ satisfies \eqref{Testfunkt} for each $j$.
Then the same is true for $\chi_{j,k}\chi_k$ and we conclude from above
that \eqref{WFCharacteq} holds for $\chi_{j,dk}\chi_{dk}P^ku$.
Since $\sum_j \chi_{j,dk}\chi_{dk}P^ku=\chi_{dk}P^ku$
we have proven \eqref{WFCharacteq}  in the general case.

The proof of the estimate in the Beurling category is analogous.
Just note that if $\fM$ is $B$-semiregular then Remark \ref{DerivClosedAlternative}(2)
implies that for all $\bN\in\fM$ there
are $\bM\in\fM$, $C,h>0$ such that \eqref{Estimate2} holds.
\end{proof}
Lemma \ref{vWFCharLem} allows us to prove an analogue of Proposition \ref{Singsupport}:
\begin{Thm}
	If $\fM$ is $[$semiregular$]$ and $u\in\Dp(U)$ then
	$U_0=U\setminus\pi_1(\WF_{[\fM]}(u;P))$ is the greatest
	open set such that $u\in\E_{loc}^{[\fM]}(U_0;P)$.
\end{Thm}
\begin{proof}
Let $U_1\subseteq U$ be an open set such that $u\in\E_{loc}^{[\fM]}(U_1;P)$. If $x\in U_1$
then by Theorem \ref{FourierCharThm} (and Proposition \ref{BoundedErsatz})
it follows that $(x,\xi)\notin\WF_{[\fM]}(u;P)$ for all $(x,\xi)\in U_1\times\R^n\!\setminus\!\{0\}$.

On the other hand if $x\in U$ is such that $(x,\xi)\notin\WF_{[\fM]}(u;P)$ for all $\xi\in\R^n\setminus\{0\}$
then we can find a compact neighborhood $K$ of $x$ such that
$K\times\R^n\cap\WF_{[\fM]}(u;P)=\emptyset$.
If we choose  functions $\chi_k\in\D(K)$ satisfying \eqref{Testfunkt}
which equal $1$ in some neighborhood $V$ of $x$, which is possible due to 
\cite[Theorem 1.4.2]{MR1996773},
then Lemma \ref{vWFCharLem} implies that $f_k=\chi_{dk+q}P^ku$ satisfies
\eqref{microlocalEst1}. Hence by Theorem \ref{FourierCharThm} 
$u\in\vDC{\fM}{V}$.
\end{proof}
\subsection{Invariance under analytic mappings}
The aim of this section is to prove the invariance of the definition of 
$\WF_{[\fM]}(u;P)$.
We begin by recalling two results from \cite{MR0294849}, see also \cite{MR1037999}.
\begin{Lem}[{\cite[Lemma 3.6]{MR0294849}}]\label{InvarianceAuxLem1}
Let $U_1\subseteq\R^{n_1}$ and $U_2\subseteq\R^{n_2}$ be two open sets,
$a\in\An(U_1)$ and $f: U_1\rightarrow U_2$ be an analytic mapping.
Furthermore assume that $\chi_k\in\D(U_2)$ is a sequence of functions
with  support in the same fixed compact set and there are constants $C,h>0$ such that 
\begin{equation*}
\Betr{D^\alpha\chi_k(x)}\leq C(hk)^{\alp}, \qquad \alp\leq k.
\end{equation*}
Then the sequence $\chi_k^\prime=a(\chi_k\circ f)$ has the same properties
 with different constants $C,h$.
\end{Lem}
\begin{Lem}[{\cite[Lemma 3.7]{MR0294849}}]\label{InvarianceAuxLem2}
	Let $\mathsf{F}$ be a compact family of analytic real-valued functions 
	on $U$ which do not have a critical point in $x_0\in U$. 
	Further suppose $\chi_k\in\D(U)$ is a sequence of functions with support in 
	the same small enough neighborhood of $x_0$ which satisfies
	\begin{equation*}
	\Betr{D^\alpha\chi_k(x)}\leq C(hk)^{\alp}, \qquad \alp\leq k,
	\end{equation*}
	for some constants $C,h>0$.
	
	Then there exist constants $C^\prime,h^\prime>0$ such that for all
	$t\in\R$ and $f\in\mathsf{F}$ we have
	\begin{equation*}
	\Betr{\int\!\chi_k(x)e^{-itf(x)}dx}\leq C^\prime \left(h^\prime k\right)^k
	\left(k+\Betr{t}\right)^{-k},\qquad k\in\N.
	\end{equation*}
\end{Lem}
\begin{Thm}\label{InvarianceAuxThm}
Let $x_0\in U$, $u\in\Dp(U)$, $P$ be a differential operator of order $d$
 with analytic coefficients in $U$ and $\mathsf{F}$ be a compact family of
 analytic real-valued functions.
 Assume also that
 $\chi_k\in\D(U)$ is a sequence of functions satisfying
 \begin{equation*}
 \Betr{D^\alpha \chi_k}\leq Ch^\alp k^\alp,\quad \alp\leq k,
 \end{equation*}
 with supports inside of the same small enough neighborhood $W$ of $x_0$.
 Then the following holds:
 \begin{enumerate}
\item If $\fM$ is an $R$-semiregular weight matrix and 
$(x_0,df(x_0))\notin\WF_{\{\fM \}}(u;P)\cup\{0\}$ for all $f\in\mathsf{F}$
then there are a sequence $\bM\in\fM$, constants $C,h>0$, $\nu^\prime\in\R$ and $q\in\N_0$ such that
\begin{equation}\label{InvarianceAux}
\Betr{\left\langle \chi_{dk+q}P^k u,e^{-itf}\right\rangle}
\leq C h^{k}M_{dk}t^{\nu^\prime},\qquad k\in\N_0,\; t\geq 1.
\end{equation}
\item If $\fM$ is $B$-semiregular and 
$(x_0,df(x_0))\notin\WF_{(\fM )}(u;P)\cup \{0\}$ for all $f\in\mathsf{F}$
then there are $\nu^\prime$ and $q\in\N_0$ such that for all $\bM\in\fM$ and
$h>0$ there is some $C>0$ satisfying \eqref{InvarianceAux}.
 \end{enumerate}
\end{Thm}
\begin{proof}
	Note first that the set $F=\{tdf(x_0):\; t>0,\, f\in\mathsf{F}\}$ is a closed cone in 
	$\R^n\!\setminus\!\{0\}$.
	Since by Proposition \ref{WFIteratesProperties}(1) $\WF_{[\fM]}(u;P)$ is a closed subset of $\COT$ which is conic in the second variable, there has to be a neighborhood $V$
	of $x_0$ and an open conic neighborhood $\Gamma\subseteq\R^n\!\setminus\!\{0\}$ of 
	$F$ such that $\WF_{[\fM]}u\cap \overline{V}\times\overline{\Gamma}=\emptyset$.
	Then Lemma \ref{vWFCharLem} implies that 
	we can find 
	a sequence $f_k\in\Ep(U)$
	and $\nu\in\R$ such that 
	the following holds.
	First, $f_k\vert_V=(P^ku)\vert_V$ and the Fourier transforms of the  $f_k$ either 
	satisfy
\begin{align}\label{b-EstWF1}
\left\lvert\hat{f}_k(\xi)\right\rvert
&\leq Ch^k\left[\bigl(M_{dk}\bigr)^{\tfrac{1}{dk}}
+\xit\right]^{\nu+dk}
& \forall k&\in\N,\;\forall \xi\in\R^n,\\
\left\lvert\hat{f}_k(\xi)\right\rvert
&\leq Ch^kM_{dk}(1+\xit)^\nu
& \forall k&\in\N,\;\forall \xi\in\Gamma\label{b-EstWF2}
\end{align}
 in the Roumieu case, for some constants $C,h$ and $\bM\in\fM$
 or, in the Beurling case, for all $\bM\in\fM$ and $h>0$ 
 there is some $C>0$ such that 
 \eqref{b-EstWF1} and \eqref{b-EstWF2} hold.

We assume for the moment that \eqref{b-EstWF1} and \eqref{b-EstWF2} holds for 
some fixed $\bM\in\fM$ and some constants $C,h>0$.
We can further suppose that $\supp\chi_k\subseteq W= V$. Moreover, we set 
$v_{k,t}=\chi_{dk+q}e^{-itf}$ for some fixed integer $q\geq n+1+\nu$.
We conclude that
\begin{equation}\label{Parseval}
\left\langle \chi_{dk+q}P^k u,e^{-itf}\right\rangle
=\frac{1}{(2\pi)^n}\int\!\hat{f}_k(\xi)\check{v}_{k,t}(-\xi)\,d\xi
\end{equation}
where 
\begin{equation*}
\check{v}_{k,t}(-\xi)=\int\!\chi_{dk+q}e^{i(x\xi-tf(x))}\,dx.
\end{equation*}

The normalized functions
\begin{equation*}
x\longmapsto \frac{x\xi-tf(x)}{\lvert t\rvert+\xit}
\end{equation*}
with $f\in\mathsf{F}$ and $t>0$ form a compact family of analytic functions
without a critical point in $x_0$ as long as $\xi\notin\Gamma$ or 
$\xi\in\Gamma$ and $\min(\lvert t\rvert/\xit,\xit/\lvert t\rvert)<\eps$
for some sufficiently small $\eps>0$.

If the supports of the $\chi_k$ are sufficiently small around $x_0$ Lemma \ref{InvarianceAuxLem2} allows us to estimate $\check{v}_{k,t}(-\xi)$.
In fact, there exist constant $C^\prime,h^\prime>0$ such that
\begin{equation}\label{b-Est3}
\Betr{\check{v}_{k,t}(-\xi)}\leq C^\prime \left(h^\prime (dk+q)\right)^{dk+q}
\left(dk+q+\lvert t\rvert+\xit\right)^{-dk-q}, \quad k\in\N_0,
\end{equation}
for $f\in\mathsf{F}$, $t>0$, $\xi\notin\Gamma$ or $\xi\in\Gamma$ 
and $\min(\lvert t\rvert/\xit,\xit/\lvert t\rvert)<\eps$.
Note that the right-hand side of \eqref{b-Est3} can be  bounded by 
$C^\prime (h^\prime)^{dk+q}$.
 
Now recall that \eqref{matrixAnal} implies that
 for all $\bM\in\fM$ there is some $\gamma>0$ such that
 \begin{equation*}
 k\leq \gamma \bigl(M_k\bigr)^{\tfrac{1}{k}},\qquad k\in\N.
 \end{equation*}
From this we obtain, with the same constant $\gamma$,
\begin{equation*}
\frac{k}{k+\tau}\leq
\frac{\gamma\bigl(M_k\bigr)^{\tfrac{1}{k}}}{\gamma\bigl(M_k\bigr)^{\tfrac{1}{k}}+\tau}
\end{equation*}
for all $k\in\N$ and all $\tau>0$.
 Hence 
  we obtain from  \eqref{b-EstWF1}, \eqref{b-EstWF2}, \eqref{Parseval} and \eqref{b-Est3}
 the following estimate
 \begin{equation*}
 \begin{split}
 \left\lvert\left\langle \chi_{dk+q}P^k u,e^{-itf}\right\rangle\right\rvert
 &=\frac{1}{(2\pi)^n}\int_{\R^n\setminus\Gamma}\! \hat{f}_k(\xi)
 \check{v}_{k,t}(-\xi)\,d\xi+\frac{1}{(2\pi)^n}\int_{\Gamma}\hat{f}_k(\xi)
 \check{v}_{k,t}(-\xi)\,d\xi\\ 
 &\leq C\left[\int_{\R^n}\! h^k \left[\bigl(M_{dk}\bigr)^{\tfrac{1}{dk}}+\xit\right]^{dk+\nu}
\right.\\
 &\qquad\qquad\times
 (h^\prime)^{dk+q}\frac{\gamma^{dk+q}M_{dk+q}}{\Bigl(\gamma 
 	\bigl(M_{dk+q}\bigr)^{\tfrac{1}{dk+q}}+\lvert t\rvert +\xit \Bigr)^{dk+q}}\,d\xi\\
 &\qquad\qquad\qquad \left.+\int\limits_{\eps t\leq \xit\leq t/\eps}\! \bigl(h^\prime)^{dk+q}h^kM_{dk}
 \bigl(1+\xit\bigr)^{\nu}\,d\xi\right].
 \end{split}
 \end{equation*}
Note that if $0<\gamma\leq 1$ then 
\begin{equation*}
	\gamma\bigl(M_{dk+q}\bigr)^{\tfrac{1}{dk+q}}+\lvert t\rvert+\xit\geq
	\gamma\left(\bigl(M_{dk+q}\bigr)^{\tfrac{1}{dk+q}}+\lvert t\rvert+\xit\right).
\end{equation*}
On the other hand, for $\gamma> 1$ we have the trivial estimate 
$(M_k)^{1/k}\leq \gamma (M_k)^{1/k}$.
Thence, since $t\geq 1$, the first integrand in the right-hand side above can be bounded by
\begin{equation*}
	C_1h_1^kM_{dk+q}t^{-q+\nu}
	\left(1+\frac{\xit}{t}\right)^{\nu-q}
\end{equation*}
with $h_1$ being a multiple of $h$, $h^\prime$ and possibly $\gamma$.

Following iterated application of \eqref{RouDeriv} we can conclude that
there are constants $C,h>0$ and a weight sequence $\bM^\prime\in\fM$ such that
 \begin{equation*}
 \Betr{\left\langle \chi_{dk+q}P^k u,e^{-itf}\right\rangle}\leq
 Ch^k M^\prime_{dk}\bigl(t^{-q+\nu+n}+\left(1+ t/\eps\right)^{\nu}t^n\bigr)
 \end{equation*}
 and we have proven the theorem in the Roumieu case.
 
 It is easy to see that the same proof holds also
 in the Beurling category.
\end{proof}

\begin{Thm}
	If $\fM$ is $[$semiregular$]$ then $\WF_{[\fM]}(u;P)$ 
	is invariant under analytic changes of coordinates.
\end{Thm}
\begin{proof}
	Let $F:U\rightarrow U^\prime$ be an analytic diffeomorphism from $U$ onto 
	an open subset $U^\prime\subseteq\R^n$ which transforms the operator 
	$P$ into the operator $P_F$ defined by
	\begin{equation*}
	P_F \psi = P(\psi\circ F)\circ F^{-1},\qquad \psi\in\D(U).
	\end{equation*}
	Then
	\begin{equation*}
	\bigl(P_F\bigr)^k\psi= P^k(\psi\circ F)\circ F^{-1}
	\end{equation*}
	 for all $k\in\N_0$.
	 We set $y=F(x)$ and $u=v\circ F$.
	We are going to show that, if $(x_0,\xi_0)\notin\WF_{[\fM]}(u;P)$
	then $(y_0,\eta_0)\notin\WF_{[\fM]}(v;P_F)$ where $y_0=F(x_0)$ and
	$\xi_0 =F^\prime(x_0)^T\eta_0$.
	
	Let $\chi_k\in\D(U)$ be a sequence of functions with supports in a small enough neighborhood of $x_0$ and which are equal to $1$ near $x_0$ and
	satisfy $\Betr{D^\alpha\chi_k}\leq C(hk)^{\alp}$ when $\alp\leq k$.
	If $\Gamma$ is the cone associated to $\xi_0$ in Definition \ref{vWFDef}
	then $(F^\prime(x_0)^T)^{-1}\Gamma$ is an open conic neighborhood of $\eta_0$.
	It follows that the family
	\begin{equation*}
	F_\eta: x\longmapsto \frac{1}{1+\etat}\langle F(x),\eta\rangle,
	\qquad \eta\in \left(F^\prime(x_0)^T\right)^{-1}\Gamma
	\end{equation*}
	is a compact set of real-valued analytic functions
	with $(x_0,dF_\eta(x_0))\notin\WF_{[\fM]}(u;P)\cup\{0\}$ since
	$(x_0,\xi_0)\notin\WF_{[\fM]}(u;P)$.
	
	According to Lemma \ref{InvarianceAuxLem1} we have that
	\begin{equation*}
	\Betr{D^\alpha\left(\Betr{F^\prime(x)^T}\chi_k(x)\right)}
	\leq C (hk)^\alp,\qquad \alp\leq k\in\N_0,\; x\in U,
	\end{equation*} 
	for some constants $C,h>0$.
	In the Roumieu case Theorem \ref{InvarianceAuxThm} implies that
	there are constants $C,h>0,\nu^\prime\in\R$ and $q\in\N$ such that
	\begin{equation*}
\Betr{\left\langle\Betr{F^\prime(x)^T}\chi_{dk+q}(x)P^ku,
e^{-iF(x)\eta}\right\rangle}
\leq Ch^kM_{dk}\bigl(1+\etat\bigr)^{\nu^\prime}.
	\end{equation*}
	If we define $\varphi_k=\chi_k\circ F^{-1}$ and 
	$g_k=\varphi_{dk+q}P_F^k v$ then we obtain
	\begin{equation*}
	\Betr{\hat{g}_k(\eta)}\leq Ch^kM_{dk}\bigl(1+\etat\bigr)^{\nu^\prime},
	\qquad k\in\N_0,\; \eta\in \bigl(F^\prime(x_0)^T\bigr)\Gamma. 
	\end{equation*}
	Furthermore, by Lemma \ref{InvarianceAuxLem1} the functions $\varphi_k$
	satisfy
	\begin{equation*}
	\Betr{D^\alpha\varphi_k}\leq C(hk)^\alp,\qquad \alp\leq k\in\N_0,
	\end{equation*}
	for some constants $C,h>0$.
	Hence, by Proposition \ref{BoundedErsatz} the estimate
	 \eqref{vWF-Def2} holds for the sequence $g_k$ too.
	Since $g_k\vert_V=P^k v$ in some neighborhood $V\subseteq U^\prime$
	of $y_0$ we have therefore shown that $(y_0,\eta_0)\notin\WF_{\{\fM\}}(v;P_F)$.
	
	Virtually the same proof gives us also the result in the Beurling case.
\end{proof}

\subsection{The elliptic Theorem of Iterates}
We are now in the position to 
prove the microlocal elliptic Theorem of Iterates for $[\fM]$-vectors.
We want to begin by showing that $\WF_{[\fM]}(u;P)$ is in fact a refinement of
$\WF_{[\fM]}u$, but to this end we need a variant of Lemma \ref{WFalternative}.
\begin{Lem}\label{WFalternativePara}
	Let $K\subseteq U$ be compact, $F\subseteq\R^n\setminus\{0\}$ be 
	a closed cone, $u\in\Dp(U)$, $P$ be an analytic differential operator
	and $\varphi_k(x,\xi)$ be a sequence of smooth functions on $U\times F$
	with $\supp \varphi_k(\,.\,,\xi)\subseteq K$ for all $k\in\N_0$ and 
	$\xi\in F$ for which there are constants $C,h>0$ such that
	\begin{equation}\label{Ehrenpreis}
	\Betr{D^\alpha \varphi_k(x,\xi)}\leq C(hk)^{\alp},\qquad \alp\leq k,\;x\in K,\; \xi\in F, \xit>k,
	\end{equation}
	for all $k\in\N_0$. Furthermore assume that 
	$\fM$ is a $[$semiregular$]$ weight matrix and 
	let $\mu$ be the order of $u$ near $K$. Then the following holds:
	\begin{enumerate}
		\item If $\WF_{[\fM]}u\cap \bigl(K\times F\bigr)=\emptyset$
		then there are $\bM\in\fM$ and constants $C,h>0$ (resp.\ for all
		$\bM\in\fM$ and  $h>0$ there exists some $C>0$) such that 
		\begin{equation*}
		\Betr{\widehat{\varphi_k u}(\xi)}\leq Ch^k M_{k-\mu-n}\xit^{\mu-n-k},
		\qquad \xi\in F,\; \xit>k,\; k\geq \mu+n.
		\end{equation*}
		\item If $\WF_{[\fM]}(u;P)\cap \bigl(K\times F\bigr)=\emptyset$ 
		then there are $\bM\in\fM$ constants $\nu\geq 0$ and $h, C>0$ 
		(resp.\ there is some $\nu\geq0$ such that 
		for all $\bM\in\fM$ and $h>0$ there exists some $C>0$) satisfying
		\begin{equation*}
		\Betr{\F(\varphi_{dk+q}P^ku)(\xi)}\leq Ch^k M_{dk}(1+\xit)^{\nu},
		  \;\,\xi\in F,\, \xit>{dk+q},\, q\geq n+\nu+1.
		\end{equation*}
	\end{enumerate}
\end{Lem}
\begin{proof}
	We begin with the proof of (1) in the Roumieu category. 
	Due to Lemma \ref{WFalternative} there is a bounded sequence
	$u_k\in\Ep(U)$ such that $u_k\vert_W=u\vert_W$ in some neighborhood
	$W$ of $K$ and
	\begin{equation*}
	\Betr{\hat{u}_k(\eta)}\leq Ch^k M_k\etat^{-k},\qquad \eta\in\Gamma
	\end{equation*}
	for some $C,h>0$ and $\bM\in\fM$
	 where $\Gamma$ is an open conic neighborhood of $F$.
	Clearly $\varphi_k u=\varphi_k u_{k^\prime}$, $k^\prime=k-\mu-n$.
	The estimate \eqref{Ehrenpreis} gives us 
	\begin{equation}\label{EhrenpreisFourier}
	\Betr{\hat{\varphi}_k(\eta,\xi)}\leq Ch^k\left(\frac{k}{k+\etat}\right)^k,
	\qquad \eta\in\R^n,\; \xi\in F,\; \xit>k,
	\end{equation}
	where $\hat{\varphi}_k(\eta,\xi)=\int\!e^{-ix\eta}\varphi_k(x,\xi)\,dx$ 
	is the partial Fourier transform of $\varphi_k$.
	Furthermore if $\xi\in F$ we can choose $0<c<1$ such that 
	$\eta\in\Gamma$ when $\lvert \xi-\eta\rvert\leq c\xit$.
	\cite[equation (8.1.3)]{MR1996773} states that 
	\begin{equation*}
		\begin{split}
	(2\pi)^n\Betr{\widehat{\varphi_k u}(\xi)}&\leq \norm[L^1]{\hat{\varphi}_k(\,.\,,\xi)}\sup_{\Betr{\eta-\xi}\leq c\xit}
	\Betr{\hat{u}_{k^\prime}(\eta)}\\
	&\qquad+ C\left(1+c^{-1}\right)^{\mu}\!
	\int\limits_{\etat>c\xit}\negmedspace\Betr{\hat{\varphi}_k(\eta,\xi)}
	(1+\etat)^\mu\,d\eta
	\end{split}
\end{equation*}
	for some $C>0$. 
	We have, if $k>\mu+n+d$,
	\begin{equation*}
	\norm[L^1]{\hat{\varphi}_k(\,.\,,\eta)}\leq Ch^k k^{\mu+n}
	\end{equation*}
	for some $C,h>0$.
	Since $\etat\leq (1-c)\xit$ we conclude that
	\begin{equation*}
	\begin{split}
	\xit^{k^\prime}\Betr{\widehat{\varphi_k u}(\xi)}
	&\leq Ch^k
	\left(k^{\mu+n}(1-c)^{-k^\prime}\sup_{\eta\in\Gamma}
	\Betr{\hat{u}_{k^\prime}(\eta)}\etat^{k^\prime}\right.\\
	&\qquad\left. +\bigl(1+c^{-1}\bigr)^\mu k^k \int_{\etat>c\xit}\! \etat^{\mu-k}\right).\\	
	\end{split}
	\end{equation*}
	Hence there are some $C,h>0$ such that
	\begin{equation*}
	\Betr{\widehat{\varphi_k u}(\xi)}\leq Ch^kM_k\xit^{-k^\prime}
	\end{equation*}
	for $\xi\in F$, $\xit>k$ and $k>\mu+n$.
	
	We now turn to the proof of (2). In the Roumieu case Lemma \ref{vWFCharLem}
	and Proposition \ref{BoundedErsatz} imply that
	there are a neighborhood $W$ of $K$, an open conic neighborhood $\Gamma$ of $F$ and a sequence $\Ep(U)$ such that $f_k=P^ku$ in $W$ and
	\begin{align*}
	\Betr{\hat{f}_k(\xi)}&\leq  Ch^k\left[\bigl(M_{dk}\bigr)^{\tfrac{1}{dk}}+\xit\right]^{\nu+dk} &
	\fa k&\in\N,\; \fa \xi\in\R^n,\\
	\Betr{\hat{f}_k(\xi)}&\leq 
	Ch^kM_{dk}\left(1+\xit\right)^\nu & \fa k&\in\N,\; \fa\xi\in\Gamma,
	\end{align*} 
	for some $\bM\in\fM$ and constants $\nu\in\R$ and $C,h>0$. Similarly to above we have
	\begin{equation*}
	(2\pi)^n\F\left(\varphi_{dk+q} P^ku\right)(\xi)=\int\! \hat{\varphi}_{dk+q}(\xi-\eta,\xi)\hat{f}_k(\eta)\,d\eta.
	\end{equation*}
	Without loss of generality we may assume that $\nu\geq 0$.
	By \eqref{Ehrenpreis} we have that there are constants $C,h>0$ such that
	\begin{equation*}
	\Betr{\hat{\varphi}_k(\eta,\xi)}\leq Ch^k(1+\etat)^{-\nu-n-1}
	\end{equation*}
	for $\xit>k$, $k\geq \nu+n+1$.
	
	It follows that
	\begin{equation*}
	\begin{split}
	\int_{\Gamma}\!\Betr{\hat{\varphi}_{dk+q}(\xi-\eta)}\Betr{\hat{f}_k(\eta)}\,d\eta
	&\leq Ch^kM_{dk}\int_{\R^n}\!(1+\lvert\xi-\eta\rvert)^{-n-\nu-1}(1+\etat)^\nu\,d\eta\\
	&\leq Ch^kM_{dk}(1+\xit)^{n+\nu+1}
	\end{split}
	\end{equation*}
	for $\xi\in F$, $\xit>{dk+q}$.
	
	Moreover, there is a constant $\kappa>0$ such that if $\xi\in F$ and $\eta\notin\Gamma$ then $\lvert\xi-\eta\rvert\geq \kappa(\xit+\etat)$.
	Hence, by \eqref{EhrenpreisFourier} we have 
	\begin{equation*}
	\begin{split}
	\int_{\R^n\setminus\Gamma}\!\Betr{\hat{\varphi}_{dk+q}(\xi-\eta)}\Betr{\hat{f}_k(\eta)}\,d\eta&\leq \int\limits_{\lvert\xi-\eta\vert\geq \kappa(\xit+\etat)}\!\Betr{\hat{\varphi}_{dk+q}(\xi-\eta)}\Betr{\hat{f}_k(\eta)}\,d\eta\\
	&\leq Ch^kM_{dk+q}\\
&\quad	\int\limits_{\lvert\xi-\eta\rvert\geq\kappa(\xit+\etat)}
	\negthickspace \negthickspace\negthickspace\negthickspace\negthickspace\negthickspace
	\left[\bigl(M_{dk+q}\bigr)^{1/(dk+q)}+\xi-\eta\rvert\right]^{-dk-q}\times\\
	&\qquad\qquad\qquad
	\times\left[\bigl(M_{dk}\bigr)^{1/(dk)}\right]^{\nu+dk}\,d\eta.
	\end{split}
	\end{equation*}

Thus there exists some $\bM^\prime\in\fM$ and constants $C,h>0$ such that
\begin{equation*}
\Betr{\F(\varphi_{dk+q}P^ku)(\xi)}\leq Ch^kM^\prime_{dk}(1+\xit)^{\nu +n+1}
\end{equation*}
for $\xi\in F$, $\xit>dk+q$.
\end{proof}
\begin{Thm}\label{Elliptic-ThmPart1}
Let $P$ be a differential operator with analytic coefficients on $U$
and $\fM$ be a $[$semiregular$]$ weight matrix. 
Then 
\begin{equation*}
\WF_{[\fM]}(u;P)\subseteq\WF_{[\fM]}(Pu)\subseteq\WF_{[\fM]} u
\end{equation*}
for $u\in\Dp(U)$.
\end{Thm}
\begin{proof}
It is enough to prove 
\begin{equation*}
\WF_{[\fM]}(u;P)\subseteq\WF_{[\fM]} u.
\end{equation*}
Indeed, the [semiregularity] gives $\WF_{[\fM]}(u;P)=\WF_{[\fM]}(Pu;P)$ 
and $\WF_{[\fM]} Pu\subseteq\WF_{[\fM]} u$ by  Theorem \ref{EllipticWFThm}.

Now assume that $(x_0,\xi_0)\notin\WF_{\{\fM\}}u$.
Then there are a neighborhood $V$ of $x_0$, a conic neighborhood $\Gamma$
of $\xi_0$ and a bounded sequence $u_k\in\Ep(U)$ with $u_k\vert_V=u\vert_V$
such that
\begin{equation*}
\lvert \xi\rvert^k \Betr{\hat{u}_k}\leq Ch^kM_k \qquad\fa \xi\in\Gamma,\fa k\in\N_0
\end{equation*}
for some $\bM\in\fM$ and some constants $C,h>0$.

Let $W\Subset V$ be a  neighborhood of $x_0$ and $F\subseteq\Gamma\cup\{0\}$
a closed conic neighborhood of $\xi_0$. 
 Choose a sequence $\chi_k\in\D(V)$ with $\chi\vert_W = 1$ and
$\lvert D^\alpha\chi_k(x)\rvert\leq Ch^\alp k^\alp$ for $\alp\leq k$.
We set $f_k=\chi_{2dk}P^k u$. 
It follows that
\begin{equation*}
\hat{f}_k(\xi)=\left\langle \chi_{2dk}P^ku,e^{-ix\xi}\right\rangle\\
=\left\langle u,Q^k\left(e^{-ix\xi}\chi_{2dk}\right)\right\rangle
\end{equation*}
where $Q$ denotes the formal adjoint of $P$ given by
$\langle Q\phi,\psi\rangle=\langle\phi,P\psi\rangle$ with $\phi,\psi\in\D$.
Hence if $P=\sum_{\alp\leq d}p_\alpha(x)D^\alpha$ then 
$Qg=\sum_{\alp\leq d}(-D)^\alpha(p_\alpha g)=\sum_{\alp\leq d} q_\alpha D^\alpha g$.
We define a new differential operator $R$ by setting
\begin{equation*}
Q\left(e^{-ix\xi}\chi_{2dk}\right)=e^{-ix\xi}\xit^{dk}R\chi_{2dk}.
\end{equation*}
It follows that $R=R_1+\dots +R_d$, where $R_j=R_j(x,\xi,D)$ is a
differential operator of order $\leq j$ with analytic coefficients which
are homogeneous of degree $-j$ with respect to $\xi$.
More precisely,
\begin{equation*}
R_j(x,\xi,D)=\sum_{\alp\leq d}\sum_{\substack{\beta\leq\alpha\\ \bet=d-j}}
	\binom{\alpha}{\beta}q_\alpha(x)\frac{\xi^\beta}{\xit^d}D^{\alpha-\beta}.
\end{equation*}
It follows that
\begin{equation*}
Q^k\left(e^{-ix\xi}\chi_{2dk}\right)=e^{-ix\xi}\xit^{dk}R^k\chi_{2dk}
=e^{-ix\xi}\xit^{dk}\sum_{\substack{0\leq j_\ell\leq d\\1\leq\ell\leq k}}
R_{j_1}\dots R_{j_k}\chi_{2dk}.
\end{equation*}
By \cite[Lemma 5.2]{MR0294849} we have, for $\bet+j\leq 2dk$ and
 $j=j_1+\dots+j_k$,
\begin{equation*}
\Betr{D^\beta R_{j_1}\dots R_{j_k}\chi_{2dk}}
\leq Ch^{k}k^{\bet +j}\xit^{-j}
\end{equation*}
for some constants $C,h>0$. Hence if $\xit\geq dk$ then
\begin{equation*}
\Betr{D^\beta R_{j_1}\dots R_{j_k}\chi_{2dk}}
\leq Ch^{k}k^{\bet}.
\end{equation*}
We conclude that 
\begin{equation*}
\Betr{D^\beta R^k\chi_{2dk}}\leq Ch^k (dk)^\bet
\end{equation*}
when $\xit\geq dk$ and $\bet\leq dk$.

Lemma \ref{WFalternativePara}(1) gives that there is some $\bM\in\fM$ such that
\begin{equation*}
\Betr{\hat{f}_k(\xi)}=\xit^{dk}\Betr{\F\left(u R^k\chi_{2dk}\right)(\xi)}\leq Ch^kM_{dk-d-\mu}
\end{equation*}
for $\xi\in F$, $\xit>dk$ where $\mu$ is the order of $u$ near $\overline{W}$.
If we set $g_k=f_{k+d+\mu}$ then since $\fM$ is $R$-semiregular we obtain that
 there is some $\bM^\prime\in\fM$ such that
\begin{equation*}
\Betr{\hat{g}_k(\xi)}\leq Ch^kM_{dk+(d-1)(d+\mu)}\leq Ch^kM^\prime_{dk}
\end{equation*}
for $\xi\in F$, $\xit>d(k+d+\mu)$.

Proposition \ref{BoundedErsatz} implies that there is some $\nu$ such that
\begin{equation}\label{vWFDef-1beta}
\begin{split}
\Betr{\hat{g}_k(\xi)}&=\Betr{\F(\chi_{2dk+2d^2+2d\mu}u)(\xi)}\leq Ch^k(dk+\xit)^{dk+\nu}\\
&\leq Ch^k\left[\bigl(M_{dk}\bigr)^{\tfrac{1}{dk}}+\xit\right]^{dk+\nu},
\qquad\quad \xi\in\R^n,\;k\in\N\
\end{split}
\end{equation}
for any $\bM\in\fM$ and hence $(g_k)_k$ satisfies \eqref{vWF-Def1}.

On the other hand, if $\xit<dk$ then by \eqref{vWFDef-1beta} we obtain
\begin{equation*}
\begin{split}
\Betr{\hat{g}_k(\xi)}&\leq Ch^k\left[\bigl(M_{dk}\bigr)^{\tfrac{1}{dk}}+dk\right]^{dk+\mu}\\
&\leq Ch^k \bigl(M_{dk}\bigr)^{\tfrac{dk+\mu}{dk}}\\
&\leq Ch^k M^{\prime\prime}_{dk}
\end{split}
\end{equation*}
for some $\bM^{\prime\prime}\in\fM$.
Hence if we choose $\bM^{\prime\prime\prime}=\max\{\bM^\prime,\bM^{\prime\prime}\}$ then
$f_k$ satisfies \eqref{vWF-Def1} and \eqref{vWF-Def2} for $\bM^{\prime\prime\prime}$.
It follows that $(x_0,\xi_0)\notin\WF_{\{\fM\}}(u;P)$.

A close inspection of the proof in the Roumieu case reveals that 
a few obvious modificiations allow us also to show that $(x_0,\xi_0)\notin\WF_{(\bM)}u$
implies $(x_0,\xi_0)\notin\WF_{(\bM)}(u;P)$.
\end{proof}

\begin{Thm}\label{VectorEllipticThm}
Let $P$ be a differential operator with analytic coefficients on $U$
and $\fM$ be a $[$semiregular$]$ weight matrix. 
Then 
\begin{equation*}
\WF_{[\fM]}u\subseteq\WF_{[\fM]}(u;P)\cup\Char P
\end{equation*}
for $u\in\Dp(U)$.
\end{Thm}

\begin{proof}
As in \cite{MR557524} for the Denjoy-Carleman case the proof follows
closely the pattern used in \cite{MR0294849} to show the elliptic regularity theorem,
see also \cite{MR2595651} and \cite{FURDOS2020123451}.

Let $(x_0,\xi_0)\in \COT$ be such that $(x_0,\xi_0)\notin\WF_{[\fM]}(U;P)$ and
$p_d(x_0,\xi_0)\neq 0$.
Thence there exist a conic neighborhood $V\times\Gamma$ of $(x_0,\xi_0)$ and 
a sequence $f_k\in\E^\prime(U)$ with $f_k\vert_V=P^ku\vert_V$ which satisfies
\eqref{vWF-Def1} and \eqref{vWF-Def2}. Furthermore there are a compact neighborhood
$K$ of $x_0$ and a conic neighborhood $F$ of $\xi_0$, closed in $\R^n\setminus\{0\}$,
such that $p_d(x,\xi)\neq 0$ for $(x,\xi)\in K\times F$.
W.l.o.g.\ we can assume that $K\times F\subseteq V\times\Gamma$.
Suppose that $\chi_k\in\D(K)$ is a sequence with
\begin{equation*}
\Betr{D^\alpha\chi_k}\leq C(hk)^\alp,\qquad \alp\leq k,
\end{equation*}
for some constants $C,h$ independent of $k$.

We set $u_k=\chi_{3d^2k}u$ and thus have
\begin{equation*}
\hat{u}_k(\xi)=\left\langle u,\chi_{3d^2k} e^{-ix\xi} \right\rangle.
\end{equation*}
If $Q$ is the adjoint of $P$ then we want to construct a solution $v$ of the equation
\begin{equation}\label{characteristic1}
Q^k(x,D)v(x)=\chi_{3d^2k}e^{-ix\xi}.
\end{equation}
We define a differential operator $R=R(x,\xi,D)$ on $K\times F$ by
\begin{equation*}
Q\left(\frac{e^{-ix\xi}g}{p_d(x,\xi)}\right)=e^{-ix\xi}(I-R)g.
\end{equation*}
Then $R=R_1+\dots + R_d$ where $R_j=R_j(x,D)$ is a differential operator of order 
$\leq j$ with analytic coefficients in $x$ which are homogeneous of degree $-j$ in 
$\xi$. By recurrence we obtain for $k\in\N$ that
\begin{equation*}
Q^k\left(\frac{e^{-ix\xi}g}{p^k_d(x,\xi)}\right)=e^{-ix\xi}
\left((I-R)p_d(x,\xi)\right)^k\bigl(p_d^{-k}g\bigr).
\end{equation*}
If we set in \eqref{characteristic1} 
\begin{equation*}
v=e^{-ix\xi}\frac{w}{p_d^k(x,\xi)}
\end{equation*}
then $w$ satisfies the equation
\begin{equation}\label{characteristic2}
\left((I-R)p_d\right)^k\frac{w}{p_d^k}=\chi_{3d^2k}(x).
\end{equation}
A formal solution of the above equation would be
\begin{equation*}
w=p_d^k(x,\xi)\left[\frac{1}{p_d(x,\xi)}\sum_{\ell=0}^\infty R^\ell \right]^k
\chi_{3d^2k}.
\end{equation*}
However, we cannot estimate arbitrary high derivatives of $\chi_{3d^2k}$, hence
we consider the following approximate solution of \eqref{characteristic2}
\begin{equation*}
w_k=p_d^k\sum_{\ell_1+\dots+\ell_k\leq dk}p_d^{-1}R^{\ell_1}\dots p_d^{-1}R^{\ell_k}
\chi_{3d^2k}.
\end{equation*}
Then we obtain
\begin{equation}\label{characteristic3}
\left((I-R)p_d\right)^k\frac{w_k}{p_d^k}=\chi_{3d^2k}-e_k
\end{equation}
where
\begin{equation*}
e_k=\sum_{j=1}^k \left((I-R)p_d\right)^{k-j}\sum_{\ell_j+\dots +\ell_k=dk}
R^{\ell_j+1}p_d^{-1}R^{\ell_{j+1}}\dots p_d^{-1}R^{\ell_k}\chi_{3d^2k}.
\end{equation*}
Inserting \eqref{characteristic3} in \eqref{characteristic1} gives
\begin{equation*}
Q^k\left(e^{-ix\xi}\frac{w_k}{p_d^k}\right)=e^{-ix\xi}\left(\chi_{3d^2k}-e_k
\right).
\end{equation*}
Hence we obtain the following representation for $\hat{u}_k$, i.e.\
\begin{equation}\label{EllipticRepr}
\hat{u}_k(\xi)=\left\langle u,e_k(x,\xi)e^{-ix\xi}\right\rangle
+\left\langle f_k,e^{-ix\xi}\frac{w_k(x,\xi)}{p_d^k(x,\xi)}\right\rangle
\end{equation}
for $\xi\in F$. 

Since $p_d^{-1}$ is real analytic in a neighborhood of $K$ and homogeneous of degree
$-d$ in $\xi\in F$ we can apply the proof of \cite[Lemma 5.2]{MR0294849}
in order to obtain that there are constants $C,h>0$ such that
\begin{align}\label{vElliptic1}
\Betr{D^{\beta}w_k(x,\xi)}&\leq Ch^k\left(dk\right)^\bet
\\
\Betr{D^{\beta}e_k(x,\xi)}&\leq Ch^k\left(dk\right)^{\bet+dk}\xit^{-dk}
\label{vElliptic2}
\end{align}
for $\bet\leq dk$, $\xit\geq dk$, $\xi\in F$ and $x\in K$.

If $\tau$ is the order of $u$ near $K$ then we can estimate the first term
on the right-hand side of \eqref{EllipticRepr} by
\begin{equation*}
\begin{split}
\Betr{\left\langle u,e_k(x,\xi)e^{-ix\xi}\right\rangle}
&\leq
C\sum_{\alp\leq\tau}\sup_{x\in K}\Betr{D^\alpha\left(
	e_k(x,\xi)e^{-ix\xi}\right)}\\
&\leq C\sum_{\alp\leq \tau}\xit^{\tau-\alp}\Betr{D^\alpha e_k(x,\xi)}
\end{split}
\end{equation*}
for $\xi\in F$, $\xit>1$. If $k\geq \tau/d$ then \eqref{vElliptic2} gives that
\begin{equation*}
\Betr{\left\langle u,e_k(x,\xi)e^{-ix\xi}\right\rangle}\leq Ch^k(dk)^{\tau +dk}
\xit^{\tau-dk}
\end{equation*}
for $\xi\in F$, $\xit>dk$.

Since $w_k$ satisfies \eqref{vElliptic1} we obtain that
\begin{equation*}
\Betr{D^\beta\left(\frac{w_k(x,\xi)}{p_d^k(x,\xi)}\right)}\leq Ch^k\frac{(dk)^\bet}{\xit^{dk}}
\end{equation*}
for $\bet\leq dk$, $\xi\in F$, $\xit>dk$ and $x\in K$.
Thus, if $(x_0,\xi_0)\notin\WF_{\{\fM\}}(u;P)$ then Lemma \ref{WFalternativePara}(2)
 implies that
there exist constants $C,h,\nu>0$ and a weight sequence $\bM\in\fM$ such that
\begin{equation*}
\Betr{\left\langle f_k,e^{-ix\xi}\frac{w_k(x,\xi)}{p_d^k(x,\xi)}\right\rangle}
\leq Ch^k\frac{M_{dk}}{\xit^{dk}}\left(1+\xit\right)^\nu
\end{equation*}
when $\xi\in F$, $\xit>dk$ and $k> (n+\nu+1)/d$.

If $\mu=\max\{\tau,\nu\}$ then we conclude that
\begin{equation}\label{vElliptic3}
\Betr{\hat{u}_k(\xi)}\leq Ch^k\bigl(M_{dk}\bigr)^{\tfrac{dk+\mu}{dk}}\xit^{\mu-dk}
\end{equation}
for $\xi\in F$, $\xit>dk$ and $k>(n+\mu+1)/d$. We set 
\begin{equation*}
\tilde{v}_k=u_{\lfloor k/d\rfloor}
\end{equation*}
where $\lfloor y\rfloor$ denotes the largest integer $\leq y\in\R$. Hence,
due to \eqref{vElliptic3} and \eqref{RouDeriv},
there are constants $C,h>0$ and a weight sequence $\bM\in\fM$ such that
\begin{equation*}
\Betr{\F\left(\tilde{v}_k\right)(\xi)}\leq Ch^kM_k\xit^{\mu-k}
\end{equation*}
for $\xi\in F$, $\xit>dk$ and $k>n+\mu+1$.
If we put $v_k=\tilde{v}_{k+n+\mu+1}$ then there exist  $C,h>0$ and 
$\bM\in\fM$ such that
\begin{equation*}
\Betr{\hat{v}_k(\xi)}\leq Ch^kM_k\xit^{-k}
\end{equation*}
when $\xi\in F$ and $\xit>dk$.

Since $u$ is of order $\tau$ near $K$ it follows that the sequence 
$v_k$ is bounded in $\E^{\prime,\tau}(K)$. Thus we have
\begin{align*}
\Betr{\hat{v}_k(\xi)}&\leq C(1+\xit)^\tau\\
\intertext{and therefore, for $\xit<dk$,}
\xit^k\Betr{\hat{v}_k(\xi)}&\leq C(dk)^{k+\tau}\\
\intertext{and since $\fM$ is $R$-semiregular we obtain that there are $C,h>0$ and $\bM\in\fM$ such that}
\Betr{\hat{v}_k(\xi)}&\leq Ch^kM_{k}\xit^{-k}
\end{align*}
for $\xit\leq dk$. We conclude that $(x_0,\xi_0)\notin\WF_{\{\fM\}}u$.

If $\fM$ is $B$-semiregular and
$(x_0,\xi_0)\notin\WF_{(\fM)}(u;P)\cap\Char P$ then we can argue similarly in
order to conclude that
$(x_0,\xi_0)\notin\WF_{(\fM)} u$.
\end{proof}

Recall that a system $\{P_1,\dotsc,P_\ell\}$ of differential operators
defined on $U$ is said to be elliptic iff 
\begin{equation*}
\bigcap_{j=1}^\ell\Char P_j=\emptyset.
\end{equation*}
\begin{Cor}\label{EllipticCor}
	Let $\sP=\{P_1,\dotsc,P_\ell\}$ be an elliptic system of analytic differential operators 
	and $\fM$ be a $[$semiregular$]$ weight matrix.
	Then
	\begin{equation*}
	\bigcap_{j=1}^\ell \vDC[P_j]{\fM}{U}=\DC{\fM}{U}.
	\end{equation*}
	In particular
	\begin{equation*}
    \vDC[\sP]{\fM}{U}=\DC{\fM}{U}.
	\end{equation*}
\end{Cor}
\begin{proof}
We have only to prove $\bigcap\vDC[P_j]{\fM}{U}\subseteq\DC{\fM}{U}$.
Assume that $u\in\bigcap\vRou[P_j]{\fM}{U}$.
Then $\WF_{[\fM]}(u;P_j)=\emptyset$ for all $j=1,\dotsc,\ell$.
Hence by Theorem \ref{VectorEllipticThm}
\begin{equation*}
\WF_{[\fM]}u\subseteq\bigcap_{j=1}^\ell\Char P_j=\emptyset.
\end{equation*}
We conclude that $u\in\DC{\fM}{U}$, cf.\ Proposition \ref{Singsupport}.
Therefore we have obtained
\begin{equation*}
\DC{\fM}{U}\subseteq\vDC[\sP]{\fM}{U}\subseteq\bigcap_{j=1}^\ell\vDC[P_j]{\fM}{U}\subseteq\DC{\fM}{U},
\end{equation*}
cf.\ Proposition \ref{VectorProperties}(3).
\end{proof}

\begin{Rem}
	Clearly the correspondence between weight functions and their associated weight matrices as
	described in Subsection \ref{subsec:Weightfct} yields instantly the transfer of all results 
	in this section to structures given by weight functions. Thus we have in particular generalized
	the results of \cite{MR3380075} to operators with analytic coefficients.
	We note here only the version of Corollary \ref{EllipticCor}:
\end{Rem}
\begin{Cor}\label{OmegaEllipticCor}
	Let $\sP=\{P_1,\dotsc,P_\ell\}$ be an elliptic system of analytic differential operators and
	 $\omega$ be a weight function such that $\omega(t)=o(t)$ for $t\rightarrow\infty$.
	 Then 
	 \begin{equation*}
	 	\bigcap_{j=1}^\ell\vDC[P_j]{\omega}{U}=\DC{\omega}{U}.
	 \end{equation*}
\end{Cor}
Here we have to generalize the definition of $\vDC[P_j]{\omega}{U}$ from Section \ref{Introduction}
in analogy to Definition \ref{DefinitionVectors}. 
However, note that by Remark \ref{SubRemark} the two definitions agree 
 for subelliptic systems of operators. 
The proof of Corollary \ref{OmegaEllipticCor} follows then immediately from Corollary 
\ref{EllipticCor}, if we recall that $\fW$ satisfies \eqref{newexpabsorb}.
We leave the details to the reader. 

\section{Ultradifferentiable scales}\label{sec:Scales}
In this section we introduce the notion of ultradifferentiable scales and apply them
to the Problem of Iterates of analytic differential operators of principal type.
\subsection{Definition}\label{Subsec:DefScales}
Let $\Lambda$ be 
a totally ordered set. 
We  call a map
 \begin{equation*}
 \zeta:\Lambda\times [0,\infty)\rightarrow [0,\infty)
 \end{equation*}
 a \emph{generating function} if
for each $\lambda\in\Lambda$ the function  
$\zeta_\lambda=\zeta(\lambda,\,.\,)$ is continuous, increasing
and satisfies the following conditions:
\begin{gather*}
 \zeta_\lambda(0)=0,\\
 \text{the mapping }k\mapsto \log k +\zeta_\lambda(k)-\zeta_\lambda(k-1)\text{ is increasing,}\\
\lim_{t\rightarrow\infty}\frac{\zeta_\lambda(t)}{t}=\infty.
\end{gather*}
For $\lambda\leq \lambda^\prime$ we also assume that $\zeta_\lambda(t)\leq
\zeta_{\lambda^\prime}(t)$ when $t\in [1,\infty)$.

To each such $\zeta$ we can associate a weight matrix 
$\fM_\zeta=\{\bM^\lambda=\bM^\lambda_\zeta:\,\lambda\in\Lambda\}$ by setting 
\begin{equation*}
M_k^\lambda=k!e^{\zeta_\lambda(k)}.
\end{equation*}
More precisely, $\bM^\lambda$
is a weight sequence satisfying \eqref{AnalyticInclusion} 
for each $\lambda\in\Lambda$ and $\bM^\lambda\leq \bM^{\lambda^\prime}$ when $\lambda\leq\lambda^\prime$, by definition.
Hence every sequence $\bM^\lambda$ is semiregular if and only if
\begin{equation}\label{singleDerivClosed}
\fa\lambda\in\Lambda\ex \gamma>0:\;\; \zeta_\lambda(p+1)-
\zeta_\lambda(p)\leq \gamma(p+1)\quad\forall p\in\N_0\tag{$\square$}.
\end{equation} 
On the other hand the matrix $\fM_\zeta$ is $R$-semiregular if and only
if 
\begin{equation}\label{RoumieuDerivClosed}
\fa \lambda\in\Lambda\;\ex \sigma\in\Lambda\;\ex \gamma>0:\;\;
\zeta_\lambda (p+1)-\zeta_\sigma (p)\leq \gamma(p+1)\quad\forall p\in\N_0\tag{$\star$}
\end{equation}
and $B$-semiregular if and only if $\zeta$ satisfies
\begin{equation}\label{BeurlingDerivClosed}
\fa \lambda\in\Lambda\;\ex \sigma\in\Lambda\;\ex \gamma>0:\;\;
\zeta_\sigma (p+1)-\zeta_\lambda (p)\leq \gamma(p+1)\quad \forall p\in\N_0.\tag{$\diamond$}
\end{equation}

For a generating function $\zeta$
 we call the ordered family of weight sequences 
$(\bM^\lambda_\zeta)_\lambda$ the \emph{ultradifferentiable scale generated by 
 $\zeta$}.
We also say that $\fM_\zeta$ is the weight matrix associated to the scale
$(\bM^\lambda_\zeta)_\lambda$.

To each ultradifferentiable scale $(\bM_{\zeta}^\lambda)_\lambda$ 
we can associate two scales of
ultradifferentiable classes, namely 
\begin{equation*}
\left(\Rou{\bM^\lambda}{U}\right)_\lambda \quad \text{and}\quad
\left(\Beu{\bM^\lambda}{U}\right)_\lambda,
\end{equation*}
the scale of Roumieu classes and of Beurling classes, respectively.
Clearly, $\DC{\bM^\lambda}{U}\subseteq\DC{\bM^\sigma}{U}$ when
$\lambda\leq\sigma$ and 
$\Beu{\fM_{\zeta}}{U}\subseteq\DC{\bM^\lambda}{U}
\subseteq\Rou{\fM_\zeta}{U}$ for all $\lambda\in\Lambda$.

We say that an ultradifferentiable scale $(\bM_\zeta^\lambda)_\lambda$
 with generating function $\zeta$ is \emph{fitting} if $\zeta$ satisfies
  \eqref{singleDerivClosed} and
\begin{equation}\tag{$\triangleright$}\label{R-scaleChange}
	\begin{gathered}
\forall\lambda\in\Lambda\;\fa\alpha>1\;\ex\lambda^\ast\geq\lambda\;\ex \gamma>0:\\
\zeta_\lambda (\alpha t)\leq \zeta_{\lambda^\ast}(t)+\gamma(t+1)
\quad\forall t\in [1,\infty). 
\end{gathered}
\end{equation}
On the other hand, the scale $(\bM_\zeta^\lambda)_\lambda$ is
\emph{apposite} if the generating function $\zeta$ obeys
\eqref{singleDerivClosed} and

\begin{equation}	
\begin{gathered}\tag{$\triangleleft$}\label{B-scaleChange}
\forall\lambda^\ast\in\Lambda\;\fa\alpha>1\;\ex\lambda\leq\lambda^\ast\;\ex \gamma>0:\\
\zeta_\lambda (\alpha t)\leq \zeta_{\lambda^\ast}(t)+\gamma(t+1)\quad\forall t\in [1,\infty).
\end{gathered}
\end{equation}
Furthermore, a scale $(\bM^\lambda_\zeta)_\lambda$ is \emph{$R$-admissible} if
\eqref{RoumieuDerivClosed} and \eqref{R-scaleChange} hold for $\zeta$
and \emph{$B$-admissible} if \eqref{BeurlingDerivClosed} and
\eqref{B-scaleChange} are satisfied.
We use the notation $[$admissible$]$ if the scale is either $R$- or $B$-admissible, 
depending on the context. 
Furthermore we say that a scale is \emph{admissible} if it is $R$- and $B$-admissible.
We observe that a fitting scale is also $R$-admissible and
an apposite scale is $B$-admissible but the other implications
do not hold in general.

If $\Lambda\subseteq V$ is the open positive cone 
of a totally ordered vector space $V$, we say that $\zeta$
is pseudo-homogeneous iff
\begin{equation}\label{PseudoHom}
\forall\lambda\in\Lambda\;\fa \alpha>1\;\ex \gamma,c,q>0:\;\; 
\zeta_\lambda(\alpha t)\leq \zeta_{c\alpha^q\lambda}(t)+\gamma(t+1)\;\;\forall t\in [1,\infty).
\end{equation}
If $\zeta$ is pseudohomogeneous then $\zeta$ satisfies both \eqref{R-scaleChange} and 
\eqref{B-scaleChange}.

\begin{Ex}\label{WFscalesExamples}
	The families of weight sequences from Example \ref{WSequencesExamples}
	are ultradifferentiable scales with pseudohomogeneous generating functions:
\begin{enumerate}
	\item The Gevrey scale $(\bG^{1+\lambda})_{\lambda>0}$ is generated by the function 
	  $\zeta(\lambda, t)=\lambda t\log t$ if $t>1$ and $\zeta(\lambda,t)=0$ 
	  for $0\leq t\leq 1$. 
	  Since the sequence $\bG^{1+\lambda}$ is semiregular for all
	  $\lambda>0$ we know that \eqref{singleDerivClosed} holds.
	  For $\alpha> 1$ and $t\geq 1$ we have that
	\begin{align*}
	\zeta(\lambda,\alpha t)
	&= \lambda \alpha t\log(\alpha t)\\
	&= (\alpha\lambda)t\left(\log\alpha+\log t\right)\\
	&\leq (\alpha\lambda)t\log t +\gamma(t+1)\\
	&=\zeta(\alpha\lambda,t) +\gamma(t+1).
	\end{align*}
	Hence $\zeta$ is pseudohomogeneous and therefore $(\bG^{1+\lambda})_\lambda$
	is a fitting and apposite scale.
	\item
	Let $r>1$. The scale $(\bL^{q,r})_{q>1}$ is generated by 
	  $\zeta^r(\lambda, t)=t^r\lambda$ where $\lambda=\log q$.
	We have that
	\begin{equation*}
	\zeta^r(\lambda,\alpha t)=(\alpha t)^r \lambda 
	=t^r(\alpha^r \lambda)=\zeta^r(\alpha^r\lambda,t).
	\end{equation*}
	It follows that the scale $(\bL^{q,r})_{q}$ is admissible.
	It is fitting and apposite if and only if $r\leq 2$.
	\item The generating function for the scale $(\bB^{1,\lambda}=\bB^\lambda)_{\lambda>0}$
	is $\zeta(\lambda,t)=\lambda t\log\log(t+e)$.
	For $\alpha>1$ and $t\geq 1$ we conclude
	\begin{align*}
	\zeta(\lambda,\alpha t)
	&=\lambda \alpha t\log\log (\alpha t+e)\\
	&\leq(\alpha\lambda) t\log\log (\alpha(t+e))\\
	&=(\alpha\lambda)t\log\left[\log(t+e)
	\left(1+\frac{\log\alpha}{\log(t+e)}\right)\right]\\
	&\leq (\alpha\lambda)t\left(\log\log(t+e)+\log(1+\log\alpha)\right)\\
	&\leq (\alpha\lambda) t\log\log(t+e)+\gamma_{\alpha,\lambda}(t+1)\\
	&=\zeta(\alpha\lambda,t)+\gamma_{\alpha,\lambda}(t+1).
	\end{align*}
	Hence the scale $(\bB^\lambda)_\lambda$ is fitting and apposite.
	\item Generally, the scale $(\bB^{j,\lambda})_\lambda$, $j\in\N$, is generated by 
	$\zeta^j(\lambda,t)=\lambda t\log^{(j+1)}(t+e^{(j)})$. If $\alpha >1$ and $t\geq 1$
	 we can argue analogously to above and  obtain
	\begin{align*}
	\zeta^j(\lambda,\alpha t)&=\lambda\alpha t\log^{(j+1)} \left(\alpha t+ e^{(j)} \right)\\
	&\leq (\alpha \lambda)t\left[\log^{(j+1)} \left(t+e^{(j)}\right)+
	\alpha^{[j]}\right]\\
	&\leq \zeta^j (\alpha\lambda,t)+\gamma_{\alpha,\lambda}(t+1),
	\end{align*}
	where $\alpha^{[j]}$ is defined recursively by $\alpha^{[1]}=\log(1+\log\alpha)$
	and $\alpha^{[j+1]}=\log(1+\log\alpha^{[j]})$.
	Therefore $(\bB^{j,\lambda})_\lambda$ is a fitting and apposite scale.
\end{enumerate}
\end{Ex}
\subsection{Vectors of operators of principal type}

If $P$ is an operator of principal type with analytic coefficients in $U\subseteq\R^n$
and  $(x_0,\xi_0)\in\COT$ then we say 
following \cite{MR0296509} that $P$ satisfies \emph{Condition} $C_{x_0,\xi_0}$
if either 
$p_d(x_0,\xi_0)\neq 0$
or 
$p_d(x_0,\xi_0)=0$ and for all $z\in\C$
with $d_\xi\real(zp_d(x_0,\xi_0))\neq 0$ we have that the function $\imag(zp_d)$,
restricted to the bicharacteristic strip of $\real(zp_d)$ through $(x_0,\xi_0)$, 
has a zero of finite even order.
We recall
\begin{Thm}[{\cite[Theorem II]{MR0296509}}]\label{TrevesThm}
Let $P$ be an analytic differential operator of principal type.
The following statements are equivalent:
\begin{enumerate}
	\item $P$ is hypoelliptic.
	 \item $P$ is subelliptic.
	 \item $P$ satisfies Condition $C_{x_0,\xi_0}$ for all $(x_0,\xi_0)\in\COT$.
\end{enumerate}	
\end{Thm}

Since $P$ is subelliptic 
we have by Remark \ref{SubRemark} 
that $u\in\vRou{\fM}{U}$ if and only if for every $V\Subset U$
there are $\bM\in\fM$ and constants $h,C>0$ such that $P^ku\in L^2(V)$ and
\begin{equation}\label{L2-estimate}
\norm[L^2(V)]{P^k u}\leq Ch^k M_{dk}
\end{equation}
for all $k\in\N_0$.
On the other hand $u\in\vBeu{\fM}{U}$ if and only if $P^ku\in L^2_{loc}(U)$ 
and for all $V\Subset U$, all $\bM\in\fM$ and all $h>0$ there is some $C>0$ 
such that \eqref{L2-estimate} is satisfied for all $k$.

The main technical result of \cite{MR654409} is the following theorem:
\begin{Thm}[{\cite[Theorem 1.2]{MR654409}}]\label{BaouendiMetivier}
	Let $P$ be a differential operator of order $d$ with analytic coefficients
	in $U\subseteq\R^n$.
	Let $(x_0,\xi_0)\in U\times\R^n\!\setminus\!\{0\}$
	and assume that there is a conic neighborhood $W_0\times\Gamma_0$ of 
	$(x_0,\xi_0)$ such that $P$ is of principal type in $W_0\times\Gamma_0$ 
	and Condition $C_{x,\xi}$ is satisfied for all $(x,\xi)\in V_0\times\Gamma_0$.
	
	Then there are neighborhoods
	$W^\prime\Subset W\Subset W_0$ of $x_0$, 
	a conical neighborhood $\Gamma\subseteq\Gamma_0$ of $\xi_0$, $C>0$, 
	$0\leq \delta<1$ and a sequence of  functions
	$(\psi_k)_k\subseteq\D(W)$ satisfying $0\leq \psi_k\leq 1$ and 
	$\psi_k\equiv 1$ on $W^\prime$ such that the following holds:
	For every $k\in\N$ and $u\in L^2(W)$ with $P^ku\in L^2(W)$ we have
	\begin{equation}\label{PrincipalEstimate}
	\left\lvert \lvert\xi\rvert^{(d-\delta)k}\widehat{\psi_k u}(\xi)\right\rvert
	\leq \frac{C^{k+1}}{(k!)^\delta}\left(\bigl\lVert P^ku\bigr\rVert_{L^2(W)}
	+(k!)^d\lVert u\rVert_{L^2(W)}\right)
	\end{equation}
	if $\xi\in\Gamma$.
\end{Thm}
\begin{Rem}\label{DeltaRemark}
According to \cite[Remark 1.2]{MR654409} the number $\delta$ in Theorem \ref{BaouendiMetivier}
can be chosen to be $0$ if $P$ is elliptic at $(x_0,\xi_0)$.
When $P$ is non-elliptic at $(x_0,\xi_0)$ then 
we can take $\delta=2k/(2k+1)$ where $2k$ is the maximum order of vanishing
of $\imag (zp_d)$ mentioned in Condition $C_{x,\xi}$, for $(x,\xi)$ in a 
compact neighborhood of $(x_0,\xi_0)$ and $z\in\C$.

Hence, if $V\Subset U$ then we set
\begin{equation*}
\delta=\delta(V)=\frac{2k}{2k+1}
\end{equation*}
where now $2k$ is the maximum order of vanishing of $\imag(zp_d)$ in 
Condition $C_{x,\xi}$ for $(x,\xi)\in V\times\R^n\!\setminus\!\{0\}$.

Note that $\delta(V)$ is closely related to the subellipticity of $P$:
For $V\Subset U$ we can choose in \eqref{Subelliptic} $\eps=d-\delta(V)$,
see \cite{MR290201}.
\end{Rem}
Now suppose that $P$ is a hypoelliptic operator of principal 
type with analytic coefficients in $U$
and that $(\bM^\lambda_\zeta)$ is a fitting ultradifferentiable scale
with generating function $\zeta$.
Recall that Theorem \ref{TrevesThm} implies that Condition $C_{x,\xi}$ holds for
all $(x,\xi)\in \COT$.
Furthermore let $u\in\D^\prime(U)$ be an $\{\bM^\lambda\}$-vector of $P$ 
for some $\lambda\in\Lambda$  and $(x_0,\xi_0)\in \COT$. 
Applying  Theorem \ref{BaouendiMetivier} we conclude that there are  neighborhoods
$W^\prime\Subset W\Subset U$ of $x_0$, a conical neighborhood $\Gamma$ of $\xi_0$, 
$0\leq\delta<1$ and a bounded sequence $u_k\in\E^\prime(W)$ 
such that $u\vert_{W^\prime}=u_k\vert_{W^\prime}$ and
\begin{equation*}
\left\lvert \lvert\xi\rvert^{(d-\delta)k}\hat{u}_k(\xi)\right\rvert
\leq \frac{C^{k+1}}{(k!)^\delta}\left(C_0 h_0^kM^\lambda_{dk}+(k!)^d\right)\\
\end{equation*}
for $\xi\in\Gamma$, where  $C,C_0,h_0>0$ are independent of $k\in \N$ and 
$\delta=\delta(V)$ is defined in Remark \ref{DeltaRemark}.
Now, since \eqref{AnalyticInclusion} is satisfied for all $\bM^\lambda$, we have that
for each $\rho>0$ there exists $C_\rho>0$ such that 
$1\leq C_\rho \rho^k m_k^\lambda$ for all $k\in\N_0$.
Applying also Stirling's formula we obtain that there are constants $h>0$ and $C>0$ such that
\begin{equation*}
	\left\lvert \lvert\xi\rvert^{(d-\delta)k}\hat{u}_k(\xi)\right\rvert
	\leq Ch^kk^{(d-\delta)k} e^{\zeta_\lambda(dk)},\qquad k\in\N.
\end{equation*}

If we denote by $\lceil y\rceil$ the smallest integer $\geq y\in\R$ then we choose for $\ell\in\N$ an integer $k_\ell$ in the following way
\begin{equation*}
k_\ell=\left\lceil\frac{\ell}{d-\delta}\right\rceil\leq \frac{\ell}{d-\delta}+1
\leq\frac{\ell +d}{d-\delta}
\end{equation*}
and therefore $\ell\leq (d-\delta)k_\ell$. 
Note that if $\delta\geq 1$ then $\delta^{k_\ell}\leq \delta^{(\ell+d)/(d-\delta)}$ and on the other hand $0<\delta<1$ implies that
$\delta^{k_\ell}\leq \delta^{\ell/(d-\delta)}$.
Thus, if we set $v_\ell=u_{k_\ell}$ then we have that
\begin{equation*}
\begin{split}
\lvert\xi\rvert^\ell\bigl\lvert \hat{v}_\ell(\xi)\bigr\rvert
&\leq \left\lvert\lvert\xi\rvert^{(d-\delta)k_\ell}\hat{u}_{k_\ell}(\xi)\right\rvert
\\
&\leq Ch^{k_\ell}
 k_\ell^{(d-\delta)k_\ell}\exp\left[\zeta_\lambda(dk_\ell)\right]\\
&\leq Ch^{\tfrac{\ell}{d-\delta}}(d-\delta)^{-\ell-d}(\ell +d)^{\ell+d}
\exp\left[\zeta_\lambda\left(\frac{d}{d-\delta}(\ell+d)\right)\right]\\
&\leq C \left(\frac{\gamma h^{1/(d-\delta)}}{d-\delta}\right)^\ell(\ell+d)^{\ell+d}\,\exp\left[\zeta_{\lambda^\ast}(\ell +d)\right]
\end{split}
\end{equation*}
for $\xi\in\Gamma$ with $\lvert\xi\rvert\geq 1$ and some $\lambda^\ast$ according
to \eqref{R-scaleChange}. 
Then \eqref{singleDerivClosed} and the Stirling formula imply that
\begin{equation*}
\lvert\xi\rvert^\ell\bigl\lvert \hat{v}_\ell(\xi)\bigr\rvert
\leq Ch^\ell \ell!\,e^{\zeta_{\lambda^\ast}(\ell)}
= Ch^{\ell}M^{\lambda^\ast}_\ell,\qquad \ell\in\N,
\end{equation*}
for some constants $C,h>0$.
Hence $(x_0,\xi_0)\notin\WF_{\{\bM^{\lambda^\ast} \}}u$.

If  $u\in\D^\prime(U)$ is a $(\bM^\lambda)$-vector of $P$ for some $\lambda$, 
then we have by essentially the same arguments
that for every 
$(x_0,\xi_0)\in \COT$ there is some
$\lambda^\ast\in\Lambda$ such that $(x_0,\xi_0)\notin\WF_{(\bM^{\lambda^\ast})} u$.
 
 In fact, we have obtained the following theorem.
\begin{Thm}\label{ScaleCorollary1}
	Let $P$ be a hypoelliptic differential operator of principal type 
	with analytic coefficients in $U\subseteq\R^n$ and 
	$(\bM^\lambda)_\lambda$ be an ultradifferentiable scale.
	Then the following holds:
	\begin{enumerate}
	\item If $(\bM^\lambda)_\lambda$ is fitting then
	for all $V\Subset U$ and all $\lambda\in\Lambda$ there is some
	$\lambda^\ast\in\Lambda$ such that every $[\bM^\lambda]$-vector 
	of $P$ in $U$ is 
	of class $[\bM^{\lambda^\ast}]$ in $V$.
	\item If the scale $(\bM^\lambda)_\lambda$ is apposite
	then for all $V\Subset U$ and all
	 $\lambda^\ast\in\Lambda$ there exists
	$\lambda\in\Lambda$ such that every $u\in\vDC{\bM^\lambda}{U}$
	is of class $[\bM^{\lambda^\ast}]$ in $V$.
	\end{enumerate}
\end{Thm}
\begin{proof}
	 Note first that by Remark \ref{DeltaRemark} 
	 for every $V\Subset U$ there is some $\delta(V)\in[0,1)$ such that 
	\eqref{PrincipalEstimate} holds with $\delta=\delta(V)$ for all 
	$(x_0,\xi_0)\in V\times\R^n\!\setminus\!\{0\}$.
	Condition \eqref{R-scaleChange} implies that 
	  for every $\lambda$ there is
	some $\lambda^\ast$ such that
	\begin{equation}\label{ScaleChangeApplied}
	\zeta_\lambda\left(\frac{d}{d-\delta(V)}t\right)\leq \zeta_{\lambda^\ast}(t)+C(t+1)
	\end{equation}
	for all $t\in[1,\infty)$ and some $C>0$. Thus the above arguments give
	\begin{equation*}
\WF_{[\bM^{\lambda^\ast}]}u	\cap(V\times\R^n\!\setminus\!\{0\})=\emptyset.
	\end{equation*}
	Hence $u$ is of class $[\bM^{\lambda^\ast}]$ in $V$ 
	by Proposition \ref{Singsupport}, which proves (1).
	
	On the other hand, by \eqref{B-scaleChange} we obtain that
	for every $\lambda^\ast$ there is some $\lambda$ such that
	\eqref{ScaleChangeApplied} holds for $t\in[1,\infty)$ and some $C>0$. 
	Adapting the arguments above we then conclude that
	$\WF_{[\bM^{\lambda^\ast}]} u\cap (V\times\R^n\!\setminus\!\{0\})=\emptyset$
	for all $u\in\vDC{\bM^\lambda}{U}$ and thus Proposition \ref{Singsupport} implies
	again that $u$ is of class $[\bM^{\lambda^\ast}]$ in $V$.
\end{proof}
For special scales, like the Gevrey scale, cf.\ \cite[Theorem 1.3]{MR654409},
  we may obtain rather precise information about the loss of regularity of vectors.
  For example, for the other scales in Example \ref{WFscalesExamples} we have
\begin{Cor}\label{Cor:qScale}
Let $P$ be as in Theorem \ref{ScaleCorollary1}, $V\Subset U$, 
$\delta=\delta(V)$ be as defined
 in Remark \ref{DeltaRemark} and $u\in\Dp(U)$.
 \begin{enumerate}
 	\item If $u$ is an $[\bL^{q,r}]$-vector of $P$ for some $q>1$ and
  $1<r\leq 2$ then $u$ is of class $[\bL^{q^\prime,r}]$ in $V$, where
 	\begin{equation*}
 	q^\prime=q^{d^r/(d-\delta)^r}.
 	\end{equation*}
 \item If $u$ is a $[\bB^{j,\lambda}]$-vector for some $j\in\N$ and $\lambda>0$, then $u$ is of class $[\bB^{j,\lambda^\prime}]$ in $V$
 where 
 \begin{equation*}
 	\lambda^\prime=\frac{d}{d-\delta(V)}\lambda.
 \end{equation*}
\end{enumerate}
\end{Cor}

\begin{Thm}\label{Theorem2}
Let $P$ be as in Theorem \ref{ScaleCorollary1}, $(\bM^\lambda_\zeta)_\lambda$ be an $[$admissible$]$ ultradifferentiable scale
and $\fM_\zeta$ the associated weight matrix.
Then 
\begin{equation*}
\vDC{\fM_\zeta}{U}=\DC{\fM_\zeta}{U}.
\end{equation*}
\end{Thm}
\begin{proof}
We begin with the Roumieu case.
If $u\in\vRou{\fM_\zeta}{U}$ then for every $V\Subset U$ there are $\lambda\in\Lambda$
and $C,h>0$ such that
\begin{equation*}
\norm[L^2(V)]{P^k u}\leq Ch^k M^\lambda_{dk}, \quad k\in\N_0.
\end{equation*}
Suppose that $(x_0,\xi_0)\in V\times\R^n\!\setminus\!\{0\}$.
As above we obtain from Theorem \ref{BaouendiMetivier} and  \eqref{PrincipalEstimate} that there is a bounded sequence $u_k\in\Ep(V)$
such that $u_k\vert_W=u\vert_W$ for some neighborhood $W\Subset V$ of $x_0$
and 
\begin{equation*}
\Betr{\xi}^{(d-\delta(W))k}\Betr{\hat{u}_k(\xi)}\leq
 Ch^{k}k^{(d-\delta)k}e^{\zeta_\lambda(dk)},\qquad \xi\in\Gamma,\;
 \lvert\xi\rvert\geq 1,
\end{equation*}
where $C>0,h>0$, $\Gamma$ is a conic neighborhood of $\xi_0$ and $\delta=\delta(V)$, depending only on the operator $P$ and $V$, is as in Remark \ref{DeltaRemark}.
If we choose $k_\ell$, $\ell\in\N$, as before and set $v_\ell=u_{k_\ell}$
then we can conclude in the same manner from \eqref{R-scaleChange} that
\begin{equation*}
\Betr{\xi}^\ell\Betr{\hat{v}_\ell(\xi)}\leq Ch^\ell (\ell+d)!
\exp\left[\zeta_{\lambda^\ast}(\ell+d)\right]
\end{equation*}
for some $\lambda^\ast\in\Lambda$. Hence \eqref{RoumieuDerivClosed} gives
\begin{equation*}
\Betr{\xi}^\ell\Betr{\hat{v}_\ell(\xi)}\leq Ch^\ell \ell!
\exp\left[\zeta_{\lambda^\prime}(\ell)\right]\leq Ch^\ell M_\ell^{\lambda^\prime}
\end{equation*}
for some constants $C,h>0$ and $\lambda^\prime\in\Lambda$
 independent of $\ell$.
Therefore, since $(x_0,\xi_0)\in V\times\R^n\!\setminus\!\{0\}$ was chosen 
arbitrarily,
\begin{align*}
\WF_{\{\bM^{\lambda^\prime}\}}u\cap \bigl(V\times\R^n\!\setminus\!\{0\}\bigr)&=\emptyset\\
\intertext{and by Theorem \ref{WF-Intersection}}
\WF_{\{\fM_\zeta\}}u\cap \bigl(V\times\R^n\!\setminus\!\{0\}\bigr)&=\emptyset.
\end{align*}
Since this holds for all $V\Subset U$ it follows that 
$\WF_{\{\fM_\zeta\}}u=\emptyset$. Hence $u\in\Rou{\fM_\zeta}{U}$ by
Proposition \ref{Singsupport}.

If $u\in\vBeu{\fM_\zeta}{U}$ then for all $V\Subset U$, $\lambda\in\Lambda$
and $h>0$ there is a constant $C>0$ such that
\begin{equation*}
\norm[L^2(V)]{P^ku}\leq Ch^kM_{dk}^\lambda,\qquad k\in\N_0.
\end{equation*}
If $(x_0,\xi_0)\in V\times\R^n\!\setminus\!\{0\}$ then
Theorem \ref{BaouendiMetivier} gives that there is a bounded sequence
$u_k\in\Ep(V)$ with $u_k\vert_W=u\vert_W$ in some neighborhood $W\Subset V$
of $x_0$.
Furthermore there is a conic neighborhood $\Gamma$ of $\xi_0$ such that for 
all $\lambda\in\Lambda$ and all $h>0$ there exists a constant $C>0$ such that
\begin{equation*}
\Betr{\xi}^{(d-\delta(W))k}\Betr{\hat{u}_k(\xi)}\leq
Ch^{k}k^{(d-\delta)k}e^{\zeta_\lambda(dk)},\qquad \xi\in\Gamma,\, \lvert\xi\rvert\geq 1.
\end{equation*}
 If $k_\ell$ for $\ell\in\N$ is defined as before then 
it is easy to see that \eqref{B-scaleChange} implies that
for all $\lambda^\ast$ and $h>0$ there is a constant $C>0$ such that
\begin{equation*}
\Betr{\xi}^\ell\Betr{\hat{v}_\ell(\xi)}\leq Ch^\ell (\ell+d)!
\exp\left[\zeta_{\lambda^\ast}(\ell+d)\right].
\end{equation*}
It follows from \eqref{BeurlingDerivClosed} that for all $\lambda^\prime\in\Lambda$
and $h>0$ there is some $C>0$ such that for all $\ell\in\N_0$ we have
\begin{equation*}
\Betr{\xi}^\ell\Betr{\hat{v}_\ell(\xi)}\leq Ch^\ell \ell!
\exp\left[\zeta_{\lambda^\prime}(\ell)\right]= Ch^\ell M_\ell^{\lambda^\prime}.
\end{equation*}
Hence
\begin{align*}
\WF_{(\bM^{\lambda^\prime})}u\cap \bigl(V\times\R^n\!\setminus\!\{0\}\bigr)&=\emptyset\\
\intertext{for all $\lambda^\prime\in\Lambda$ and therefore by Proposition
	\ref{WF-Intersection}}
\WF_{(\fM_\zeta)}u\cap \bigl(V\times\R^n\!\setminus\!\{0\}\bigr)&=\emptyset.
\end{align*}
This means that $\WF_{(\fM_\zeta)}u=\emptyset$ and
consequently $u\in\Beu{\fM_\zeta}{U}$.
\end{proof}
\begin{Cor}
	Let $P$ be as in Theorem \ref{ScaleCorollary1}. Then
	\begin{align*}
		\vDC{\fQ^r}{U}&=\DC{\fQ^r}{U}, & r&>1,\\
		\intertext{and}
		\vDC{\fB^j}{U}&=\DC{\fB^j}{U}, & j&\in\N.
	\end{align*}
\end{Cor}
\begin{Ex}\label{MixedExample}
	Let $P$ be as in Theorem \ref{ScaleCorollary1}.
	\begin{enumerate}
		\item 
If we consider the scale $(\bL^{e,r})_r$ with
		generating function $\zeta(r,t)=t^r$ and associated weight matrix $\fR$ 
		then we have also that
		\begin{equation*}
		\vDC{\fR}{U}=\DC{\fR}{U}.
		\end{equation*}
		Indeed, since $\bL^{q,r_1}\lhd\bL^{e,r_2}$ for all $q>1$ and $r_1<r_2$, 
		we obtain that
		for all $\alpha>1$ and $r_1 <r_2$ there is a constant $\gamma>0$ such that
		\begin{equation*}
		\zeta(r_1,\alpha t)= \alpha^{r_1}t^{r_1}\leq \zeta(r_2,t)+
		\gamma (t+1),\qquad t\geq 1.
		\end{equation*}
\item
	We can also show that
\begin{equation*}
\vBeu{\fJ}{U}=\Beu{\fJ}{U}
\end{equation*}
where $\fJ=\fJ^1=\{\bB^{j,1}:\; j\in\N\}$.
Indeed, $\fJ$ is associated to the scale $(\bB^{j,1})_{j\in\N}$, which is 
generated by $\zeta(j,t)=t\log^{(j+1)}(t+e^{(j)})$.
Here we consider $\Lambda=(\N,\preceq)$ with the inverse order $\preceq$ defined by
\begin{equation*}
j\preceq k:\Longleftrightarrow k\leq j
\end{equation*}
for $j,k\in\N$. 
More generally, the function $\zeta_{\sigma}(j,t)=\sigma \zeta(j,t)$ generates
the ultradifferentiable scale $(\bB^{j,\sigma})_j$ for $\sigma>0$.
If $\alpha>0$ then we compute
\begin{equation}\label{MixedCond1}
\begin{split}
\zeta(j,\alpha t)&=(\alpha t)\log^{(j+1)}(\alpha t+e^{(j)})\\
&\leq (\alpha t)\log^{(j+1)}(t+e^{(j)})+\gamma_{j,\alpha}(t+1)\\
&\leq \alpha t\log^{(j)}(t+e^{(j-1)})+\gamma_{j,\alpha}(t+1)\\
&=\zeta_{\alpha}(j-1,t)+\gamma_{j,\alpha}(t+1)
\end{split}
\end{equation}
for $t\geq 1$ since $\log^{(j+1 )}(t+e^{(j)})\leq \log^{(j)}(t+e^{(j-1)})$ when $t\geq 1$.
Since $\fJ^\sigma$ is the weight matrix associated to $(\bB^{j,\sigma})_j$
we obtain by arguing as in the proof of Theorem \ref{Theorem2} that for
$V\Subset U$ all $(\fJ)$-vectors of $P$ are of class $(\fJ^\alpha)$ in $V$,
where $\alpha=d/(d-\delta(V))$ and $\delta(V)$ is as in Remark \ref{DeltaRemark}.
We have proven the claim because $\Beu{\fJ^\sigma}{V}=\Beu{\fJ}{V}$ 
for all $\sigma>0$, cf.\ Example \ref{WMatrixExamples}(3).
\end{enumerate}
\end{Ex}

\begin{Rem}\label{MixedRemark}
	In the last example the estimate \eqref{MixedCond1} involved two different scales
	in a ``mixed'' version of \eqref{PseudoHom}.
We can use ``mixed'' versions of \eqref{R-scaleChange} and \eqref{B-scaleChange} to
obtain results similar to Theorem \ref{ScaleCorollary1} in the case of weight matrices.
More precisely, 
let $P$ be a hypoelliptic analytic differential operator of principal type
 on $U\subseteq\R^n$ with analytic coefficients,
$V\Subset U$ and $\delta(V)$ as in Remark \ref{DeltaRemark}.
In the Roumieu case 
	we consider two  ultradifferentiable scales
	 $(\bM^\lambda_\zeta)_{\lambda\in\Lambda}$ and
	$(\bN^\upsilon_\eta)_{\upsilon\in\Upsilon}$ with generating functions 
	$\zeta:\Lambda\times [0,\infty)\rightarrow [0,\infty)$ and 
	$\eta:\Upsilon\times [0,\infty)\rightarrow [0,\infty)$ which satisfy both
\eqref{RoumieuDerivClosed} and 
	\begin{equation*}
		\fa\lambda\in\Lambda\,\ex\upsilon\in\Upsilon \;\ex \gamma>0:\;
		\zeta_{\lambda}(\alpha t)\leq \eta_{\upsilon}(t)+\gamma(t+1)\quad \fa t\in [1,\infty)
	\end{equation*}
where $\alpha=d/(d-\delta(V))$ and $d$ denotes the order of the operator.
Then every $u\in\vRou{\fM_{\zeta}}{U}$ is of class $\{\fN_\eta\}$ in $V$,
where $\fM_{\zeta}$ and $\fN_\eta$ are the weight matrices associated
to the scales $(\bM^\lambda_\zeta)_\lambda$ and $(\bN^\upsilon_\eta)_\upsilon$,
respectively.
	In the Beurling setting we assume that the generating functions $\zeta$ and $\eta$
	satisfy both \eqref{BeurlingDerivClosed} and 
	\begin{equation*}
		\fa\upsilon\in\Upsilon\,\ex\lambda\in\Lambda \,\ex\gamma>0:\;
		\zeta_{\lambda}(\alpha t)\leq \eta_{\upsilon}(t)+\gamma(t+1)\quad \fa t\in [1,\infty).
	\end{equation*}
	Then any $(\fM_{\zeta})$-vector $u\in\Dp(U)$ is of class $(\fN_\eta)$ in $V$.
\end{Rem}

\begin{Rem}
	Another important fact in Example \ref{MixedExample}(2) was that 
	the weight matrices $\fJ^\sigma$ associated to the scales $(\bB^{j,\sigma})_j$
	satisfy $\fJ^\rho(\approx)\fJ^\sigma$ for all $\rho,\sigma$. Of course,
	we can express this property in terms of the generating functions of the scales.
	
Assume, again, that two ultradifferentiable
scales $(\bM^\lambda_\zeta)_{\lambda\in\Lambda}$ and
	$(\bN^\upsilon_\eta)_{\upsilon\in\Upsilon}$  with generating functions 
	$\zeta:\Lambda\times [0,\infty)\rightarrow [0,\infty)$ and 
	$\eta:\Upsilon\times [0,\infty)\rightarrow [0,\infty)$, respectively, are given.
For such a pair of generating functions we define an auxillary function 
$\Phi^\zeta_\eta:\Lambda\times\Upsilon\times (0,\infty)\rightarrow \R$ by
	\begin{equation*}
		\Phi^\zeta_\eta(\lambda,\upsilon;t)=
		\frac{\zeta_{\lambda}(t)-\eta_{\upsilon}(t)}{t}.
	\end{equation*}
It is clear that $\bM^\lambda\preceq\bN^\upsilon$ if 
$\limsup_{t\rightarrow\infty} \Phi^\zeta_{\eta}(\lambda,\upsilon;t)<\infty$.
We can distinguish the following cases:
	\begin{enumerate}
		\item We have that $\fM_\zeta\{\preceq\}\fN_\eta$ when
		\begin{equation*}
			\fa \lambda\in\Lambda \ex \upsilon\in\Upsilon:\;
			\limsup_{t\rightarrow\infty} \Phi_\eta^\zeta(\lambda,\upsilon;t)<\infty.
		\end{equation*}
	\item On the other hand $\fM_\zeta(\preceq)\fN_\eta$ if
	\begin{equation*}
		\fa\upsilon\in\Upsilon\ex\lambda\in\Lambda:\;
		\limsup_{t\rightarrow\infty} \Phi_\eta^\zeta(\lambda,\upsilon;t)<\infty.
	\end{equation*}
\item  Finally $\fM_\zeta\{\lhd)\fN_\eta$ when
\begin{equation*}
	\fa\lambda\in\Lambda \fa\upsilon\in\Upsilon:\; \lim_{t\rightarrow\infty}
	\Phi_\eta^\zeta(\lambda,\upsilon;t)=-\infty.
\end{equation*}
	\end{enumerate}

We might also ask ourselves, when do two ultradifferentiable scales generate the same 
scales of Denjoy-Carleman classes?
In order to give an answer to this question, we say that two generating functions
$\zeta:\Lambda\times[0,\infty)\rightarrow [0,\infty)$ 
and $\eta:\Upsilon\times[0,\infty)\rightarrow [0,\infty)$ are comparable if
there is a bijective mapping $\chi:\Lambda\rightarrow \Upsilon$ such that
for each $\lambda\in\Lambda$ we have
\begin{equation*}
	-\infty< \liminf_{t\rightarrow\infty}\Phi_\eta^\zeta(\lambda,\chi(\lambda);t)
\leq\limsup_{t\rightarrow\infty}\Phi_\eta^\zeta(\lambda,\chi(\lambda);t)
<+\infty.
\end{equation*}
If $\zeta$ and $\eta$ are comparable then $\bM^\lambda_\zeta\approx\bN_\eta^{\chi(\lambda)}$
and thus $\DC{\bM^\lambda}{U}=\DC{\bN^{\chi(\lambda)}}{U}$ 
and $\vDC[\sP]{\bM^\lambda}{U}=\vDC[\sP]{\bN^{\chi(\lambda)}}{U}$ for all $\lambda\in\Lambda$ and all systems $\sP$ of differential operators.
\end{Rem}

\section{Scales induced by weight functions}\label{WeightFctScales}
\subsection{Condition \eqref{om7}}
In this section we are going to prove Theorem \ref{omMainThm},
but first we need to analyze condition \eqref{om7}.
It is useful for our deliberations to set
\begin{equation*}
	\begin{gathered}
\W_0=\bigl\{\omega\in\CC([0,\infty);\R):\;\,\omega(t)\rightarrow\infty 
\text{ is increasing},\qquad\\\qquad\qquad\qquad\qquad\qquad\qquad\qquad\qquad
 \omega\vert_{[0,1]}\equiv 0\text{ and }
\omega\text{ satisfies \eqref{om3} and \eqref{om4}}\bigr\}
	\end{gathered}
\end{equation*}
since we have the following Lemma.
\begin{Lem}\label{Basicom7Lemma}
If $\omega\in\W_0$ satisfies $\eqref{om7}$ then $\omega$ is a weight function.
Furthermore there is some $0<\alpha<1$ such that $\omega=O(t^\alpha)$.
\end{Lem}
\begin{proof}
	It is easy to see that \eqref{om7} implies \eqref{om2}.
	On the other hand, by \cite[Lemma A.1]{sectorialextensions1} and 
	\cite[Lemma 4.3]{sectorextensions} 
	we know that $\omega$ satisfies the strong non-quasianalyticity condition:
	\begin{equation*}
	\exists C>0:\fa y>0:\; \int\limits_1^\infty \frac{\omega(yt)}{t^2}\leq C\omega(y)+C.
	\end{equation*}
	Hence \cite[Corollary 4.3]{MeiseTaylor88} states that there has to be some $0<\alpha<1$ 
	such that $\omega(t)=O(t^\alpha)$ for $t\rightarrow \infty$.
\end{proof}

We continue by recalling from \cite[Lemma A.1, Remark A.2]{sectorialextensions1},
cf.\ also Example \ref{q-GevreyOmega}, the following fact.
\begin{Prop}\label{FacultyAbsorb}
	Let $\omega$ be a weight function which satisfies \eqref{om7} and denote
	its associated weight matrix by $\fW=\{\bW^\lambda:\;\lambda>0\}$.
	If we define another weight matrix $\widehat{\fW}$ by
	\begin{equation*}
\widehat{\fW}:=\left\{\left(k!W^\lambda_k\right)_k:\lambda>0,\;\bW^\lambda\in\fW \right\}
	\end{equation*}
	then $\fW\{\approx\}\widehat{\fW}$ and $\fW(\approx)\widehat{\fW}$.
\end{Prop}

The main idea of the proof of Theorem \ref{omMainThm} is to associate to $\omega$ the
scale generated by
\begin{equation*}
 \zeta_\omega(\lambda,t)=\frac{1}{\lambda}\varphi_\omega^\ast(\lambda t).
 \end{equation*} 
If $\omega$ satisfies $\eqref{om7}$ and 
$\widehat{\fW}$ is the weight matrix 
associated to the scale generated by $\zeta_\omega$ then
$\DC{\omega}{U}=\DC{\widehat{\fW}}{U}$ by Proposition \ref{FacultyAbsorb}.
The generating function $\zeta_\omega$ satisfies \eqref{RoumieuDerivClosed} and \eqref{BeurlingDerivClosed}:
\begin{Lem}
	Let $\omega\in\W_0$ and 
$\varphi^\ast_{\lambda,\omega}(t):=\tfrac{1}{\lambda}\varphi^\ast_\omega(\lambda t)$
for $\lambda>0$. Then we have
\begin{equation*}
\varphi^{\ast}_{\lambda,\omega}(t+1)\leq \varphi^\ast_{2\lambda,\omega}(t)
+\varphi^\ast_{2\lambda,\omega}(1),\qquad t\geq 0,
\end{equation*}
for all $\lambda>0$.
\end{Lem}
\begin{proof}
The argument is similar to the one in the proof of \eqref{omMG}, 
cf.\ \cite{MR3285413}.
We include the proof for the convenience of the reader.

Let $\lambda>0$ be arbitrary but fixed. 
The convexity of $\varphi^\ast_{\omega}$ implies that
$\varphi^\ast_\omega((t+s)/2)\leq\frac{1}{2}\varphi^\ast_\omega(t)
+\frac{1}{2}\varphi^\ast_{\omega}(s)$ for all $s,t\geq 0$.
Hence the choices $t^\prime:=\frac{t}{2\lambda}$ and $s:=2\lambda$ yield
$\varphi^\ast_\omega(\lambda t^\prime+\lambda)
\leq\frac{1}{2}\varphi^\ast_\omega(2\lambda t^\prime) 
+\frac{1}{2}\varphi^\ast_\omega(2\lambda)$. 
Thus we obtain for all $t^\prime\geq 0$ that
\begin{equation*}
\frac{1}{\lambda}\varphi^\ast_\omega(\lambda(t^\prime+1))\leq
\frac{1}{2\lambda}\varphi^\ast_\omega (2\lambda t^\prime)
+\frac{1}{2\lambda}\varphi^\ast_{\omega}(2\lambda).
\end{equation*}
\end{proof}
\begin{Lem}[cf.\ {\cite[Appendix A]{sectorialextensions1}}]\label{om7Char}
Let $\omega\in\W_0$. The following statements are equivalent:
\begin{enumerate}
	\item $\omega$ satisfies \eqref{om7}.
	\item For all $\gamma>1$ there is a constant $C>0$ such that
	\begin{equation}\label{om7gen}
	\omega\left(t^\gamma\right)\leq C(\omega(t)+1),\qquad t\geq 0.
	\end{equation}
\end{enumerate}
\end{Lem}
\begin{proof}
Condition \eqref{om7} is equivalent to the existence of constants $C,H>0$ such that
\begin{equation}\tag{$\Xi^\prime$}\label{om7alt}
\omega\left(t^2\right)\leq C(\omega(Ht)+1),\qquad t\geq 0.
\end{equation}
Hence \eqref{om7gen} implies \eqref{om7}.

For $\gamma>1$ fixed choose $j\in\N$ such that $\gamma\leq 2^j$.
If we iterate \eqref{om7alt} we conclude that
\begin{equation*}
\omega(t^\gamma)\leq \omega\bigl(t^{2^j}\bigr)\leq C(\omega(t)+1),\qquad t\geq 0,
\end{equation*}
for some constant $C>0$, since $\omega$ is increasing and \eqref{om7} implies \eqref{om2}.
\end{proof}

\begin{Lem}\label{omMixedScale}
Let $\alpha>1$ and $\omega,\sigma\in\W_0$. Then the following are equivalent:
\begin{enumerate}
\item
$\ex H\geq 1 \ex C>0:\qquad\qquad\quad\;\;\, \omega(t^\alpha)\leq C(\sigma(Ht)+1),
\qquad\quad\;\;\, t\geq 0$,
\item $\ex A\geq 1 \fa \lambda>0 \ex D>0: \quad \varphi^\ast_{\lambda,\sigma}(\alpha t)\leq
\varphi^\ast_{A\lambda,\omega}(t)+D(t+1),\quad t\geq 0$,
\item $\ex A\geq 1 \ex \lambda>0 \ex D>0:\quad \varphi^\ast_{\lambda,\sigma}(\alpha t)
\leq \varphi^{\ast}_{A\lambda,\omega}(t)
+D(t+1),\quad t\geq 0$.
\end{enumerate}
\end{Lem}
\begin{proof}
The implication $(2)\Rightarrow(3)$ is trivial. If $(3)$ holds then we have
for some $\lambda>0$
\begin{equation*}
\begin{split}
\varphi^{\ast\ast}_\sigma(y)&=\sup_{x\geq 0}\left[xy-\varphi^\ast_{\sigma}(x)\right]
=\sup_{x^\prime\geq 0}\left[\lambda\alpha x^\prime y-\varphi^\ast_{\sigma}(\lambda\alpha x^\prime)\right]\\
&\geq\sup_{x^\prime\geq 0}\left[\lambda\alpha x^\prime y- A^{-1}
\varphi^\ast_\omega(A\lambda x^\prime)-\lambda D(x^\prime +1)\right]\\
&=A^{-1}\sup_{w\geq 0}\left[\alpha wy-\varphi^\ast_\omega(w) -Dw\right]-D\lambda\\
&=A^{-1}\varphi^{\ast\ast}_\omega(\alpha y-D)-D\lambda.
\end{split}
\end{equation*}
Since $\varphi^{\ast\ast}_\tau=\varphi_\tau$ for any $\tau\in\W_0$, we conclude that
\begin{equation*}
\omega(t^\alpha)\leq A\sigma(e^{D/\alpha}t)+DA\lambda,\quad t\geq 0.
\end{equation*}
Hence we have proven (1) with $H=e^{D/\alpha}$ and $C=\max\{A,DA\lambda\}$.

On the other hand, if (1) holds then there are constants $C,h>0$ such that
\begin{equation*}
\varphi_\omega(\alpha t)\leq C\varphi_\sigma (t+h) +C, \quad t\geq 0.
\end{equation*}
Thus for $t\geq 0$
we can compute that
\begin{equation*}
\begin{split}
\varphi^\ast_{\omega}(t)&=\sup_{s\geq 0}
\left[\alpha st-\varphi_\omega(\alpha s)\right]
\\
&\geq \sup_{s\in\R}\left[\alpha st-C\varphi_\sigma(s+h)\right]-C\\
&\geq C\sup_{u\in\R}\left[\frac{t}{C}\alpha u-\varphi_\sigma(u)\right]
-h\alpha t-C\\
&=C\varphi_\sigma^\ast\left(\alpha C^{-1}t\right)-h\alpha t-C
\end{split}
\end{equation*}
where we have $\varphi_{\sigma}(u)=0$ for $u<0$ by normalization. 
Hence for all $\lambda>0$ and $t\geq 0$ we have
\begin{equation*}
\frac{1}{\lambda}\varphi^\ast_\sigma(\lambda \alpha t)
\leq \frac{1}{C\lambda}\varphi_\omega^\ast(C\lambda t)+h\alpha t+\frac{1}{\lambda}.
\end{equation*}
Thus (2) is verified with the constants $A:=C$ and $D:=\max\{h\alpha,\lambda^{-1}\}$.
Observe that $A$ does not depend on $\lambda$.
\end{proof}
An immediate consequence of Lemma \ref{omMixedScale} is
\begin{Cor}
If $\omega\in\W_0$ then the following are equivalent:
\begin{enumerate}\label{MixedCor1}
	\item For all $\alpha>1$ there exists $\sigma\in\W_0$ and $L\geq 1$ such that
	\begin{equation*}
	\omega(t^\alpha)\leq L(\sigma(Lt)+1),\quad t\geq 0.
	\end{equation*}
	\item For all $\alpha>1$ there exists $\sigma\in\W_0$ such that 
	\begin{equation*}
	\ex A\geq 1\fa \lambda>0\ex D>0:\quad 
	\varphi^\ast_{\lambda,\sigma}(\alpha t)
	\leq\varphi^\ast_{A\lambda,\omega}(t)+D(t+1),
	\quad t\geq 0.
	\end{equation*}
\end{enumerate}
\end{Cor}
Hence, if we combine Corollary \ref{MixedCor1} with Lemma \ref{om7Char} we obtain
\begin{Cor}
Let $\omega\in\W_0$. The following two conditions are equivalent:
\begin{enumerate}
	\item $\omega$ satisfies $\eqref{om7}$.
	\item The function $\zeta_\omega(\lambda,t)=\varphi^\ast_{\lambda,\omega}(t)$
	satisfies
	\begin{equation*}
	\fa \alpha>1\ex A\geq 1\fa \lambda>0 \ex D>0:\quad
	\varphi^\ast_{\lambda,\omega}(\alpha t)\leq
	\varphi^\ast_{A\lambda,\omega}(t)+D(t+1),\quad t\geq 0.
	\end{equation*}
\end{enumerate}
\end{Cor}
If we summarize we have proven
\begin{Prop}\label{omScale}
Let $\omega$ be a weight function such that \eqref{om7} holds.
Then the scale $(\bM^\lambda_\zeta)_{\lambda>0}$ generated by
$\zeta(\lambda,t)=\zeta_\omega(\lambda,t)=\varphi^\ast_{\lambda,\omega}(t)$
is admissible.
Furthermore, if $\fM=\fM_{\zeta}$ denotes the weight matrix associated to the scale
$(\bM^\lambda_\zeta)_\lambda$ then 
\begin{equation*}
\DC{\omega}{U}=\DC{\fM}{U},\qquad \vDC{\omega}{U}=\vDC{\fM}{U} 
\end{equation*}
for any differential operator $P$ with analytic coefficients.
\end{Prop}

\begin{proof}[Proof of Theorem \ref{omMainThm}]
Combine Theorem \ref{Theorem2} with Proposition \ref{omScale}.
\end{proof}

\begin{Rem}
It is clear that Theorem \ref{omMainThm} cannot hold for general weight functions. For example, if $s>1$ then  $\Rou{t^{1/s}}{U}\subsetneq\vRou{t^{1/s}}{U}$ for all 
non-elliptic operators $P$ by \cite[Theorem 1.3]{MR654409}. 
Using the proof of \cite[Theorem 1.2]{doi:10.1080/03605307808820078}
 Boiti and Jornet \cite[Example 3.1]{MR3652556} showed that
  if $P$ is not elliptic then there is
a weight function $\omega_P$ which is not equivalent to any Gevrey weight
function $t^{1/s}$ such that $\Rou{\omega_P}{U}\subsetneq\vRou{\omega_P}{U}$.
This example does not contradict Theorem \ref{omMainThm} since
$\omega_P$ does not satisfy \eqref{om7}.
In fact, for each $\omega_P$ there exist $1<s<s^\prime$ by construction 
such that $\G^s(U)\subsetneq\Rou{\omega_P}{U}\subsetneq\G^{s^\prime}(U)$,
but the class associated with a weight function satisfying \eqref{om7} 
is not contained
in any Gevrey class as the following result shows.
\end{Rem}
\begin{Prop}
	Let $\omega\in\W_0$ be such that \eqref{om7} holds.
	Then $\DC{\omega}{\R}\nsubseteq\DC{t^{1/s}}{\R}$ for all $s>1$.
\end{Prop}
\begin{proof}
	Suppose that $\DC{\omega}{\R}\subseteq\DC{t^{1/s}}{\R}$ for some $s>1$.
	Then according to \cite[Corollary 5.17(i)]{MR3285413} we obtain
	that $\omega\preceq t^{1/s}$, i.e.\ there is some $B>0$ such that
	\begin{equation*}
	t^{\tfrac{1}{s}}\leq B(\omega(t)+1), \quad t\geq 0,
	\end{equation*}
	and therefore by Lemma \ref{om7Char} for all $\alpha>1$ we can find
	a constant $B_1>0$ such that
	\begin{equation*}
	t^{\tfrac{\alpha}{s}}\leq B(\omega(t^\alpha)+1)\leq B_1(\omega(t)+1),
	\quad t\geq 0.
	\end{equation*}
	Hence if we choose $\alpha=s$ then $\omega\preceq t$, which means
	that $\DC{\omega}{\R}$ is contained in the space of analytic functions on $\R$.
	
	However, by Lemma \ref{Basicom7Lemma} there is some $0<\gamma<1$ 
	such that $t^\gamma\preceq\omega$, which in particular implies
	that the space of analytic functions is strictly contained in
	$\DC{\omega}{\R}$.
\end{proof}
\subsection{Some remarks}
We can use the ``mixed'' conditions of Corollary \ref{MixedCor1} to obtain results
like Theorem \ref{ScaleCorollary1}, cf.\ also Remark \ref{MixedRemark},
for weight functions. In fact, the conditions in Corollary \ref{MixedCor1} seem 
to be similar to those in Remark \ref{MixedRemark}.
However, arguing absolutely analogously to Section \ref{sec:Scales} we would not obtain results for some weight functions $\omega$ and $\sigma$ and their
associated weight matrices $\fW$ and $\fS$ but for the weight matrices
$\widehat{\fW}$ and $\widehat{\fS}$, cf. Proposition \ref{FacultyAbsorb}.
As we have seen that does not matter if $\omega=\sigma$ satisfies \eqref{om7}.

But for the ``mixed'' setting note first that we can drop $(k!)^{-\delta}$ in
\eqref{PrincipalEstimate} since $(k!)^{\delta}\geq 1$ for all $k\in\N_0$ and
 $\delta>0$. The other estimates before Theorem \ref{ScaleCorollary1} remain also
 valid if we drop the ``factorial'' factors of the form $k^{k(d-\delta)}$.
 We obtain therefore the following Theorem, but we need to  discuss subsequently
 how it fits in the theory presented in Section \ref{sec:Scales}.

\begin{Thm}\label{MixedThm}
Let $P$ be a hypoelliptic operator of principal type with analytic coefficients 
in $U\subseteq \R^n$, $V\Subset U$ and $\delta=\delta(V)$
as in Remark \ref{DeltaRemark}. 
Furthermore suppose that $\omega$ and $\sigma$ are two weight functions satisfying
\begin{equation*}
\omega(t^{\alpha})=O(\sigma(Ht)),\qquad t\rightarrow\infty,
\end{equation*}
 where $H\geq 1$ and 
$\alpha =d/(d-\delta)$.
Then every $[\sigma]$-vector of $P$ is
an ultradifferentiable function of class $[\omega]$ in $V$.
\end{Thm}
\begin{proof}
	We denote by $\fW=\{\bW^\lambda:\,\lambda>0\}$, $W^\lambda_k=\varphi^\ast_{\lambda,\omega}(k)$,
	 and $\fS=\{\bS^\lambda:\,\lambda>0\}$, $S^\lambda_k=\varphi^\ast_{\lambda,\sigma}(k)$, the weight matrices
	 associated to $\omega$ and $\sigma$, respectively.
	According to Corollary \ref{MixedCor1} there is a constant $A>0$
	such that for every $\lambda>0$ we have 
	\begin{equation}\label{Mixedpseudohomogeneous}
	\varphi^{\ast}_{\lambda,\sigma}(\alpha t)\leq \varphi^\ast_{A\lambda,\omega}+ D(t+1)
	\end{equation}
	for some constant $D>0$.
	
If $u\in\vRou{\sigma}{U}$ then there exist $\lambda>0$, $h>0$ and $C>0$ such that 
\begin{equation*}
\norm[L^2(V)]{P^ku}\leq Ch^k S^\lambda_k
\end{equation*}
for all $k\in\N_0$.
Now \eqref{PrincipalEstimate} and \eqref{Mixedpseudohomogeneous} imply similarly to the argument before Theorem \ref{ScaleCorollary1} that
\begin{equation*}
\WF_{\{\bW^{A\lambda}\}}u\cap V\times\R^n\!\setminus\!\{0\}=\emptyset.
\end{equation*}
Hence  $u\vert_V\in\Rou{\fW}{V}=\Rou{\omega}{V}$
by Proposition \ref{WF-Intersection} and 
Proposition \ref{Singsupport}.

The Beurling case follows analogously.
\end{proof}

\begin{Rem}
If we set in Theorem \ref{MixedThm} $\omega(t)=t^{1/(\alpha s)}$ and $\sigma(t)=t^{1/s}$ 
then we obtain that any $s$-Gevrey vector is a $\alpha s$-Gevrey function in $V$.
But this is a weaker result than  \cite[Theorem 1.3]{MR654409}. In particular by
Theorem \ref{MixedThm} we would only obtain that an analytic vector is an
$\alpha$-Gevrey function in $V$.

This reflects the difference in the definition of the ultradifferentiable scales:
In section \ref{sec:Scales} we have defined the weight sequences $\bM^\lambda$
 of the scale generated by $\zeta$ by $m_k^\lambda=\exp\circ\zeta_{\lambda}(k)$, i.e.\
$M^\lambda_k=k!(\exp\circ\zeta_{\lambda}(k))$, whereas the definition of the scale associated
to a weight function in this section corresponds to $M_k^\lambda=\exp\circ\zeta_{\lambda}(k)$. 
By Proposition \ref{FacultyAbsorb} the two definitions 
are essentially equivalent when the weight 
function satisfies \eqref{om7}.
For the moment we may call a scale $(\bM^\lambda)_\lambda$ of 
semiregular weight sequences weak if it is defined
via the sequences 
$M_k^\lambda=\exp\circ\zeta_{\lambda}(k)$.\footnote{In order to guarantee that 
	the sequences $\bM^\lambda$ of a weak scale
	are semiregular weight sequences we need to change the definition of generating functions in Subsection \ref{Subsec:DefScales} a little bit. For example, instead of demanding only that 
$\log k +\zeta_\lambda(k)-\zeta_{\lambda}(k-1)$ is increasing in $k$ for fixed $k$
we need that the sequence $\zeta_\lambda(k)-\zeta_{\lambda}(k-1)$ is increasing.
Furthermore we replace $\lim (t^{-1}\zeta_\lambda(t))=\infty$
 for each $\lambda$
with $\lim (t^{-1}\zeta_{\lambda}(t)-\log(t))=\infty$.
In fact, these are the only changes necessary.}
On the other hand, we might say that the scales from Section \ref{sec:Scales},
i.e.\ those given by $k!\exp\circ\zeta_{\lambda}(k)$, are strong.

We observe that Theorem \ref{MixedThm} shows that it would make a big difference if we would have used weak scales in Section \ref{sec:Scales}:
As we pointed out above, for the Gevrey scale it would mean that we could only prove
a weaker version of \cite[Theorem 1.3]{MR654409}, and we would prove
the Roumieu version of Proposition \ref{GevreyCor} but not the Beurling version.
We note also that in this situation
 the scales $(\bB^{j,\sigma})_\sigma$ are not recognized under
the framework of weak scales, as the fact from above that analytic vectors might only be
Gevrey functions indicates.

On the other hand, if we consider scales that are larger than the Gevrey scales there
is not much difference. In the case of the scales $(\bL^{q,r})_q$
we have already noted that for the proof of Theorem \ref{Theorem2} for 
the matrices $\fQ^r$ (and therefore for the weight matrix $\fR$)
there is no real difference if we use the scale $(\bL^{q,r})_q$ or the
scale $(\bN^{q,r})_q$ given by $N^{q,r}_k=q^{k^r}$.
In fact, we have the following variant of Corollary \ref{Cor:qScale}:
\end{Rem}
\begin{Cor}
	Let $q>1$, $1<r\leq 2$ and $P$ be as in Theorem \ref{ScaleCorollary1} and 
	suppose that $u$ is an $[\bN^{q,r}]$-vector.
	If $V\Subset U$ then $u$ is of class $[\bN^{q^\prime,r}]$ in $V$,
	where $q^\prime$ is as in Corollary \ref{Cor:qScale}(1).
\end{Cor}
\begin{Rem}
In order to decide which kind of scales should be used for studying 
the regularity of vectors of a given operator, 
one can, in the case of operators which have been already studied, 
look at the regularity of Gevrey vectors.
The technical reason 
why strong scales are  advantageous for the study of vectors of operators of principal type is the factor $(k!)^{-\delta}$ in 
the main estimate
\eqref{PrincipalEstimate}, cf.\ the estimates before Theorem \ref{ScaleCorollary1}.
For another example using the definition from section \ref{sec:Scales} 
see Subsection \ref{Appendix1}.

On the other hand, in the case of H\"ormander's sum of squares operators, introduced
in \cite{Hoermander1967}, there is some $r>1$ depending on the operator such 
that $s$-Gevrey vectors are $rs$-Gevrey vectors and these results are strict,
see \cite{doi:10.1080/03605307808820078}, \cite{MR3556261} and also the survey in \cite{Derridj2017}.
A similar result was obtained for some class of locally integrable structures 
of corank one in \cite{MR3043156}.
Hence in these two instances weak scales
are more appropriate for the study of ultradifferentiable vectors.

We can try to analyze these examples to find some general conditions which can help 
to decide which  kind of scales to use for the study of vectors of a given operator or system of operators.
It seems that two properties play an important role: 
subellipticity and that analytic vectors are analytic. 
We have seen that hypoelliptic operators of principal type satisfy both conditions
(as the systems of vector fields from Subsection \ref{Appendix1} do).
In contrast, the sums of squares operator of H\"ormander are subelliptic but
there are analytic vectors which are not analytic functions, 
see \cite{doi:10.1080/03605307808820078} and also \cite{MR3556261}.
The analytic vectors of the locally integrable structures considered in \cite{MR3043156} are analytic but locally integrable structures are in general 
not subelliptic, cf.\ \cite{SEDP_2005-2006____A14_0}.
In this case we refer also to the discussion in  \cite[Section 10]{MR3043156}.
\end{Rem}

 \section{Miscellanea}

\subsection{Systems of vector fields}\label{Appendix1}
Let $X_1,\dotsc X_\ell$ be smooth real-valued vector fields on $U\subseteq\R^n$.
We say that the family $\sX=\{X_1,\dotsc,X_\ell\}$ is of finite type of order $\nu\in\N$ 
if the tangent space $T_x U$ at each point $x\in U$ is generated by the iterated
Lie brackets of order $\leq \nu$
\begin{equation*}
X_I=[X_{j_1},\dotsc,[X_{j_{k-1}},X_{j_k}]\dots],
\qquad I=(j_1,\dotsc,j_k)\in\{1,\dotsc,\ell\}^k,\;k\leq \nu.
\end{equation*}
The main result of \cite{Hoermander1967} is that if the system 
$\sX=\{X_1,\dotsc,X_\ell\}$ is of finite type then $\sX$ is hypoelliptic. 
In the case of analytic vector fields
\cite{SEDP_1970-1971____A12_0} showed that the finite type condition is also necessary for smooth hypoellipticity.

In \cite{doi:10.1080/03605308008820164} it was proven that  if
the family $\sX=\{X_1,\dotsc,X_\ell\}$ is analytic and of finite type then
\begin{equation*}
\bigcap_{j=1}^\ell\An (U;X_j)=\An(U).
\end{equation*}
For Gevrey vectors \cite{ASENS_1980_4_13_4_397_0} showed
that 
\begin{equation*}
\G^{1+\sigma}(U;\sX)\subseteq\G^{1+\nu\sigma}(U),\qquad \sigma\geq 0,
\end{equation*}
if the collection of analytic $X_1,\dotsc,X_\ell$ is of finite type of order $\nu$ and
generates a stratified nilpotent Lie algebra $\mathsf{G}$ of rank $\nu$, i.e.
\begin{equation*}
\mathsf{G}=\mathsf{G}_1\oplus\dots\oplus\mathsf{G}_\nu
\end{equation*}
with
\begin{align*}
\left[\mathsf{G}_1,\mathsf{G}_j\right]&=\mathsf{G}_{j+1},\qquad 1\leq j<\nu,\\
\left[\mathsf{G}_1,\mathsf{G}_\nu\right]&=0.
\end{align*}
The theory of ultradifferentiable scales from Section \ref{sec:Scales}
 allows us to generalize the result
of \cite{ASENS_1980_4_13_4_397_0}:
\begin{Thm}\label{DH-Theorem}
Let $\sX=\{X_1,\dotsc,X_\ell\}$ be a family of analytic, real-valued vector fields on $U\subseteq\R^n$ that is of finite type of order $\nu$ and generates a stratified
Lie algebra of rank $\nu$.
\begin{enumerate}
\item If $(\bM^{\lambda}_\zeta)_\lambda$ is a fitting scale then
for every $\lambda\in\Lambda$ there is some $\lambda^\ast\in\Lambda$
such that 
$\vDC[\sX]{\bM^\lambda}{U}\subseteq\DC{\bM^{\lambda^\ast}}{U}$.
\item If $(\bM_\zeta^\lambda)_\lambda$ is apposite then
for all $\lambda^\ast\in\Lambda$ there exists some $\lambda\in\Lambda$
such that
$\vDC[\sX]{\bM^\lambda}{U}\subseteq\DC{\bM^{\lambda^\ast}}{U}$.
\item If $(\bM_{\zeta}^{\lambda})_\lambda$ is an $[$admissible$]$ scale then
$\vDC[\sX]{\fM_\zeta}{U}=\DC{\fM_\zeta}{U}$.
\end{enumerate}
\end{Thm}
\begin{proof}
If $\lvert I\rvert\leq \nu$ then 
according to \cite[Corollaire 2.3]{ASENS_1980_4_13_4_397_0}
there exist $N\in\N$, $a_1,\dotsc,a_N\in\R$ and $j(1),\dotsc,j(N)\in\{1,\dotsc,\ell\}$ such that for all $k\in\N$
we have
\begin{equation*}
\frac{X_I^k}{k!}=\sum_{k_1+\dots +k_N=k\lvert I\rvert}
\frac{a_1^{k_1}X_{j(1)}^{k_1}\cdots a_N^{k_N}X_{j(N)}^{k_N}}{k_1!\cdots k_N!}.
\end{equation*}
Hence if $u\in\vRou[\sX]{\bM^\lambda}{U}$ and $\lvert I\vert=\mu$
then  for $V\Subset U$  there are constants $C,h>0$ such that
\begin{equation*}
\frac{\norm[L^2(V)]{X_I^ku}}{k!}
\leq Ch^{\mu k}\exp[\zeta_\lambda(\mu k)]
\sum_{k_1+\dots +k_N=\mu k}\frac{(\mu k)!\lvert a_1\rvert^{k_1}\cdots\lvert a_N\rvert^{k_N}}{k_1!\cdots k_N!}.
\end{equation*}
If we apply \eqref{R-scaleChange} then we obtain that
\begin{equation*}
\frac{\norm[L^2(V)]{X_I^ku}}{k!}
\leq C\gamma\left(\gamma h^{\mu}\left(\lvert a_1\rvert+\dots +\lvert a_N\rvert\right)^{\mu}\right)^k\exp[\zeta_{\lambda_\mu}( k)]
\end{equation*}
for some $\lambda_\mu$. In other words
\begin{equation*}
\norm[L^2(V)]{X_I^ku}\leq C_1h_1^k M^{\lambda_\mu}_k
\end{equation*}
for some constants $C_1,h_1>0$
and therefore $u\in\vRou[X_I]{\bM^{\lambda_{\mu}}}{U}$.
If $\lambda^\ast=\max\left\{\lambda_1,\dotsc,\lambda_{\nu}\right\}$
then
\begin{equation*}
u\in\bigcap_{\lvert I\rvert\leq \nu}\vRou[X_I]{\bM^{\lambda^\ast}}{U}.
\end{equation*}
Now Corollary \ref{EllipticCor} implies that $u\in\Rou{\bM^{\lambda^\ast}}{U}$.

The rest of the theorem can be shown in a similar manner.
\end{proof}
\begin{Cor}
Let $\sX$ be as in Theorem \ref{DH-Theorem}
 and assume that $\omega$ is a weight function which satisfies \eqref{om7}. Then
\begin{equation*}
\vDC[\sX]{\omega}{U}=\DC{\omega}{U}.
\end{equation*}
\end{Cor}

\subsection{Weight sequences, associated weight functions and condition \eqref{om7}}\label{Appendix2}
If $\bM$ is a weight sequence then the associated weight function
\begin{equation*}
\omega_\bM(t)=\sup_{k\in\N_0}\log\frac{t^k}{M_k}
\end{equation*}
may not satisfy \eqref{om2} in general, cf.~\cite[Theorem 3.1]{Schindl21}
Therefore $\omega_\bM$ is not necessarily a weight function 
in the sense of \cite{MR1052587}.
However, $\omega_\bM\in\W_0$, see e.g.\ \cite{Komatsu73} 
or \cite{MR0051893}.
Thus, if $\omega_\bM$ satisfies \eqref{om7} then $\omega_\bM$ 
is a weight function
in the sense of \cite{MR1052587} by Lemma \ref{Basicom7Lemma}.
We also recall that 
\begin{equation}\label{InverseWeight}
M_k=\sup_{t>0}\frac{t^k}{\exp(\omega_\bM(t))}
\end{equation}
for all $k\in\N_0$.

We want to characterize those weight sequences $\bM$ 
for which $\omega_{\bM}$ 
satisfies \eqref{om7}. 
For technical reasons we will also assume in this section that $M_1\geq 1$ for all weight
sequences $\bM$. 

We begin  with an analogue of 
Lemma \ref{omMixedScale}.
\begin{Lem}
	Let $\bM,\bN$ be two weight sequences. The following statements are equivalent:
	\begin{enumerate}
		\item $\ex H\geq 1 \ex C>0:\quad \omega_\bN(t^2)\leq C
		\left(\omega_\bM(Ht)+1\right),\quad t\geq 0,$
		\item $\ex q\in\N \ex A,\gamma>0:\quad M_k^{2q}\leq A\gamma^k N_{qk},\quad k\in\N_0.$
	\end{enumerate}
\end{Lem}
\begin{proof}
	If (1) holds then we can without loss of generality assume that $C\in\N$. 
	Hence \eqref{InverseWeight} gives
	\begin{equation*}
	\begin{split}
	M_{2k}&=\sup_{t>0}\frac{(Ht)^{2k}}{\exp(\omega_\bM(Ht))}\\
	&\leq eH^{2k}\sup_{t>0}\frac{t^{2k}}{\exp(C^{-1}\omega_\bN(t^2))}\\
	&=eH^{2k}\left(\sup_{s> 0}\frac{s^{Ck}}{\exp(\omega_\bN(s))}\right)^{1/C}\\
	&=eH^{2k}N^{1/C}_{Ck}.
	\end{split}
	\end{equation*}
	Since $M_k^2\leq M_{2k}$ by \eqref{logconvexity} we have proven
	(2) for $q=C$, $A=e^C$ and $\gamma=H^{2C}$.
	
	On the other hand (2) implies that 
	\begin{equation*}
	M_{k}^2\leq A^{1/q}\gamma^{k/q}N_{qk}^{1/q}=A_1\gamma_1^kN_{qk}^{1/q}.
	\end{equation*}
	We denote by $\fM=\{\bM^{(\lambda)}:\;\lambda>0\}$ resp.\ 
	$\fN=\{\bN^{(\lambda)}:\;\lambda>0\}$ the weight matrix associated to
	$\omega_\bM$ resp.\ $\omega_\bN$.
	It is easy to show that $\bM=\bM^{(1)}$ 
	(see for example the proof of \cite[Theorem 6.4]{MR3601829})
	and observe also that
	\begin{equation*}
	N^{(q)}_k=\exp\left[q^{-1}\varphi_{\omega_\bN}^\ast(qk)\right]
	=\left(N^{(1)}_{qk}\right)^{1/q}=N_{qk}^{1/q}
	\end{equation*}
	for all $q\in\N$ and $k\in\N_0$.
	Therefore from (2) we obtain
	\begin{equation*}
	\ex q\in\N\,\ex A_1,\gamma_1>0:\quad M_k^2\leq A_1\gamma_1^kN_k^{(q)},
	\quad k\in\N_0,
	\end{equation*}
	and thus
	\begin{equation*}
	\log \left(\frac{t^k}{N_k^{(q)}}\right)\leq 
	\log \left(\frac{(\gamma_1t)^k}{M_k^2}\right)+\log A_1
	=2\log\left(\frac{(\gamma_1 t)^{k/2}}{M_k}\right)+\log A_1
	\end{equation*}
	for all $t>0$ and $k\in\N_0$. Hence by definition
	\begin{equation*}
	\omega_{\bN^{(q)}}\bigl(t^2\bigr)\leq 2\omega_\bM\left(\sqrt{\gamma_1}t\right)+\log A_1,\quad t\geq 0.
	\end{equation*}
We recall that $\omega_{\bN^{(\lambda)}}\sim\omega_\bN$, more precisely we have
\begin{equation*}
\fa \lambda>0 \ex D_\lambda >0:\quad \lambda\omega_{\bN^{(\lambda)}}(t)\leq \omega_\bN(t)\leq 
2\lambda\omega_{\bN^{(\lambda)}}(t)+D_\lambda,\quad t\geq 0,
\end{equation*}
cf.\ \cite[Section 5]{MR3285413} or \cite[Lemma 2.5]{sectorextensions}.

Combining the last two estimates together we conclude that
\begin{equation*}
\omega_{\bN}(t^2)\leq 4q\omega_\bM\bigl(\sqrt{\gamma_1}t\bigr) +2q\log A_1+D_q,\quad
t\geq 0.
\end{equation*}
Hence (1) is proven with $H=\sqrt{\gamma_1}$ and $C=\max\{4q,2q\log (A_1)+D_q\}$.
\end{proof}
\begin{Cor}
	Let $\bM$ be a weight sequence. Then the following are equivalent:
	\begin{enumerate}
		\item The associated weight function $\omega_\bM$ satisfies \eqref{om7}.
		\item There is a positive integer $p\in\N$ and constants $A,B>0$ such that
		\begin{equation}\label{om7Seq}
		\bigl(M_k\bigr)^{2p}\leq AB^kM_{pk}
		\end{equation}
		holds for all $k\in\N_0$.
	\end{enumerate}
\end{Cor}
Note that in \eqref{om7Seq} we can assume that $p\geq 2$, because
$p=1$ would yield that $\sup_k M_k^{1/k}<\infty$.

It is a natural question to ask if there is a weight sequence $\bM$ such that
\eqref{om7Seq} and $\DC{\bM}{U}=\DC{\omega_\bM}{U}$.
However, according to \cite{BonetMeiseMelikhov07}, a necessary condition for
the last identity is for $\bM$ to be of moderate growth, i.e.\
there is a constant $\gamma>0$ such that
\begin{equation}\label{ModerateGrowth}
M_{j+k}\leq \gamma^{j+k}M_{j}M_k
\end{equation}
for all $j,k\in\N_0$.
\begin{Lem}
A weight sequence $\bM$ does not satisfy simultaneously \eqref{om7Seq} and
\eqref{ModerateGrowth}.
\end{Lem}
\begin{proof}
Assume that both \eqref{om7Seq} and \eqref{ModerateGrowth} hold 
for $\bM$. Then \eqref{ModerateGrowth} implies that
\begin{equation*}
M_{pk}\leq \gamma^{p^2k}\bigl(M_k\bigr)^p, \quad k\in\N_0,
\end{equation*}
where $p$ is the integer from \eqref{om7Seq}.
Hence we have, if we combine the estimate above with \eqref{om7Seq}, that
\begin{equation*}
\bigl(M_k\bigr)^{2p}\leq AB^kM_{pk}\leq AB^k\gamma^{p^2k}\bigl(M_k\bigr)^p.
\end{equation*}
It follows that $\sup_k (M_k)^{1/k}<\infty$ and therefore $\bM$ is not
a weight sequence.
\end{proof}
\begin{Ex}
	\begin{enumerate}
	\item
The sequences $\bN^{q,r}=q^{k^r}$ satisfy \eqref{om7Seq} for 
$p\geq 2^{1/(r-1)}$:
Then $2p\leq p^r$ and therefore 
\begin{equation*}
\left(N_k^{q,r}\right)^{2p}=q^{2pk^r}\leq q^{(pk)^r}=N^{q,r}_{pk}.
\end{equation*}

\item The sequence $\bM$ given by $M_0=1$ and $M_k=e^{e^k}$, $k\in\N$,
 satisfies 
\eqref{om7Seq} with $p=8$ because we have the estimate
\begin{equation*}
\left(e^{e^k}\right)^{16}=e^{16e^k}\leq e^{e^{8k}}
\end{equation*}
since $4+k\leq 8k$ for all $k\in\N$.
\end{enumerate}
\end{Ex}
\subsection{Families of weight functions: An example}
Let $P$ be again a hypoelliptic operator of principal type with analytic coefficients in an open set $U\subseteq\R^n$
and consider $\Omega=\{\omega_s:\;s>0\}$, where $\omega_s(t)=(\max\{0,\log(t))\})^s$ is the weight function from Example
\ref{q-GevreyOmega}.
Then by Theorem \ref{omMainThm} we know that $\vDC{\omega_s}{U}=\DC{\omega_s}{U}$, but
analogously to the case of weight matrices, i.e.\ families of weight sequences, we can also consider 
the spaces associated to $\Omega$, i.e.\ we define
\begin{align*}
\Rou{\Omega}{U}&=\left\{f\in\E(U):\;\fa V\Subset U\ex s>1\ex h>0\quad
\norm[V,\omega_s,h]{f}<\infty\right\},\\
\Beu{\Omega}{U}&=\left\{f\in\E(U):\;\fa V\Subset U\fa s>1\fa h>0\quad
\norm[V,\omega_s,h]{f}<\infty\right\}\\
\intertext{and also}
\vRou{\Omega}{U}&=\Bigl\{u\in\Dp(U):\;\fa V\Subset U\ex s>1 \ex h>0\quad
\lVert u\rVert^P_{V,\omega_s,h}<\infty\Bigr\},\\
\vBeu{\Omega}{U}&=\Bigl\{u\in\Dp(U):\;\fa V\Subset U\fa s>1 \fa h>0\quad
\lVert u\rVert^P_{V,\omega_s,h}<\infty\Bigr\}.
\end{align*}
We have
\begin{Prop}
If $P$ is a hypoelliptic analytic operator of principal type then
\begin{equation*}
\vDC{\Omega}{U}=\DC{\Omega}{U}.
\end{equation*}
\end{Prop}
\begin{proof}
Observe 
in the Beurling case that 
\begin{align*}
\vBeu{\Omega}{U}=\bigcap_{s>1}\vBeu{\omega_s}{U}=\bigcap_{s>1}\Beu{\omega_s}{U}
=\Beu{\Omega}{U}.
\end{align*}

On the other hand, if $u\in\vRou{\Omega}{U}$ then for all $V\Subset U$ there is some $s>1$ such that 
\begin{equation*}
u\vert_V\in\vRou{\omega_s}{V}=\Rou{\omega_s}{V}\subseteq\Rou{\Omega}{V}.
\end{equation*}
Hence $u\in\Rou{\Omega}{U}$.

For the other direction, recall that
$\DC{\omega_s}{U}\subseteq\vDC{\omega_s}{U}$ for all $s>0$
 by Proposition \ref{VectorProperties}.
 Arguing analogously to above gives the desired inclusion.
\end{proof}

However, it turns out that we have already encountered the spaces $\E^{[\Omega]}$:
\begin{Thm}\label{ThmAppC}
	Let $P$ as above and $\fR$ be as in Example \ref{WMatrixExamples}(2). 
	Then
\begin{align*}
\DC{\Omega}{U}&=\DC{\fR}{U},\\
\vDC{\Omega}{U}&=\vDC{\fR}{U}.
\end{align*}
\end{Thm}

The first equality follows from a more general theorem in \cite{MR3601829}.
In order to state that theorem we need to recall some notations.
If $\fM$ is a weight matrix we denote by 
$\Omega_\fM=\{\omega_{\bM}:\;\bM\in\fM\}$
the family of weight functions associated to $\fM$.
Similarly to above we can define the spaces of 
ultradifferentiable functions associated
to $\Omega_\fM$:
\begin{align*}
\Rou{\Omega_\fM}{U}&=\left\{f\in\E(U):\;\fa V\Subset U\ex \bM\in\fM\ex h>0\quad
\norm[V,\omega_\bM,h]{f}<\infty\right\},\\
\Beu{\Omega_\fM}{U}&=\left\{f\in\E(U):\;\fa V\Subset U\fa \bM\in\fM\fa h>0\quad
\norm[V,\omega_\bM,h]{f}<\infty\right\}
\end{align*}
We consider the following conditions
\begin{gather}
\fa\bM\in\fM\,\ex \bN\in\fM\,\ex C>0,\fa j,k\in\N_0:\;M_{j+k}\leq C^{j+k}N_jN_k,
\label{R-mg}\\
\fa\bN\in\fM\,\ex\bM\in\fM\,\ex C>0,\fa j,k\in\N_0:\; M_{j+k}\leq C^{j+k}N_jN_k,
\label{B-mg}\\
\fa\bM\in\fM\,\fa h>0\,\ex \bN\in\fM\,\ex D>0\,\fa k\in\N_0:\; h^kM_k\leq DN_k,
\label{R-L}\\
\fa\bN\in\fM\,\fa h>0\,\ex \bM\in\fM\,\ex D>0\,\fa k\in\N_0:\; h^kM_k\leq DN_k.
\label{B-L}
\end{gather}
In \cite{MR3601829} the following result was shown:
\begin{Thm}\label{SchindlThm}
	Let $\fM$ be a weight matrix. Then we have
	\begin{enumerate}
		\item If $\fM$ satisfies \eqref{R-mg} and \eqref{R-L} then
		$\Rou{\fM}{U}=\Rou{\Omega_\fM}{U}$.
		\item If $\fM$ satisfies \eqref{B-mg} and \eqref{B-L} then
		$\Beu{\fM}{U}=\Beu{\Omega_\fM}{U}$.
	\end{enumerate}
\end{Thm}
\begin{proof}[Proof of Theorem \ref{ThmAppC}]
	We recall from Example \ref{q-GevreyOmega} that the weight matrix associated
	to $\omega_s$, $s>1$, is $\fW_s=\{\bN^{q,r}:\;q>1\}$, where $r=s/(s-1)$ and $\bN^{q,r}=q^{k^r}$.
	Then $\omega_{\bN^{q,r}}\sim\omega_s$ for all $s>0$ and $q>0$ by 
	\cite[Lemma 5.7]{MR3285413}.
	Note also that by Proposition \ref{FacultyAbsorb} we have 
	$\fW_s[\approx]\fQ^r$.
	Thus $\DC{\Omega}{U}=\DC{\Omega_\fR}{U}$ and 
	$\vDC{\Omega}{U}=\vDC{\Omega_\fR}{U}$.
	
	For given $r^\prime>r>1$ choose $A>0$ large enough such that
	\begin{equation*}
	2^{r-1}\leq k^{r^\prime-r}+k^{1-r}\frac{\log A}{\log 2}
	\end{equation*}
	for all $k\in\N$. We conclude that
	\begin{equation*}
	N^{2,r}_{2k}\leq A^{k}\left(N_k^{2,r^\prime}\right)^2
	\end{equation*}
	which implies \eqref{R-mg} and \eqref{B-mg} by 
	\cite[Theorem 9.5.1 and Theorem 9.5.3]{SchindlThesis}.
	
	On the other hand, for any $h>0$ and $r>1$
	 we can choose $r^\prime>1$ and $D>0$ large enough such that
	\begin{equation*}
	k\log h\leq \log 2\left(k^{r^\prime}-k^{r}\right)+\log D
	\end{equation*}
	for all $k$
	which gives
	\begin{equation*}
	h^kN^{2,r}_k\leq DN^{2,r^\prime}_k,\qquad k\in\N_0.
	\end{equation*}
	It follows  that
	$\fR$ satisfies \eqref{R-L} and \eqref{B-L}.
	
	Hence $\DC{\Omega}{U}=\DC{\fR}{U}$ by Theorem \ref{SchindlThm}.
	A close inspection shows that the proof of 
	Theorem \ref{SchindlThm} 
	in \cite{MR3601829} applies also to the spaces $\vDC{\fM}{U}$.

	Therefore
	\begin{equation*}
	\vDC{\Omega}{U}=\vDC{\fR}{U}=\DC{\fR}{U}=\DC{\Omega}{U}.
	\end{equation*}
\end{proof}
\subsection{A characterization of ellipticity by non-Gevrey vectors}
The aim of this section is to prove Theorem \ref{MetivierThm2}.
We begin with noticing two easy observations, which we will need later on:
\begin{Lem}\label{D-AuxLemma1}
	Let $\bM$ be a weight sequence and $\rho, R\geq 1$. Then
	\begin{equation*}
	\rho^jM_{k+\ell}R^\ell\leq \rho^{j+\ell}M_k+M_{j+k+\ell}R^{j+\ell}
	\end{equation*}
	for all $j,k,\ell\in\N_0$.
\end{Lem}
\begin{proof}
	If $\mu_k=M_{k}/M_{k-1}$ then \eqref{logconvexity} implies that the sequence
	$(\mu_k)_k$ is increasing.
	For $\rho\geq \mu_{k+\ell}R$ we obtain that
	\begin{equation*}
	M_{k+\ell}R^\ell=M_k\mu_{k+1}R\dots \mu_{k+\ell}R\leq \rho^{\ell}M_k.
	\end{equation*}
	If $\rho\leq \mu_{k+\ell}R$ then
	\begin{equation*}
	\rho^j\leq \mu_{k+\ell+1}R\dots \mu_{k+\ell+j}R\leq \frac{M_{j+k+\ell}}{M_{k+\ell}}R^j.
	\end{equation*}
\end{proof}
\begin{Lem}[{cf.\ \cite[Lemma 5.7]{MR3285413}}]\label{D-AuxLemma2}
	Let $\omega\in\W_0$  and 
	$\fW=\{\bW^{\rho}:\rho>0\}$ be the weight matrix associated to $\omega$.
	If $\omega_{\rho}$ is the weight function associated to $\bW^{\rho}$
	then 
	\begin{equation*}
	\omega_{\rho}(t)\leq\frac{\omega(t)}{\rho}
	\end{equation*}
	for all $t> 0$.
\end{Lem}
\begin{proof}
	Since $\W^\rho_k=\exp\left[\rho^{-1}\varphi^\ast_\omega(\rho k)\right]$
	we obtain
	\begin{equation*}
	\omega_\rho(t)=\sup_{k\in\N_0}\left[k\log t-\rho^{-1}\varphi^\ast_\omega(\rho k)\right]
	\leq\sup_{s\geq 0}\left[s\log t-\rho^{-1}\varphi_\omega^\ast(\rho s)\right]
	=\frac{1}{\rho}\omega(t).
	\end{equation*}
\end{proof}
Before we can begin with the proof of Theorem \ref{MetivierThm2} we also need to
take a closer look at the scale $(\bN^q)_q$, given by $N_k^q=q^{k^2}$, specifically.
Recall that $\fN=\{\bN^q:\,q>1\}$ is the weight matrix associated to
$\omega_2(t)=(\max\{0;\log t\})^2$. More precisely, 
$\varphi_2(t)=\omega_2\circ\exp(t)=t^2$ and $\varphi_2^\ast(t)=t^2/4$.
Hence the canonical weight matrix $\fW^2=\{\bW^{2,\rho}:\,\rho>0\}$
associated to $\omega_2$
is given by $W^{2,\rho}_k=\exp(\rho k^2/4)$, cf.\ \cite[Section 5.5]{MR3711790}.
Thus it is convenient to set $\lambda=\log q$ and to write 
in a slight abuse of notation $\bN^\lambda=\bN^q$.
It follows that $\bN^\lambda=\bW^{2,4\lambda}$ and therefore Lemma \ref{D-AuxLemma2}
implies that
\begin{equation}\label{MellinEstimate}
\omega_{\bN^\lambda}(t)\leq \frac{(\log t)^2}{4\lambda},\qquad t\geq 1.
\end{equation}

Now observe that $(\bN^\lambda)_\lambda$ is the weak scale 
associated to the 
generating function $\zeta(t,\lambda)=\lambda t^2$,
which clearly can be extended to an entire function $\zeta(z,\lambda)$ 
in the first variable. Hence $\theta(z,\lambda)=\exp\circ\zeta(z,\lambda)$
is holomorphic in $z$ and when $\lambda$ is fixed
we have that for every strip $G=\{w=u+iv\in\C:\,a<u<b\}$ there is a constant $C>0$
such that $\Betr{\theta(z,\lambda)}\leq Ce^{-\Betr{\imag z}^2}$. 
It follows that
$\theta(\,.\,,\lambda)$ is the Mellin transform of the function
\begin{equation*}
\Theta(t,\lambda)=\frac{1}{2\pi i}\int\limits_{i\R}\!t^{-w}\theta(w,\lambda)\,dw
=\frac{1}{2\pi}\int\limits_{-\infty}^{\infty}\!
e^{-i\sigma\log t}e^{-\lambda\sigma^2}\,d\sigma,
\end{equation*}
that is
\begin{equation*}
\theta(z,\lambda)=\int\limits_{0}^{\infty}\!t^{z-1}\Theta(t,\lambda)\,dt,
\end{equation*}
see e.g.\ \cite{zbMATH03028224} or \cite{MR1511564}. In particular
\begin{equation}\label{GammaIdentity}
N_k^\lambda=\int\limits_0^\infty\! t^{k-1}\Theta(t,\lambda)\,dt.
\end{equation}
If we set $t=e^s$ we can compute 
\begin{equation*}
\Theta(e^s,\lambda)=\frac{1}{2\pi}\int\limits_{-\infty}^{\infty}\! e^{-i\sigma s}e^{-\lambda\sigma^2}\,d\sigma
=\frac{1}{\sqrt{4\pi\lambda}}e^{-\tfrac{s^2}{4\lambda}}
\end{equation*}
and therefore
\begin{equation*}
\Theta(t,\lambda)=\frac{1}{\sqrt{4\pi\lambda}}\exp\left[-\frac{(\log t)^2}{4\lambda}\right].
\end{equation*}

In order to prove Theorem \ref{MetivierThm2} it is enough to show the following
statement.
\begin{Thm}\label{MetivierThm3}
	Let $P$ be a differential operator with analytic coefficients on $U$ which is not
	elliptic at some point $x_0\in U$.
	Then  we have
	\begin{equation*}
	\Rou{\bN^\lambda}{U}\subsetneq\vRou{\bN^\lambda}{U}
	\end{equation*}
for all $\lambda>0$.
\end{Thm}
\begin{proof}
It is sufficient for given $\lambda>0$ to construct  a function $u$ which
is an $\{\bN^\lambda\}$-vector of $P$ but is not in $\Rou{\bN^\lambda}{U}$.
In order to do so we shall try to follow the pattern of the proof of  
\cite[Theorem 2.3]{doi:10.1080/03605307808820078}.
From now on let $\lambda>0$ be arbitrary but fixed and choose parameters
$\eps$, $\lambda^\prime>0$ and $0<\lambda_0<\lambda$ depending on $\lambda$
which will later be specified.
Since $P$ is not elliptic at $x_0$ there exists some $\xi_0\in S^{n-1}$ such that
\begin{equation}\label{notElliptic}
p_d(x_0,\xi_0)=0.
\end{equation}
Let $\delta>0$ be such that $B_0=\{x\in\R^n:\;\Betr{x-x_0}<2\delta\}\Subset U$ and let
$\psi\in\Rou{\bN^{\lambda_0}}{\R^n}$ be such that 
$\supp \psi\subseteq \{x\in\R^n:\, \Betr{x}<2\delta\}$ and
$\psi(x)=1$ for $\Betr{x}\leq \delta$.
Thus there are constants $C_0,h_0>0$ such that
\begin{equation}\label{D-Estimate1}
\Betr{D^\nu\psi(x)}\leq C_0h_0^{\lvert\nu\rvert}N^{\lambda_0}_{\lvert\nu\rvert}
\end{equation}
for all $x\in \R^n$.
This is possible since $\bN^\tau$ is non-quasianalytic for any $\tau>0$.

Then we define the function $u$ to be of the form
\begin{equation*}
u(x)=\int\limits_{1}^\infty\! \psi\left(t^\eps(x-x_0)\right)
\Theta(t,\lambda^\prime)e^{it(x-x_0)\xi_0}\,dt.
\end{equation*}
It follows that
\begin{equation*}
D^k_{\xi_0} u(x_0)=\int\limits_{1}^{\infty}\! t^k\Theta(t,\lambda^\prime)\,d t,
\end{equation*}
where $D_{\xi_0}=-i\tfrac{\partial}{\partial {\xi_0}}$ is the directional derivative in direction $\xi_0$.
Thence \eqref{GammaIdentity} implies that
\begin{equation*}
D^k_{\xi_0}u(x_0)= N^{\lambda^\prime}_{k+1} -\int\limits_0^1\!
t^k\Theta\left(t,\lambda^\prime\right)\,dt.
\end{equation*}
Since $\int_0^1 t^k\Theta(t,\lambda^\prime)\,dt\rightarrow 0$  when $k \rightarrow\infty$ we have shown that
 $u$ cannot be of class $\{\bN^\tau\}$ in any neighborhood of $x_0$
for all $\tau<\lambda^\prime$.

On the other hand, it is easy to see that
\begin{equation*}
P^ku(x)=\int_{1}^{\infty}\limits\!Q_k(x,t)\Theta(t,\lambda^\prime)
e^{it(x-x_0)\xi_0}\,dt
\end{equation*}
where $Q_k$ is defined recursively by
\begin{align*}
Q_0(x,t)&=\psi\left(t^\eps(x-x_0)\right)\\
\intertext{and}
Q_{k+1}(x,t)&=\sum_{\alp \leq d}\frac{1}{\alpha!}\partial^\alpha_\xi
p(x,t\xi_0) D^\alpha_xQ_k(x,t).
\end{align*}
Since $P$ is analytic in $U$ we have that there is a constant $H>0$ such that
 for all $\nu,\alpha\in\N_0^n$ with $\alp\leq d$, all $x\in B_{2\delta}(x_0)$ and
 all $t\geq 1$:
\begin{gather}\label{D-DiffProp1}
\Betr{D^\nu_x\partial^\alpha_\xi p(x,t\xi_0)}\leq H^{\lvert\nu\rvert+1}
\lvert\nu\rvert! t^{d-\alp},\\
\intertext{and due to \eqref{notElliptic} for all $0<\eps<1$ there is $C_1>0$
such that for all $t>0$ and all $x\in U$ with $\Betr{x-x_0}\leq 2\delta t^{-\eps}$:}
\Betr{p\bigl(x,t\xi_0\bigr)}\leq C_1 t^{d-\eps}\label{D-DiffProp2}.
\end{gather}
Using the above estimates \eqref{D-DiffProp1} and \eqref{D-DiffProp2} together with
Lemma \ref{D-AuxLemma1} it is easy to see that we can adapt the proof
of \cite[Lemme 2.1]{doi:10.1080/03605307808820078} and therefore obtain the following
statement.
\begin{Lem}
	There exists a constant $A>0$ such that for all $k\in\N_0$,
	all $\nu\in\N_0^n$, all $x\in B_0$ and all $t\geq 1$ we have
	\begin{equation}\label{D-eq1}
	\Betr{D^\nu_xQ_k(x,t)}\leq C_0\left(h_0 t^\eps\right)^{\Betr{\nu}}A^k
	\left[t^{(d-\eps)k}N^{\lambda_0}_{\Betr{\nu}}
	+t^{(2d-1)k\eps}N^{\lambda_0}_{\Betr{\nu}+kd}\right].
	\end{equation}
\end{Lem}
If we set $\nu=0$ in \eqref{D-eq1} then we get
\begin{equation*}
\Betr{Q_k(x,t)}\leq C_0A^k\left(t^{(d-\eps)k}+t^{(2d-1)k\eps}N^{\lambda_0}_{dk}\right).
\end{equation*}
When we set $\rho=t^{1-\eps/d}$ and $R=t^{\eps(2-1/d)}$ then 
$\rho^{dk}=t^{(d-\eps)k}$ and $R^{dk}=t^{(2d-1)k\eps}$, respectively.
Hence, if $\lambda_1=\lambda -\lambda_0$ then we have
\begin{align*}
\rho^{dk}&\leq N^{\lambda}_{dk}e^{\omega_{\bN^{\lambda}}(\rho)}
\leq N^{\lambda}_{dk}
\exp\left[\frac{\left(\log t\right)^2}{4\lambda\tfrac{d^2}{(d-\eps)^2}}\right],
\\
R^{dk}&\leq N_{dk}^{\lambda_1}e^{\omega_{\bN^{\lambda_1}}(R)}
\leq N^{\lambda_1}_{dk}
\exp\left[\frac{(\log t)^2}{4\lambda_1\tfrac{d^2}{\eps^2(2d-1)^2}}\right]
\end{align*}
by \eqref{MellinEstimate}
and thus
\begin{equation*}
\Betr{Q_k(x,t)}\leq C_0A^kN^\lambda_{dk}\exp\left[
\frac{(\log t)^2}{4}\left(\frac{(d-\eps)^2}{\lambda d^2}
+\frac{\eps^2(2d-1)^2}{\lambda_1d^2}\right)\right].
\end{equation*}
If for fixed $\lambda>0$ we choose the parameters
$0<\lambda_0<\lambda$ and $\eps$ such that
\begin{equation*}
0<\eps\leq\frac{d\sqrt{\lambda-\lambda_0}}{\sqrt{\lambda-\lambda_0}+\sqrt{\lambda}(2d-1)}<\frac{1}{2},
\end{equation*}
then
\begin{equation*}
\frac{\eps^2(2d-1)^2}{(\lambda-\lambda_0)d^2}\leq \frac{(d-\eps)^2}{\lambda d^2}.
\end{equation*}

It follows that 
\begin{equation*}
\Betr{P^ku(x)}\leq C_0A^kN_{dk}^\lambda\int\limits_{1}^\infty
\!\exp\left[\frac{\left(\frac{(d-\eps)^2}{\lambda d^2}
	-\frac{1}{\lambda^\prime}\right)(\log t)^2}{4}\right]\,dt.
\end{equation*}
The integral converges as long as
\begin{equation*}
\lambda^\prime< \frac{d^2}{(d-\eps)^2}\lambda.
\end{equation*}
The proof of Theorem \ref{MetivierThm3} is complete if we put 
additionally $\lambda^\prime>\lambda$.
\end{proof}

\appendix
\section{Subelliptic estimates}\label{SubAppendix}
The aim of this appendix is to indicate how 
\eqref{Subelliptic} implies \eqref{Subelliptic1}.
Following \cite{MR0350177} we introduce the local Sobolev space $H^\sigma(V)$,
$\sigma\in\R$, over
an arbitrary open set $V\subseteq\R^n$ as the quotient space
$H^\sigma(V)=H^\sigma(\R^n)/F^\sigma(V)$,
where $F^\sigma(V)$ is the space of all functions $f\in H^\sigma(\R^n)$ which
vanish on $V$. Clearly $F^\sigma(V)$ is a closed subspace of $H^\sigma(\R^n)$
hence $H^\sigma(V)$ is a Hilbert space with the structure inherited from
$H^\sigma(\R^n)$.
\begin{Rem}
	It is easy to see that $H^0(V)=L^2(V)$ for all open sets $V\subseteq\R^n$.
	However, if we consider the classical Sobolev space
	\begin{equation*}
		W^k(V)=\left\{f\in L^2(V):\; \partial^\alpha f\in L^2(V)\;\fa \alp\leq k\right\}
	\end{equation*}
then we cannot conclude in general that $W^k(V)=H^k(V)$ for $k\in\N$, unless 
$V$ is a relatively compact set in $\R^n$ with smooth boundary.
\end{Rem}
We denote the (quotient) norm of $H^\sigma(V)$ by $\norm[H^\sigma(V)]{\:.\:}$.
For $f\in H^\infty_{loc}$ this agrees with the previous definition of 
$\norm[H^\sigma(V)]{f}$ in \eqref{SobolevNorm}.
More precisely, if $U$ is an open set in $\R^n$, $V\Subset U$ and $\iota$ is 
the natural inclusion map of $H^\infty_{loc}(U)$ into $H^\sigma(V)$ then
$\norm[H^\sigma(V)]{f}=\norm[H^\sigma(V)]{\iota(f)}$ for all
$f\in H^\infty_{loc}(U)$.

Now suppose that $U$ and $V$ are given open sets in $\R^n$ such that $V\Subset U\subseteq\R^n$  and $\sP=\{P_1,\dotsc,P_\ell\}$
is a family of analytic partial differential operators of orders $d_j$ on $U$
satisfying 
 \begin{equation}\label{SubellipticVariant}
 	\norm[\sigma+\eps]{\varphi}\leq 
 	C\left[\sum_{j=1}^\ell \norm[\sigma]{P_j\varphi}+\norm[\sigma]{\varphi}\right]
 \end{equation}
for all $\varphi\in\D(V)$ and some constant $C>0$ independent of $\varphi$.
If we multiply all of the coefficients of the operator $P_j$ with a  test function
$\chi\in\D(U)$ satisfying $\chi\vert_V=1$ we may assume that the operator $P_j$
is a continuous mapping from the space $H^\sigma(\R^n)$ 
into $H^{\sigma-d_j}(\R^n)$ for all $\sigma$. This clearly does not
change the value of $\norm[\sigma]{P_j\varphi}$ when $\varphi\in\D(V)$
or of $\norm[H^\sigma(V)]{P_jg}$ when $g\in H^\sigma_{loc}(U)$.
Therefore the mapping $P_j:H^\infty(\R^n)\rightarrow H^\infty(\R^n)$, where
$H^\infty(\R^n)=\proj_{\sigma} H^\sigma(\R^n)$, is also continuous.
Moreover, observe that
 $F^\infty(V)=\bigcap_{\sigma} F^\sigma(V)$ is closed in $H^\infty(\R^n)$.
 Similarly, $H^\infty(V)=\proj_{\sigma} H^\sigma(V)$ is a Fr\'{e}chet space and
 $P_j$ is a continuous automorphism on $H^\infty(V)$ since $P_j$ is a local operator.
 
We recall that we want to show that \eqref{SubellipticVariant} implies
\begin{equation}\label{SubellipticVariant1}
	\norm[H^{\sigma+\varepsilon}(V)]{g}\leq C\left[\sum_{j=1}^\ell 
	\norm[H^\sigma(V)]{P_jg}+\norm[H^\sigma(V)]{g}\right]
\end{equation}
for all $g\in \E(U)$.
 
For $V\Subset U$ given we denote by $\iota_V: H^\infty_{loc}(U)\rightarrow H^\infty(V)$ the canonical inclusion mapping.
Due to the continuity of the operators $P_j$ we would be done 
if we could show that $\iota_V(H^\infty_{loc}(U))\subseteq\hat{\D}(V)$, where
$\hat{\D}(V)$ is the closure of $\D(V)$ in the topology of $H^\infty(V)$.
If $V=B$ we have the following result:
\begin{Thm}\label{BallTheorem}
	Let $U\subseteq\R^n$ be an open set and $B$ be an open ball such that $B\Subset U$.
	 Then	we have $\iota_B(H^\infty_{loc}(U))\subseteq\hat{\D}(B)$.
\end{Thm} 
\begin{proof}
	For each $g\in H^\infty_{loc}(U)=\E(U)$ we have to find a 
	sequence $\varphi_j\in\D(B)$ such that 
	$\dot{\varphi}_j=\varphi_j+F^\infty(B)$ converges to $\dot{g}=\iota_B(g)$
	in $H^\infty(B)$.
	A representative of $\dot{g}$ is given by $\chi g$ where
	$\chi\in\D(U)$ with $\chi\vert_B=1$.
	We choose two sequences $(K_j)_j$, $(L_j)_j$ of compact subsets of $U$ with the
	following properties:
	\begin{itemize}
		\item $K_j\subseteq B$ and $\dist(K_j,U\setminus B)\rightarrow 0$ when $j\rightarrow \infty$.
		\item $\overline{B}\subseteq L_j^\circ$ and $\dist(\overline{B},\partial L_j)
		\rightarrow 0$ if $j\rightarrow \infty$.
	\end{itemize}
	For each $j\in\mathbb{N}$ choose test functions $\psi_j\in\D(B)$ and $\lambda_j\in\D(U)$
	such that $0\leq\psi_j,\lambda_j\leq 1$, $\psi_j\vert_{K_j}=1$, $\lambda_j\vert_{\supp\chi\setminus L_j}=1$ and
	$\supp \lambda_j\subseteq U\setminus \overline{B}$.
	We set $\varphi_j=\psi_j g\in\D(B)$ and $h_j=\lambda_j\chi g\in\D (U)$.
	Then $\supp(\chi g-\varphi_j-h_j)\subseteq L_j\setminus K_j$ and
	\begin{equation*}
		\left\lvert\,\int_{L_j\setminus K_j} (\chi(x) g(x)-\varphi_j(x)-h_j(x))\Phi(x)\,dx\,
		\right\rvert\leq
		\sup \Betr{(\chi g-\varphi_j-h_j)\Phi}\cdot \underbrace{\Betr{L_j\setminus K_j}}_{\rightarrow \Betr{\partial B}=0}
	\end{equation*}
	for $\Phi\in\E(U)$, i.e.\ $\chi g-\varphi_j-h_j\rightarrow 0$ in $\temp$ and thus
	$\mathcal{F}(\chi g-\varphi_j-h_j)\rightarrow 0$ in $\temp$ since 
	$\mathcal{F}:\temp\rightarrow\temp$ 
	is continuous.
	But this means that $\mathcal{F}(\chi g-\varphi_j-h_j)(\xi)\rightarrow 0$
	almost everywhere. Note also that since $\chi g-\varphi_j-h_j\in \D(\supp\chi\!\setminus\! K_j)$ and
	\begin{equation*}
		\Betr{\chi g-\varphi_j-h_j}\leq\sup\Betr{\chi g}
	\end{equation*}
	we have by the Paley-Wiener Theorem, see \cite[page 181]{MR1996773},
	that for each $N\in\N$ there is a constant $C_N$ depending on 
	$N$, $g$ and $\chi$ such that
	\begin{equation*}
		\Betr{\mathcal{F}(\chi g-\varphi_j-h_j)(\xi)}\leq C_N(1+\xit)^{-N}
	\end{equation*}
	for all $\xi\in\R^n$.
	Hence the dominated convergence theorem implies that
	\begin{equation*}
		\norm[\sigma]{\chi g-\varphi_j-h_j}^2=\int \bigl(1+\xit^2\bigr)^{\sigma}
		\Betr{\mathcal{F}(\chi g-\varphi_j-h_j)(\xi)}^2\,d\xi
		\longrightarrow 0
	\end{equation*}
	for all $\sigma\in\R$.
	It follows that $\dot{\varphi}_j\rightarrow \dot{g}$ in $H^\infty(B)$.
\end{proof}
For the proof of \eqref{SubellipticVariant1} we also need the following
fact.
\begin{Prop}\label{Hilbert}
	Let $E$ be a Hilbert space and $\{M_j:\;j\in I\}$ a family of closed
	subspaces of $E$. If $M=\bigcap_{j\in I} M_j$ and if
	$\norm[M]{\,.\,}$ and $\norm[M_j]{\,.\,}$ are the quotient norms of 
	$E/M$ and $E/M_j$, respectively, then 
	\begin{equation*}
		\norm[M]{x}=\sup_{j\in I}\norm[M_j]{f_j(x)},\qquad x\in E,
	\end{equation*} 
	where the $f_j$ are the induced canonical projections $E/M\rightarrow E/M_j$
	given by $x+M\mapsto x+M_j$. 
\end{Prop}
\begin{proof}
	Let denote the inner product on $E$ by $\langle\,.\,,\,.\,\rangle_E$
	and
	\begin{align*}
		M_j^\bot&=\{x\in E\mid \langle x,y\rangle_E=0\;\,\forall y\in M_j \}\\
		M^\bot&=\{x\in E\mid \langle x,y\rangle_E=0\;\,\forall y\in M \}
	\end{align*}
	be the orthogonal complements of $M_j$ and $M$, respectively.
	It is well-known that $M_j^\bot$ and $M_j$ are Hilbert spaces as closed subspaces
	of $E$.
	Using the canonical Hilbert space isomorphisms $M_j^\bot\cong E/M_j$,  and $M^\bot\cong E/M$ we can identify $f_j$ with the canonical projection
	$M^\bot\rightarrow M_j^\bot$.
	Since $M=\bigcap_{j\in I}M_j$ we have that $\bigcap \ker f_j=\{0\}$.
	It is easy to see that the topology on $M^\bot$ is equivalent to
	 the initial topology with respect
	to the mappings $f_j$. Indeed, the closed subsets of $M^\bot$ and $M^\bot_j$
	are exactly the sets of the form $V=A\cap M^\bot$ and $V_j=A\cap M^\bot_j$, respectively, where $A\subseteq E$ is closed.
	Clearly, the canonical topology on $M^\bot$ is finer than
	the initial topology induced by the $f_j$'s which is generated by 
	\begin{equation*}
		f^{-1}_j(V_j)=A\cap M^\bot+\ker f_j.
	\end{equation*}
	It follows that
	\begin{equation*}
		\bigcap_{j\in I}f^{-1}(V_j)=A\cap M^\bot+\bigcap_{j\in I}\ker f_j=A\cap M^\bot.
	\end{equation*}
	
	Now set $\varphi_j(x)=\norm[M_j]{x}$, $x\in E/M\cong M^\bot$.
	Obviously $\varphi_j$ is a seminorm on $E/M\cong M^\bot$.
	The same is true for 
	\begin{equation*}
		\Phi(x)=\sup_{j\in I}\varphi_j(x),\qquad x\in E/M.
	\end{equation*}
	In fact $\Phi$ is a norm. Suppose $\Phi(x)=0$ for some $x\in E/M$, thus
	$\varphi_j(x)=0$ for all $j\in I$.
	Hence $f_j(x)=0$ for all $j\in I$, since $\norm[M_j]{\,.\,}$ is a norm.
	We conclude that
	\begin{equation*}
		x\in\bigcap_{j\in I} \ker f_j=\{0\}.
	\end{equation*}
	
	If $B$ is the closed unit ball in $E$ then $B_0=B\cap M^\bot$ and 
	$B_j=B\cap M^\bot_j$ 
	are the unit balls in $M^\bot$ and $M^\bot_j$, respectively.
	By the above we know that 
	\begin{equation*}
		\bigcap f^{-1}_{j\in I}(B_j)=B_0
	\end{equation*}
	and furthermore
	\begin{equation*}
		B_\Phi=\{x\in M^\bot\mid \Phi(x)\leq 1\}=\{x\in M^\bot_j\mid \norm[M_j]{x}\leq 1,\;\,\forall j\in I\}=\bigcap_{j\in I}f^{-1}_j(B_j)
	\end{equation*}
	is the closed unit ball for the norm $\Phi$.
	Since both norms have the same closed unit ball, they have to agree everywhere.
\end{proof}

Now there are at most countable many open balls $B_j$, $j\in J$, such that
$V=\bigcup_{j\in J}B_j$. 
 In particular, $F^\sigma(V)=\bigcap_{j\in J}F^\sigma(B_j)$ for all $\sigma\in\R$.
 Hence Proposition \ref{Hilbert} implies that
 \begin{equation}\label{Norms}
 	\norm[H^\sigma(V)]{\,.\,}=\sup_{j\in I}\norm[H^\sigma(B_j)]{\,.\,}.
 \end{equation} 

As indicated above, Theorem \ref{BallTheorem} shows \eqref{SubellipticVariant1}
if $V$ is a ball. The general case follows from \ref{Norms}.

\bibliographystyle{abbrv}
\bibliography{vectors.bib}

\begin{thebibliography}{10}

\bibitem{MR2595651}
A.~A. Albanese, D.~Jornet, and A.~Oliaro.
\newblock Quasianalytic wave front sets for solutions of linear partial
  differential operators.
\newblock {\em Integral Equations Operator Theory}, 66(2):153--181, 2010.

\bibitem{MR3999031}
V.~Asensio and D.~Jornet.
\newblock Global pseudodifferential operators of infinite order in classes of
  ultradifferentiable functions.
\newblock {\em Rev. R. Acad. Cienc. Exactas F\'{\i}s. Nat. Ser. A Mat. RACSAM},
  113(4):3477--3512, 2019.

\bibitem{MR654409}
M.~S. Baouendi and G.~Métivier.
\newblock Analytic vectors of hypoelliptic operators of principal type.
\newblock {\em American Journal of Mathematics}, 104(2):287--319, 1982.

\bibitem{Beurling61}
A.~Beurling.
\newblock Quasi-analyticity and general distributions.
\newblock Lecture 4 and 5, AMS Summer Institute, Stanford, 1961.

\bibitem{Bjorck66}
G.~Bj\"{o}rck.
\newblock Linear partial differential operators and generalized distributions.
\newblock {\em Ark. Mat.}, 6:351--407, 1966.

\bibitem{MR3595351}
C.~Boiti, R.~Chaili, and T.~Mahrouz.
\newblock Iterates of systems of operators in spaces of
  {$\omega$}-ultradifferentiable functions.
\newblock {\em Ann. Polon. Math.}, 118(2-3):95--111, 2016.

\bibitem{MR3380075}
C.~Boiti and D.~Jornet.
\newblock The problem of iterates in some classes of ultradifferentiable
  functions.
\newblock In {\em Pseudo-differential operators and generalized functions},
  volume 245 of {\em Oper. Theory Adv. Appl.}, pages 21--33.
  Birkh\"{a}user/Springer, Cham, 2015.

\bibitem{MR3661157}
C.~Boiti and D.~Jornet.
\newblock A characterization of the wave front set defined by the iterates of
  an operator with constant coefficients.
\newblock {\em Rev. R. Acad. Cienc. Exactas, F\'{\i}s. Nat. Ser. A. Math.
  RACSAM}, 111(3):891--919, 2017.

\bibitem{MR3652556}
C.~Boiti and D.~Jornet.
\newblock A simple proof of {K}otake-{N}arasimhan theorem in some classes of
  ultradifferentiable functions.
\newblock {\em J. Pseudo-Diff. Oper. Appl.}, 8:297--317, 2017.

\bibitem{MR3208537}
C.~Boiti, D.~Jornet, and J.~Juan-Huguet.
\newblock Wave front sets with respect to the iterates of an operator with
  constant coefficients.
\newblock {\em Abstr. Appl. Anal.}, 2014.

\bibitem{MR548225}
P.~Bolley and J.~Camus.
\newblock Powers and {G}evrey's regularity for a system of differential
  operators.
\newblock {\em Czechoslovak Math. J.}, 29(4):649--661, 1979.

\bibitem{MR632764}
P.~Bolley and J.~Camus.
\newblock Regularité {G}evrey et itérés pour une classe d'opérateurs
  hypoelliptiques.
\newblock {\em Comm. Part. Diff. Eq.}, 6(10):1057--1110, 1981.
\newblock 

\bibitem{MR557524}
P.~Bolley, J.~Camus, and C.~Mattera.
\newblock Analyticité microlocale et itérés d'opérateurs.
\newblock In {\em Séminaire {G}oulaouic-{S}chwartz (1978/1979)}, pages Exp.
  No. 13, 9. École Polytech., Palaiseau, 1979.

\bibitem{MR1037999}
P.~Bolley, J.~Camus, and L.~Rodino.
\newblock Hypoellipticité analytique-{G}evrey et itérés d'opérateurs.
\newblock {\em Università e Politecnico di Torino. Seminario Matematico.
  Rendiconti}, 45(3):1--61, 1987.

\bibitem{MR220049}
J.~Boman.
\newblock On the intersection of classes of infinitely differentiable
  functions.
\newblock {\em Ark. Mat.}, 5:301--309 (1963/65), 1963/65.

\bibitem{BonetMeiseMelikhov07}
J.~Bonet, R.~Meise, and S.~N. Melikhov.
\newblock {A comparison of two different ways to define classes of
  ultradifferentiable functions}.
\newblock {\em Bulletin of the Belgian Mathematical Society - Simon Stevin},
  14(3):425--444, 2007.

\bibitem{MR1052587}
R.~W. Braun, R.~Meise, and B.~A. Taylor.
\newblock Ultradifferentiable functions and {F}ourier analysis.
\newblock {\em Results Math.}, 17(3-4):206--237, 1990.

\bibitem{MR3556261}
N.~Braun~Rodrigues, G.~Chinni, P.~D. Cordaro, and M.~R. Jahnke.
\newblock Lower order perturbation and global analytic vectors for a class of
  globally analytic hypoelliptic operators.
\newblock {\em Proc. Amer. Math. Soc.}, 144(12):5159--5170, 2016.

\bibitem{MR2801277}
D.~Calvo and L.~Rodino.
\newblock Iterates for operators and {G}elfand-{S}hilov classes.
\newblock {\em Integr. Transf. Spec. F.}, 22:269--276, 2011.

\bibitem{MR3043156}
J.~E. Castellanos, P.~D. Cordaro, and G.~Petronilho.
\newblock Gevrey vectors in involutive tube structures and {G}evrey regularity
  for the solutions of certain classes of semilinear systems.
\newblock {\em J. Anal. Math.}, 119:333--364, 2013.

\bibitem{ASENS_1980_4_13_4_397_0}
M.~Damlakhi and B.~Helffer.
\newblock Analyticité et itères d'un système de champs non elliptique.
\newblock {\em Annales scientifiques de l'École Normale Supérieure}, 4e
  série, 13(4):397–403, 1980.

\bibitem{SEDP_1970-1971____A12_0}
M.~Derridj.
\newblock Sur une classe d'opérateurs différentiels hypoelliptiques à
  coefficients analytiques.
\newblock {\em Séminaire Équations aux dérivées partielles
  (Polytechnique)}, 1970-1971.
\newblock talk:12.

\bibitem{Derridj2017}
M.~Derridj.
\newblock On {G}evrey vectors of some partial differential operators.
\newblock {\em Complex Var. Elliptic Eq.}, 62(10):1474--1491, 2017.
\newblock 

\bibitem{MR4036740}
M.~Derridj.
\newblock Local estimates for {H}\"{o}rmander's operators of first kind with
  analytic {G}evrey coefficients and application to the regularity of their
  {G}evrey vectors.
\newblock {\em Pacific J. Math.}, 302(2):511--543, 2019.

\bibitem{derridj2019}
M.~Derridj.
\newblock Local estimates for {H}örmander's operators with {G}evrey
  coefficients and application to the regularity of their {G}evrey vectors.
\newblock {\em Tunisian J. Math.}, 1(3):321--345, 2019.

\bibitem{Derridj2019a}
M.~Derridj.
\newblock On {G}evrey vectors of {L}. {H}\"{o}rmander's operators.
\newblock {\em Trans. Amer. Math. Soc.}, 372(6):3845--3865, 2019.

\bibitem{MR4098643}
M.~Derridj.
\newblock Gevrey regularity of {G}evrey vectors of second-order partial
  differential operators with non-negative characteristic form.
\newblock {\em Complex Anal. Synerg.}, 6(2):Paper No. 10, 16, 2020.

\bibitem{MR4149078}
S.~F\"{u}rd\"{o}s.
\newblock Geometric microlocal analysis in {D}enjoy-{C}arleman classes.
\newblock {\em Pacific J. Math.}, 307(2):303--351, 2020.

\bibitem{FURDOS2020123451}
S.~F{\"u}rd{\"o}s, D.~N. Nenning, A.~Rainer, and G.~Schindl.
\newblock Almost analytic extensions of ultradifferentiable functions with
  applications to microlocal analysis.
\newblock {\em Journal of Mathematical Analysis and Applications},
  481(1):123451, 2020.

\bibitem{MR0230128}
I.~M. Gel'fand and G.~E. Shilov.
\newblock {\em Generalized functions. {V}ol. 2. {S}paces of fundamental and
  generalized functions}.
\newblock Translated from the Russian by Morris D. Friedman, Amiel Feinstein
  and Christian P. Peltzer. Academic Press, New York-London, 1968.

\bibitem{doi:10.1080/03605308008820164}
B.~Helffer and C.~Mattera.
\newblock Analyticite et iteres reduits d'un systeme de champs de vecteurs.
\newblock {\em Comm. Part. Diff. Equ.}, 5:1065--1072, 1980.
\newblock 

\bibitem{MR3936096}
G.~Hoepfner and A.~Raich.
\newblock Microglobal regularity and the global wavefront set.
\newblock {\em Math. Z.}, 291(3-4):971--998, 2019.

\bibitem{HoepfnerRampazo}
G.~Hoepfner and P.~Rampazo.
\newblock {T}he global {K}otake-{N}arasimhan theorem.
\newblock {\em {P}roc. {A}mer. {M}ath. {Soc.}}, 150:1041--1057, 2022.

\bibitem{Hoermander1967}
L.~{H}\"{o}rmander.
\newblock Hypoelliptic second order differential equations.
\newblock {\em Acta Mathematica}, 119:147–171, 1967.

\bibitem{MR0294849}
L.~H\"{o}rmander.
\newblock Uniqueness theorems and wave front sets for solutions of linear
  differential equations with analytic coefficients.
\newblock {\em Comm. Pure Appl. Math.}, 24:671–704, 1971.

\bibitem{MR1996773}
L.~H\"{o}rmander.
\newblock {\em The analysis of linear partial differential operators. {I}}.
\newblock Classics in Mathematics. Springer-Verlag, Berlin, 2003.
\newblock Distribution theory and Fourier analysis, Reprint of the second
  (1990) edition.

\bibitem{MR2304165}
L.~H\"{o}rmander.
\newblock {\em The analysis of linear partial differential operators. {III}}.
\newblock Classics in Mathematics. Springer, Berlin, 2007.
\newblock Pseudo-differential operators, Reprint of the 1994 edition.

\bibitem{MR2512677}
L.~H\"{o}rmander.
\newblock {\em The analysis of linear partial differential operators. {IV}}.
\newblock Classics in Mathematics. Springer-Verlag, Berlin, 2009.
\newblock Fourier integral operators, Reprint of the 1994 edition.

\bibitem{sectorextensions}
J.~Jim{\'e}nez-Garrido, J.~Sanz, and G.~Schindl.
\newblock Sectorial extensions, via {L}aplace transforms, in ultraholomorphic
  classes defined by weight functions.
\newblock {\em Results Math.}, 74(27), 2019.

\bibitem{sectorialextensions1}
J.~Jim{\'e}nez-Garrido, J.~Sanz, and G.~Schindl.
\newblock Sectorial extensions for ultraholomorphic classes defined by weight
  functions.
\newblock {\em {M}ath.\ {N}achr.}, 293(11):2140--2174, 2020.

\bibitem{SEDP_2005-2006____A14_0}
J.-L. Journ\'e and J.-M. Tr\'epreau.
\newblock Hypoellipticit\'e sans sous-ellipticit\'e~: le cas des syst\`emes de
  $n$ champs de vecteurs complexes en $(n+1)$ variables.
\newblock {\em S\'eminaire \'Equations aux d\'eriv\'ees partielles
  (Polytechnique)}, 2005.

\bibitem{MR2721087}
J.~Juan-Huguet.
\newblock Iterates and hypoellipticity of partial differential operators on
  non-quasianalytic classes.
\newblock {\em Integr. Equ. Oper. Th.}, 68:263--286, 2010.

\bibitem{komatsu1962}
H.~Komatsu.
\newblock A proof of {K}otak\'{e} and {N}arasimhan's theorem.
\newblock {\em Proc. Japan Acad.}, 38(9):615--618, 1962.

\bibitem{Komatsu73}
H.~Komatsu.
\newblock Ultradistributions. {I}. {S}tructure theorems and a characterization.
\newblock {\em J. Fac. Sci. Univ. Tokyo Sect. IA Math.}, 20:25--105, 1973.

\bibitem{MR550685}
H.~Komatsu.
\newblock An analogue of the {C}auchy-{K}owalevsky theorem for
  ultradifferentiable functions and a division theorem for ultradistributions
  as its dual.
\newblock {\em J. Fac. Sci. Univ. Tokyo Sect. IA Math.}, 26(2):239--254, 1979.

\bibitem{KotakeNarasimhan}
T.~Kotake and M.~Narasimhan.
\newblock Regularity theorems for fractional powers of a linear elliptic
  operator.
\newblock {\em Bulletin de la Soci{\'e}t{\'e} Math{\'e}matique de France},
  90:449--471, 1962.

\bibitem{MR2822315}
A.~Kriegl, P.~W. Michor, and A.~Rainer.
\newblock The convenient setting for quasianalytic {D}enjoy-{C}arleman
  differentiable mappings.
\newblock {\em J. Funct. Anal.}, 261(7):1799--1834, 2011.

\bibitem{MR0350177}
J.-L. Lions and E.~Magenes.
\newblock {\em Non-homogeneous boundary value problems and applications. {V}ol.
  {I}}.
\newblock Springer-Verlag, New York-Heidelberg, 1972.
\newblock Translated from the French by P. Kenneth, Die Grundlehren der
  mathematischen Wissenschaften, Band 181.

\bibitem{MR0051893}
S.~Mandelbrojt.
\newblock {\em S{\'e}ries adh{\'e}rentes, r{\'e}gularisation des suites,
  applications}.
\newblock Gauthier-Villars, Paris, 1952.

\bibitem{MeiseTaylor88}
R.~Meise and B.~A. Taylor.
\newblock Whitney's extension theorem for ultradifferentiable functions of
  {B}eurling type.
\newblock {\em Ark. Mat.}, 26(2):265--287, 1988.

\bibitem{MR1511564}
H.~Mellin.
\newblock Abri{\ss} einer einheitlichen {T}heorie der {G}amma- und der
  hypergeometrischen {F}unktionen.
\newblock {\em Math. Ann.}, 68(3):305--337, 1910.

\bibitem{doi:10.1080/03605307808820078}
G.~Métivier.
\newblock Propriete des iteres et ellipticite.
\newblock {\em Communications in Partial Differential Equations},
  3(9):827–876, 1978.

\bibitem{MR107176}
E.~Nelson.
\newblock Analytic vectors.
\newblock {\em Ann. of Math. (2)}, 70:572--615, 1959.

\bibitem{MR0318660}
E.~Newberger and Z.~Ziele\'{z}ny.
\newblock The growth of hypoelliptic polynomials and {G}evrey classes.
\newblock {\em Proc. Amer. Math. Soc.}, 39:547--552, 1973.

\bibitem{MR3285413}
A.~Rainer and G.~Schindl.
\newblock Composition in ultradifferentiable classes.
\newblock {\em Studia Math.}, 224(2):97–131, 2014.

\bibitem{MR3462072}
A.~Rainer and G.~Schindl.
\newblock Equivalence of stability properties for ultradifferentiable function
  classes.
\newblock {\em Rev. R. Acad. Cienc. Exactas Fí s. Nat. Ser. A Math. RACSAM},
  110(1):17–32, 2016.

\bibitem{MR3711790}
A.~Rainer and G.~Schindl.
\newblock Extension of {W}hitney jets of controlled growth.
\newblock {\em Mathematische Nachrichten}, 290(14-15):2356–2374, 2017.

\bibitem{MR3722569}
A.~Rainer and G.~Schindl.
\newblock On the {B}orel mapping in the quasianalytic setting.
\newblock {\em Mathematica Scandinavica}, 121(2):293–310, 2017.

\bibitem{MR3865684}
A.~Rainer and G.~Schindl.
\newblock On the extension of {W}hitney ultrajets.
\newblock {\em Studia Mathematica}, 245(3):255–287, 2019.

\bibitem{SchindlThesis}
G.~Schindl.
\newblock {\em Exponential laws for classes of {Denjoy-Carleman-differentiable}
  mappings}.
\newblock Phd-thesis, University of Vienna, 2014.
\newblock Available at
  \url{http://othes.univie.ac.at/32755/1/2014-01-26_0304518.pdf}.

\bibitem{MR3601829}
G.~Schindl.
\newblock Characterization of ultradifferentiable test functions defined by
  weight matrices in terms of their {F}ourier transform.
\newblock {\em Note Mat.}, 36(2):1--35, 2016.

\bibitem{Schindl21}
G.~Schindl.
\newblock On subadditivity-like conditions for associated weight functions.
\newblock {\em Bull.~Belg.~Soc.-Simon Stevin}, 28(3):399--427, 2022.

\bibitem{MR3857012}
D.~S. Tartakoff.
\newblock On local {G}evrey regularity for {G}evrey vectors of subelliptic sums
  of squares: an elementary proof of a sharp {G}evrey {K}otake-{N}arasimhan
  theorem.
\newblock {\em Ann. Univ. Ferrara Sez. VII Sci. Mat.}, 64(2):437--447, 2018.

\bibitem{MR2384272}
V.~Thilliez.
\newblock On quasianalytic local rings.
\newblock {\em Expo. Math.}, 26(1):1–23, 2008.

\bibitem{zbMATH03028224}
E.~C. {Titchmarsh}.
\newblock {\em {Introduction to the theory of Fourier integrals}}.
\newblock Oxford University Press, 1948.

\bibitem{MR132896}
F.~Tr\`eves.
\newblock An invariant criterion of hypoellipticity.
\newblock {\em Amer. J. Math.}, 83:645--668, 1961.

\bibitem{MR290201}
F.~Tr\`eves.
\newblock A new method of proof of the subelliptic estimates.
\newblock {\em Comm. Pure Appl. Math.}, 24:71--115, 1971.

\bibitem{MR0296509}
F.~Trèves.
\newblock Analytic-hypoelliptic partial differential equations of principal
  type.
\newblock {\em Communications on Pure and Applied Mathematics}, 24:537–570,
  1971.

\end{thebibliography}
\end{document}